\documentclass[12pt]{article}
\usepackage{graphicx}
\usepackage{amsmath}
\usepackage{amsfonts}
\usepackage{enumerate}
\usepackage{latexsym}
\usepackage{epsfig}
\usepackage{float}
\usepackage{color}
\usepackage{amscd}
\usepackage{amssymb}
\usepackage{epstopdf}

\newtheorem{theorem}{Theorem}[section]
\newtheorem{lemma}[theorem]{Lemma}

\def\proofbox{\begin{picture}(6.5,6.5)
\put(0,0){\framebox(6.5,6.5){}}\end{picture}}
\newenvironment{proof}{\noindent{\it Proof.\quad}}{\hfill\proofbox}

\begin{document}

\title{Exhausting Curve Complexes by Finite Superrigid Sets on Nonorientable Surfaces\\}
\author{Elmas Irmak}

\maketitle

\renewcommand{\sectionmark}[1]{\markright{\thesection. #1}}

\thispagestyle{empty}
\maketitle
\begin{abstract} Let $N$ be a compact, connected, nonorientable surface of genus $g$ with $n$ boundary components. Let $\mathcal{C}(N)$ be the curve complex of $N$. 
We prove that if $(g, n) \neq (1,2)$ and $g + n \neq 4$, then there is an exhaustion of $\mathcal{C}(N)$ by a sequence of finite superrigid sets.\end{abstract}
  
{\small Key words: Mapping class group, curve complex, nonorientable surface, rigidity

}
 
\section{Introduction} Let $N$ be a compact, connected, nonorientable surface of genus $g$ with $n$ boundary components. Mapping class group,
$Mod_N$, of $N$ is defined to be the group of isotopy classes of all self-homeomorphisms of $N$. The \textit{complex of curves},
$\mathcal{C}(N)$, on $N$ is defined as an abstract simplicial complex as follows: A simple closed curve on $N$ is called nontrivial
if it does not bound a disk, a M\"{o}bius band, and it is not isotopic to a boundary component of $N$. The vertex set of $\mathcal{C}(N)$
is the set of isotopy classes of nontrivial simple closed curves on $N$. A set of vertices forms a simplex in $\mathcal{C}(N)$ if
they can be represented by pairwise disjoint simple closed curves. 
Let $\mathcal{A}$ be a subcomplex in $\mathcal{C}(N)$. A simplicial map
$\lambda : \mathcal{A} \rightarrow \mathcal{C}(N)$ is called superinjective if the following condition holds:
if $\alpha, \beta$ are two vertices in $\mathcal{A}$ such that $i(\alpha,\beta) \neq 0$, then
$i(\lambda(\alpha),\lambda(\beta)) \neq 0$. We will say $\mathcal{A}$ is a  superrigid set if for every superinjective simplicial map $\lambda : \mathcal{A} \rightarrow \mathcal{C}(N)$, there exists a homeomorphism $h : N \rightarrow N$ such that $H(\alpha) = \lambda(\alpha)$
for every vertex $\alpha$ in $\mathcal{A}$ where $H=[h]$. The main result is the following:
 
\begin{theorem} \label{main} Let $N$ be a compact, connected, nonorientable surface of genus $g$ with $n$ boundary components. If $(g, n) \neq (1,2)$ and $g + n \neq 4$, then 
there exists a sequence $\mathcal{Z}_1 \subset \mathcal{Z}_2 \subset \cdots \subset \mathcal{Z}_n \subset \cdots$ such that $\mathcal{Z}_i$ is a finite superrigid set in $\mathcal{C}(N)$ and $\bigcup_{i \in \mathbb{N}} \mathcal{Z}_i = \mathcal{C}(N)$. If $g+n \geq 5$, then $\mathcal{Z}_i$ can be chosen to have trivial pointwise stabilizer in $Mod_N$ for each $i$.\end{theorem}
   
Many results about simplicial maps of curve complexes on orientable surfaces were proved after Ivanov's theorem: Ivanov proved that every automorphism of the complex of curves is induced by a homeomorphism of the surface if the genus is at least two in \cite{Iv1} for compact, connected, oreintable surfaces. He used this result to give classification of isomorphisms between any two finite index subgroups of the extended mapping class group of the surface in \cite{Iv1}. Ivanov's results for lower genus cases were proved by Korkmaz in \cite{K1} and indepedently by Luo in \cite{L}. The author proved that the superinjective simplicial maps of the complexes of curves on a compact, connected, orientable surface are induced by homeomorphisms if the genus is at least two, and using this result she gave a classification of injective homomorphisms from finite index subgroups to the extended 
mapping class group in \cite{Ir1}, \cite{Ir2}, \cite{Ir3}. These results 
were extended to lower genus cases by Behrstock-Margalit in \cite{BM} and Bell-Margalit in \cite{BeM}. Shackleton proved that locally injective simplicial maps of the curve complex are induced by homeomorphisms in \cite{Sh} for orientable surfaces. In \cite{H2} Hern\'andez proved that edge-preserving maps on orientable surfaces are induced by homeomorphisms. The author gave a new proof of this result (including also low genus cases) for edge preserving maps of the curve graphs, where she also proved the result for the nonseparating curve graphs, see \cite{Ir6} and \cite{Ir7}. A subcomplex $\mathcal{A}$ of the curve complex $\mathcal{C}(R)$ is called rigid if for every locally injective (injective on the star of every vertex) simplical map $\tau : \mathcal{A} \rightarrow \mathcal{C}(R)$, there exists a homeomorphism $h : R \rightarrow R$ such that $H(\alpha) = \tau(\alpha)$
for every vertex $\alpha$ in $\mathcal{A}$ where $H=[h]$. On orientable surfaces Aramayona-Leininger proved that there exists a finite rigid subcomplex in the curve complex if $R$ is of finite type in \cite{AL1}, and there is an exhaustion of the curve complex by a sequence of finite rigid sets in \cite{AL2}. 

On nonorientable surfaces there were similar results: Atalan-Korkmaz proved that the automorphism group of the curve complex is isomorphic to the mapping class group for most cases in \cite{AK}. The author proved that each superinjective simplicial map of the complex of curves is induced by a homeomorphism in most cases in \cite{Ir4}. She also proved that an injective simplicial map of the curve complex is induced by a homeomorphism in \cite{Ir5}. This result implies that superinjective simplicial maps and automorphisms are induced by homeomorphisms since they are injective.  
Irmak-Paris proved that superinjective simplicial maps of the two-sided curve complex are induced by homeomorphisms on compact, connected, nonorientable surfaces when the genus is at least 5 in \cite{IrP1}. They also gave a classification of injective homomorphisms from finite index subgroups to the whole mapping class group on these surfaces in \cite{IrP2}. Recently, in \cite{IlK} Ilbira-Korkmaz proved the existence of finite rigid subcomplexes in the curve complex for nonorientable surfaces when $g+n \neq 4$. In this paper we start with Ilbira-Korkmaz's rigid set and enlarge it in steps to get superrigid sets and prove that there exists an exhaustion of the curve complex by a sequence of finite superrigid sets on nonorientable surfaces when $g+n \neq 4$. We use some techniques given by Aramayona-Leininger in \cite{AL2} and Irmak-Paris in \cite{IrP1}.

\section{Some small genus cases}

\begin{figure}
	\begin{center}
		\hspace{0.006cm} \epsfxsize=1.05in \epsfbox{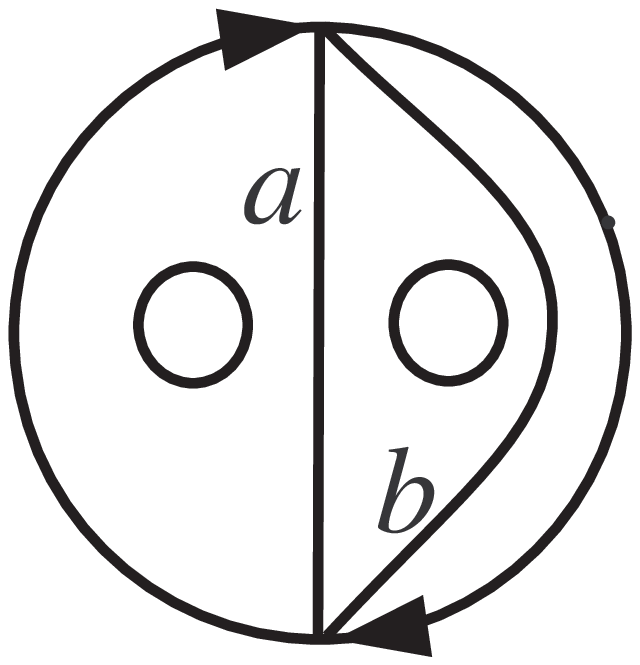} \hspace{1.25cm} \epsfxsize=2in \epsfbox{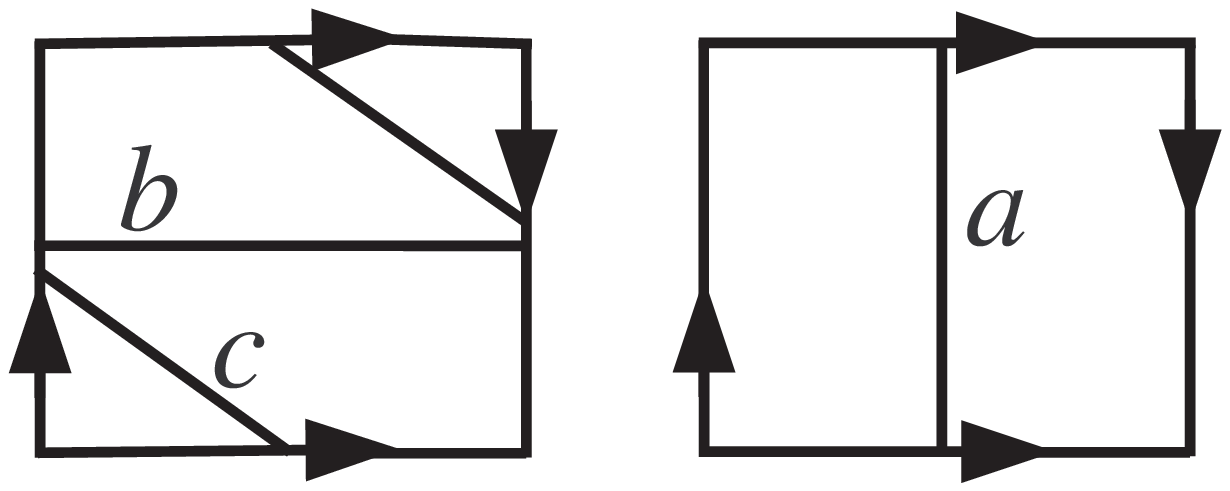}
		
\hspace{0.1cm} (i) \hspace{4.7cm} (ii) \hspace{1.1cm}
		
		\caption{Vertices of $\mathcal{C}(N)$ for (i) $(g, n)=(1,2)$, and (ii) $(g, n)=(2,0)$}
		\label{f1}
	\end{center}
\end{figure}   

\begin{figure}
	\begin{center}
 \hspace{-0.9cm} \epsfxsize=1.2in \epsfbox{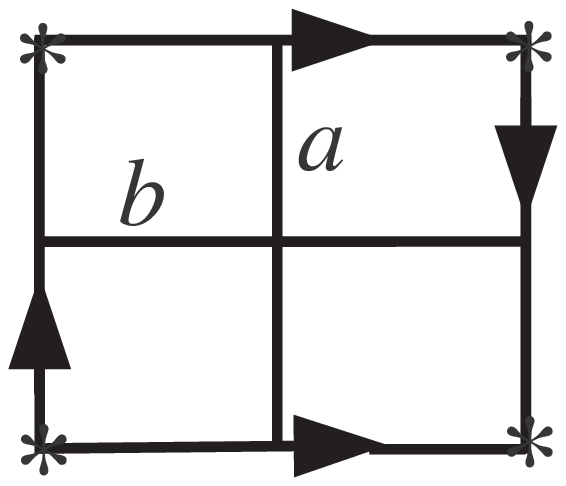} \hspace{-0.4cm}
		\epsfxsize=3.7in \epsfbox{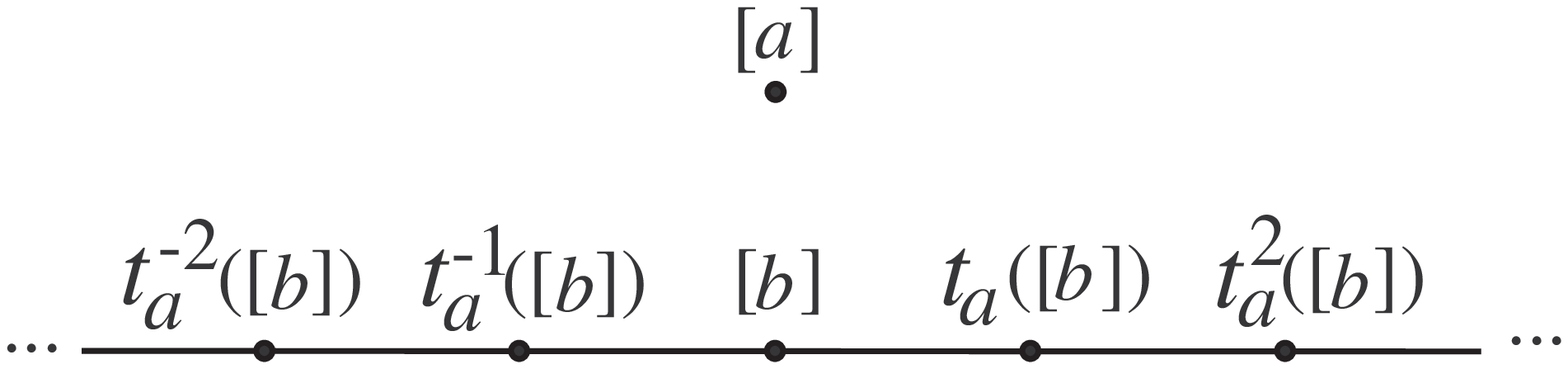} \hspace{-0.4cm}	\epsfxsize=1.2in \epsfbox{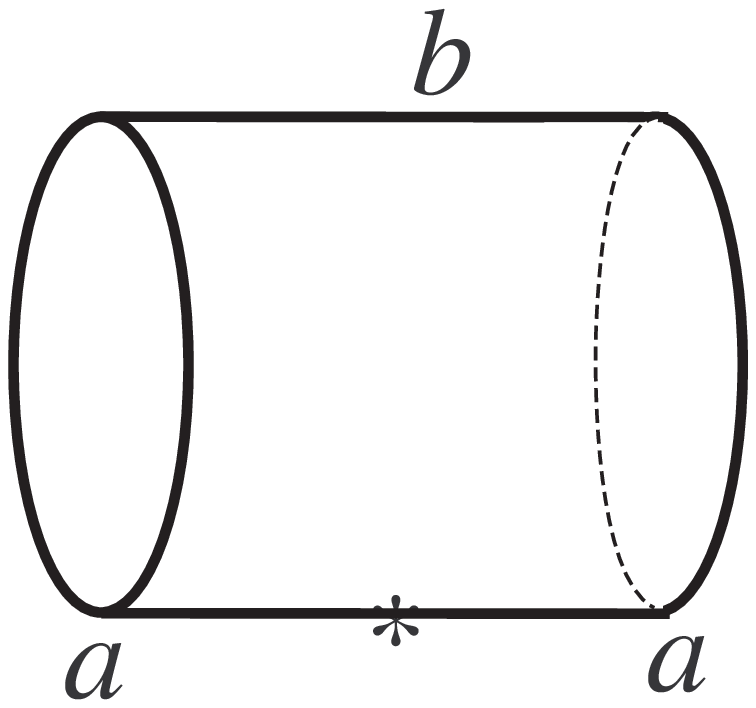}
		
\hspace{-0.55cm}	 (i) \hspace{5.6cm} (ii)   \hspace{5.6cm} (iii)  
		\caption{Curve complex for $(g, n)=(2,1)$}
		\label{f2}
	\end{center}
\end{figure}
 
A simple closed curve on $N$ is called 1-sided if a regular neighborhood of it is homeomorphic to a M\"obius band and it is called 2-sided if a regular neighborhood of it is homeomorphic to an annulus. If $x$ is a simple closed curve on $N$, then $i([x], [x])=0$ if and only if $x$ is a 2-sided curve, and 
$i([x], [x])=1$ if and only if $x$ is a 1-sided curve.
Since superinjective simplicial maps preserve geometric intersection zero and nonzero properties, they send 2-sided curves to 2-sided curves and 1-sided curves to 1-sided curves. 
If $(g,n) = (1,0)$, $N$ is the projective plane. There is only one element (isotopy class of a 1-sided curve) in the curve complex. If $(g,n) = (1,1)$, 
$N$ is Mobius band. There is only one element (isotopy class of a 1-sided curve) in the curve complex. So, if $(g,n) = (1,0)$ or $(g,n) = (1,1)$, we see easily that $\mathcal{C}(N)$ is a superrigid set. If $(g,n) = (1,2)$, there are only two vertices in the curve complex (see \cite{Sch}). They are the isotopy classes of two 1-sided curves $a$ and $b$ as shown in Figure \ref{f1} (i). We have  $i([a],[b]) = 1$. A simplicial map sending both of these vertices to $[a]$ is superinjective and it is not induced by a homeomorphism. So, $\mathcal{C}(N)$ is not a superrigid set. If $(g,n) = (1,2)$, it is easy to see that every injective simplicial map of the curve complex is induced by a homeomorphism.

If $(g,n) = (2,0)$, there are only three vertices and one edge in the curve complex (see \cite{Sch}). The vertices are the isotopy classes of $a$, $b$ and $c$ as shown in Figure \ref{f1} (ii). The curve $a$ is 2-sided, and $b$ and $c$ are both 1-sided. We have $i([a], [b]) = 1$, $i([a], [c]) = 1$ and $i([b], [c]) = 0$. Let $\lambda : \mathcal{C}(N) \rightarrow \mathcal{C}(N)$ be a superinjective simplicial map. Since superinjective simplicial maps preserve geometric intersection zero and nonzero properties, $\lambda$ fixes $[a]$. If it also fixes each of $[b]$ and $[c]$ then it is
induced by the identity homeomorphism, if it switches $[b]$ and $[c]$ then it is induced by a homeomorphism that switches the one sided curves $b$
and $c$, while fixing $a$ up to isotopy. So, $\mathcal{C}(N)$ is a finite superrigid set in this case.

If $(g, n) = (2, 1)$, then the curve complex is given by Scharlemann in \cite{Sch} as follows: Let $a$ and $b$ be as in
Figure \ref{f2} (i). We see that $i([a], [b]) = 1$. The vertex set of the curve complex is $\{[a], [b], t_a^m ([b]) : m \in \mathbb{Z} \}$. The complex is shown in Figure \ref{f2} (ii). The curve $a$ is a 2-sided curve, and $b$ is a 1-sided curve. Let $\mathcal{Z}_i= \{[a], [b], t_a^m ([b]) : -i \leq m \leq i \}$. Since superinjective simplicial maps send 2-sided curves to 2-sided curves and 1-sided curves to 1-sided curves, they fix $[a]$. By cutting $N$ along $a$ we get a cylinder with one puncture as shown in Figure \ref{f2} (iii). There is a reflection of the cylinder
interchanging the front face with the back face which fixes $b$ pointwise. This gives a homeomorphism $r$ of $N$ such that $r(b) = b$ and $r(a) = a^{-1}$. So, $(r)_\#([b]) = [b]$ and $(r)_\#(t_a([b])) = t_{a^{-1}}([b]) = {t_a}^{-1}([b])$. The map $r_\#$ reflects the graph in Figure \ref{f2} (ii) at $[b]$ and fixes $[a]$. By using that superinjective simplicial maps preserve geometric intersection zero and nonzero properties, it is easy to see that any superinjective simplicial map on $\mathcal{Z}_i$ is induced by $t_a^k$ or $t_a^k \circ r$ for some $k \in \mathbb{Z}$ for each $i$. So, each $\mathcal{Z}_i$ is a finite superrigid set and $\bigcup_{i \in \mathbb{N}} \mathcal{Z}_i = \mathcal{C}(N)$.
  
\begin{figure}[t]
	\begin{center}
		\epsfxsize=2.7in \epsfbox{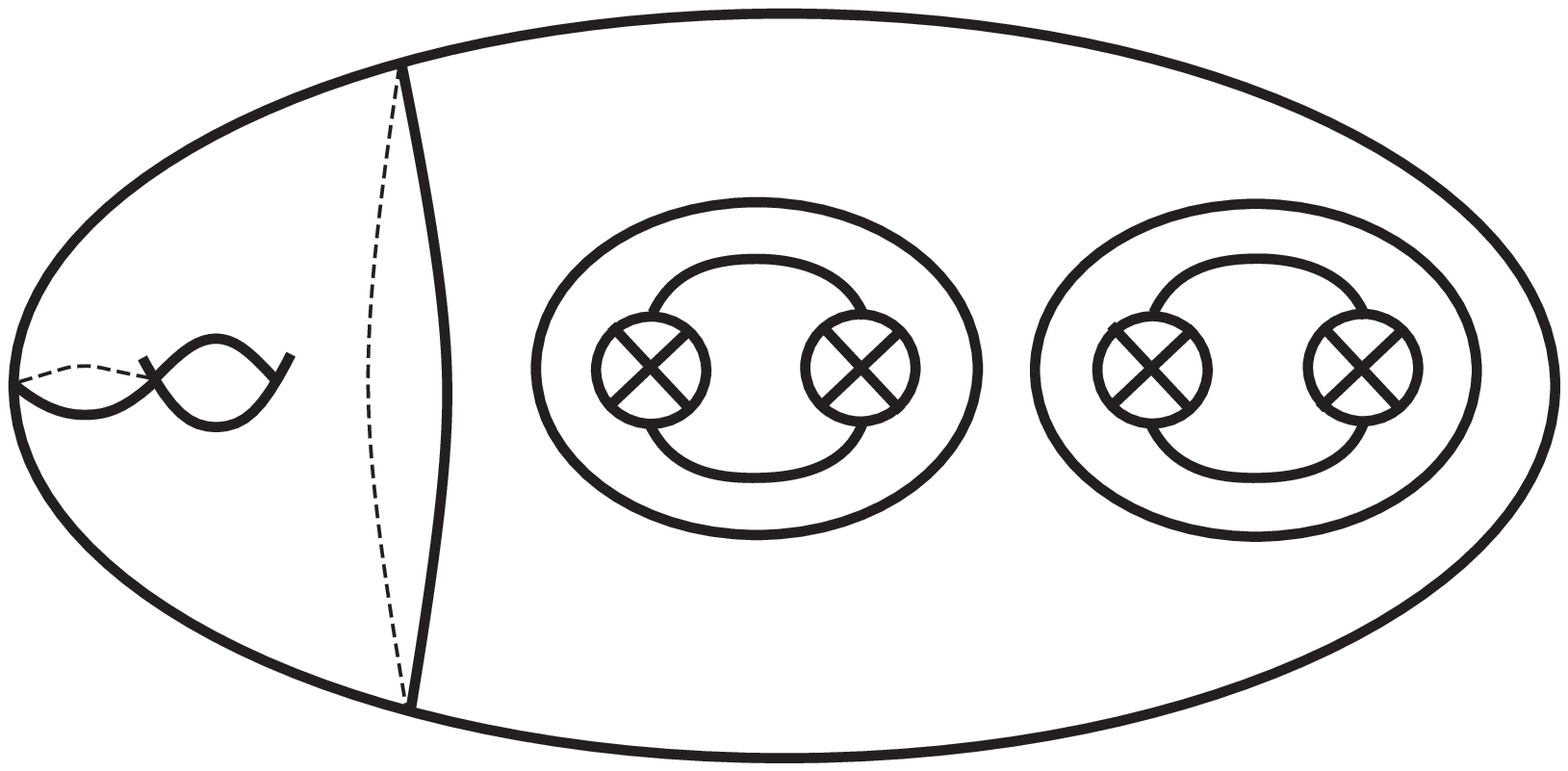} \hspace{0.1in}
		\epsfxsize=2.7in \epsfbox{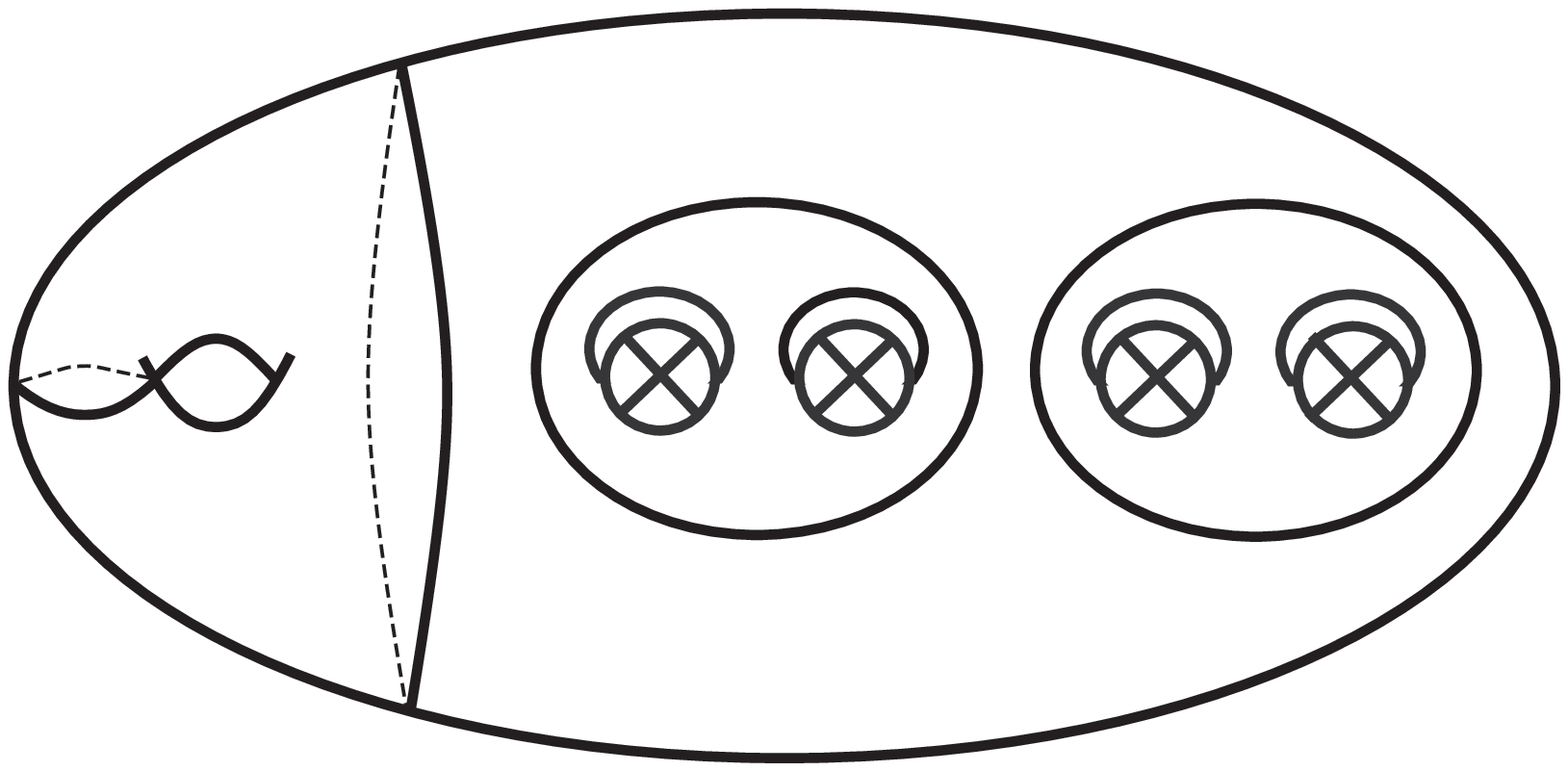}
		\caption{Different dimensional maximal simplices in $\mathcal{C}(N)$}
		\label{Fig0}
	\end{center}
\end{figure}
  
\section{Exhaustion of $\mathcal{C}(N)$ by finite superrigid sets when $g+n \geq 5$}

In this section we will always assume that $g+n \geq 5$. We first give some definitions. A set $P$ of pairwise disjoint, nonisotopic, nontrivial simple closed curves on $N$ is called a pants decomposition, if each
component of the surface $N_P$, obtained by cutting $N$ along $P$, is a pair of pants. Let $P$ be a pants decomposition of $N$. Let $[P]$ be the set of isotopy classes of elements of $P$. Note that $[P]$ is a maximal simplex of $\mathcal{C}(N)$. Every maximal simplex $\sigma$ of $\mathcal{C}(N)$ is equal to
$[P]$ for some pants decomposition $P$ of $N$. There are different dimensional maximal simplices in $\mathcal{C}(N)$ (see Figure \ref{Fig0}). In the figure we see cross signs. This means that the interiors of the disks with cross signs inside are removed and the antipodal points on the resulting boundary components are identified. 

\begin{figure}
	\begin{center}  \epsfxsize=2.5in \epsfbox{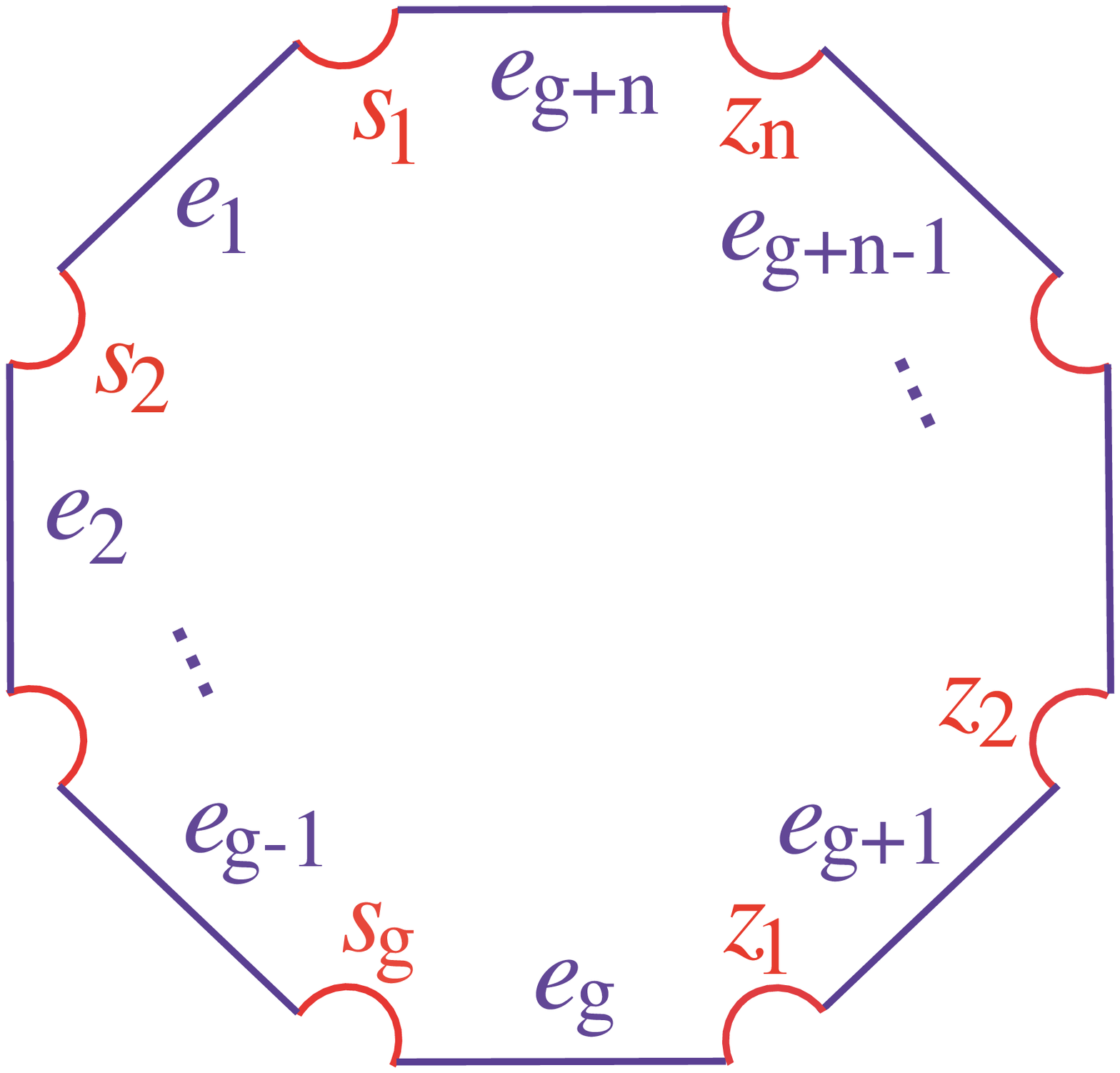}  \hspace{0.5cm}
		 \epsfxsize=2.2in \epsfbox{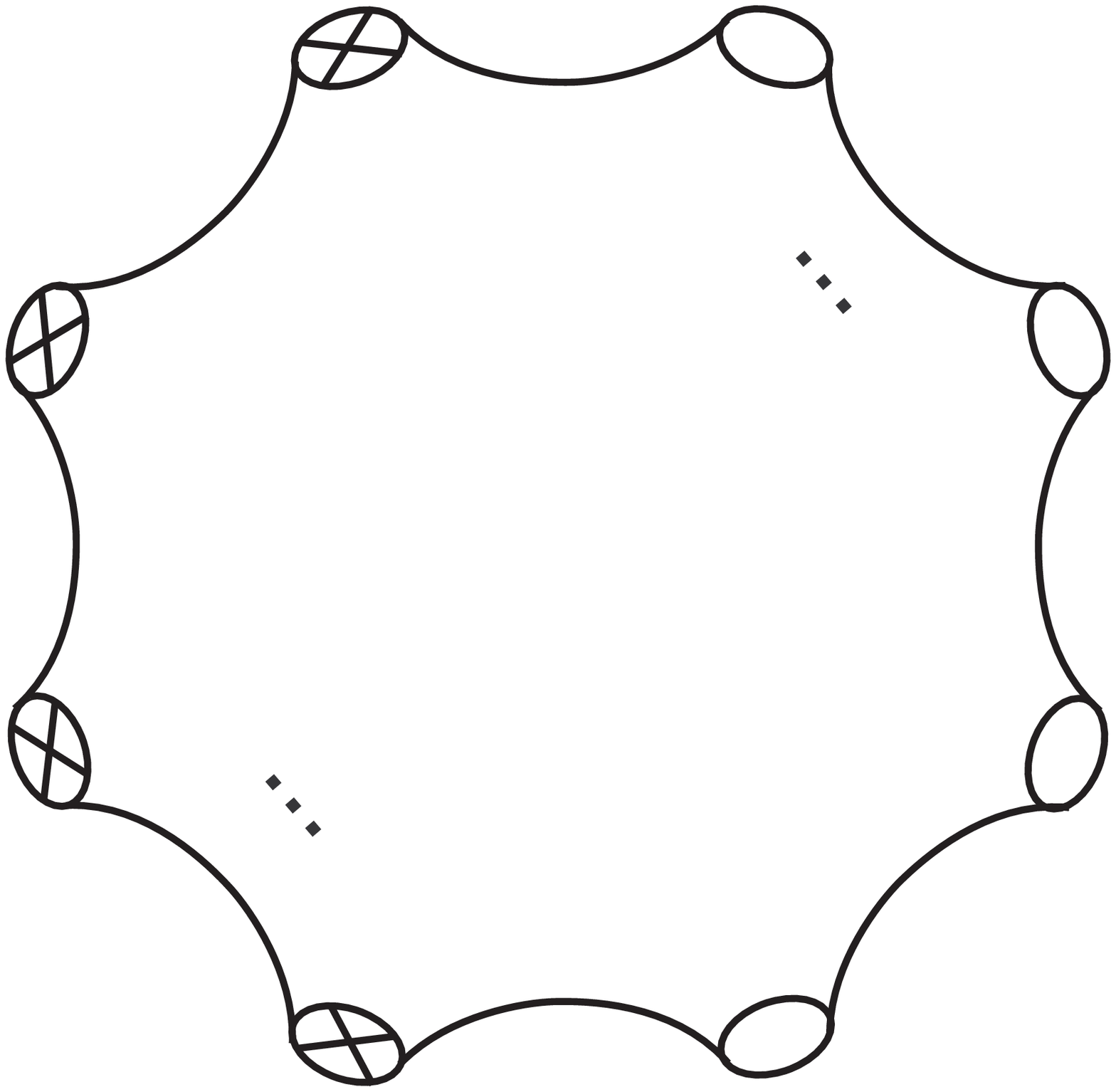} 
 
 		(i)   \hspace{6.2cm} (ii)  
		\caption {Curves in $\mathcal{X}$} \label{fig-2}
	\end{center}
\end{figure}

A subcomplex $\mathcal{A}$ of $\mathcal{C}(N)$ is called rigid if for every locally injective (injective on the star of every vertex) simplical map 
$\tau : \mathcal{A} \rightarrow \mathcal{C}(N)$, there exists a homeomorphism $h : N \rightarrow N$ such that $H(\alpha) = \tau(\alpha)$
for every vertex $\alpha$ in $\mathcal{A}$ where $H=[h]$. 
The finite rigid subcomplex in $\mathcal{C}(N)$ given by Ilbira-Korkmaz in \cite{IlK} is as follows:
Let $D$ be a disk in the plane whose boundary is a $(2g + 2n)$-gon such that the edges are labeled as $s_1, e_1, s_2, e_2, \cdots, s_g, e_g, z_1, e_{g+1}, z_2,  e_{g+2}, \cdots, e_{g+n-1}, z_n, e_{g+n}$ as shown in Figure \ref{fig-2} (i). 
By glueing two copies of $D$ along $e_i$ for each $i$ we get a sphere $S$ with 
$g+n$ holes. By identifying the antipodal points on each boundary component of $S$ that are formed by $s_i$ for each $i$, we get a nonorientable surface $N$ of genus $g$ with $n$ boundary components. Some arcs on $D$ correspond to simple closed curves on $N$ after gluing. By considering this correspondence some curves are defined as follows:

$\bullet$ For each $i$, let $a_i$ be the 1-sided curve on $N$ that corresponds to the arc $s_i$ on $D$.  

$\bullet$ For $1 \leq i \leq g$, $1 \leq j \leq g + n$  with $j \neq i$ , $j \neq i-1 \ (mod \ (g + n))$, let $a_{i,j}$ be the one-sided curve on $N$ that corresponds to a line segment joining the midpoint of $s_i$ and a point of $e_j$ on $D$. 

$\bullet$ For $1 \leq i, j \leq g+n$ and $|i-j| \geq 2$, let $b_{i,j}$ be the two-sided curve on $N$ that corresponds to a line segment joining a point of $e_i$ to a point of $e_j$ on $D$. Note that $b_{i,j} = b_{j,i}$.

We will denote by $t_{g+j}$ the boundary component of $N$ which corresponds to the arc $z_j$  for each $j$. 

Let $\mathcal{X}_1 = \{a_i, a_{i,j} :  1 \leq i \leq g, 1 \leq j \leq g + n, j \neq i, j \neq i-1 \ (mod \ (g + n))\}$, 

$\mathcal{X}_2 = \{b_{i,j} :  1 \leq i, j \leq g+n, 2 \leq |i-j| \leq g+n-2\}$, and $\mathcal{X} =  \mathcal{X}_1 \cup \mathcal{X}_2$. 

\begin{theorem} \label{rigid} (Ilbira-Korkmaz) Let $N$ be a compact, connected, nonorientable surface of genus $g$ with $n$ boundary components. If $g + n \neq 4$, then $\mathcal{X}$ is a finite rigid set in $\mathcal{C}(N)$.\end{theorem}

\begin{figure} \begin{center}  \epsfxsize=2in \epsfbox{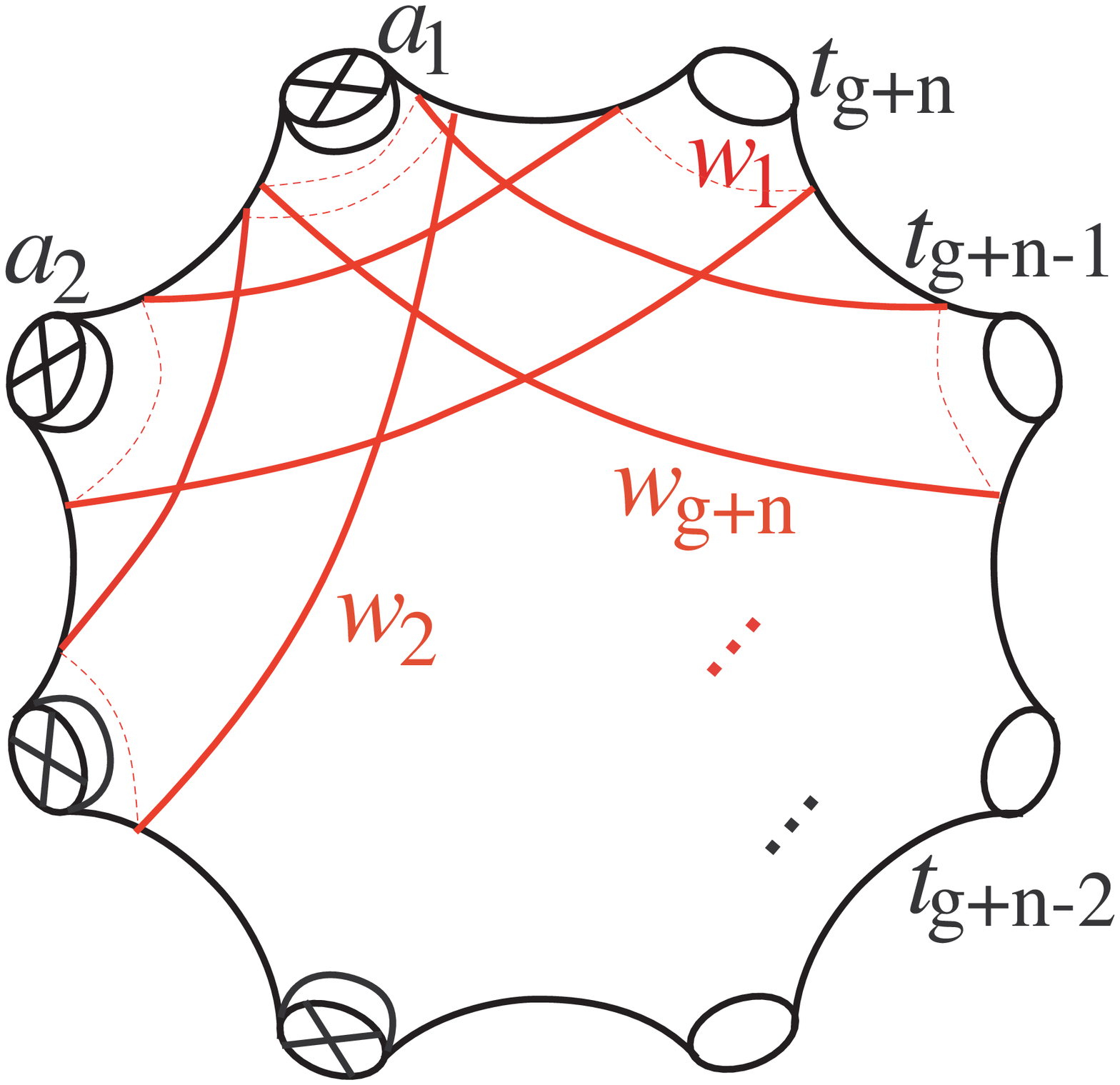} \hspace{0.2cm} \epsfxsize=2in \epsfbox{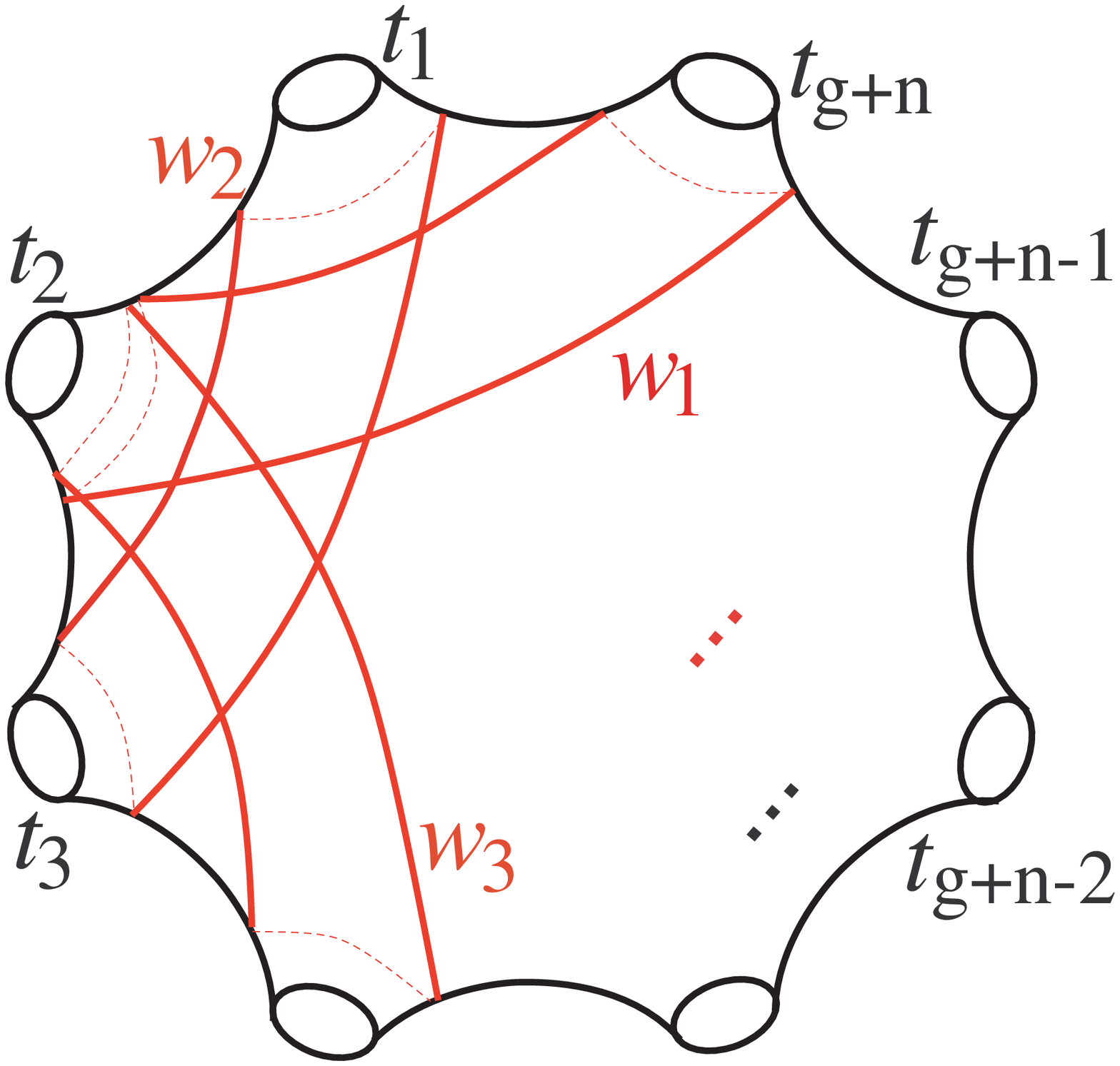}
		
		(i)   \hspace{4.7cm} (ii)  
		
		\epsfxsize=2in \epsfbox{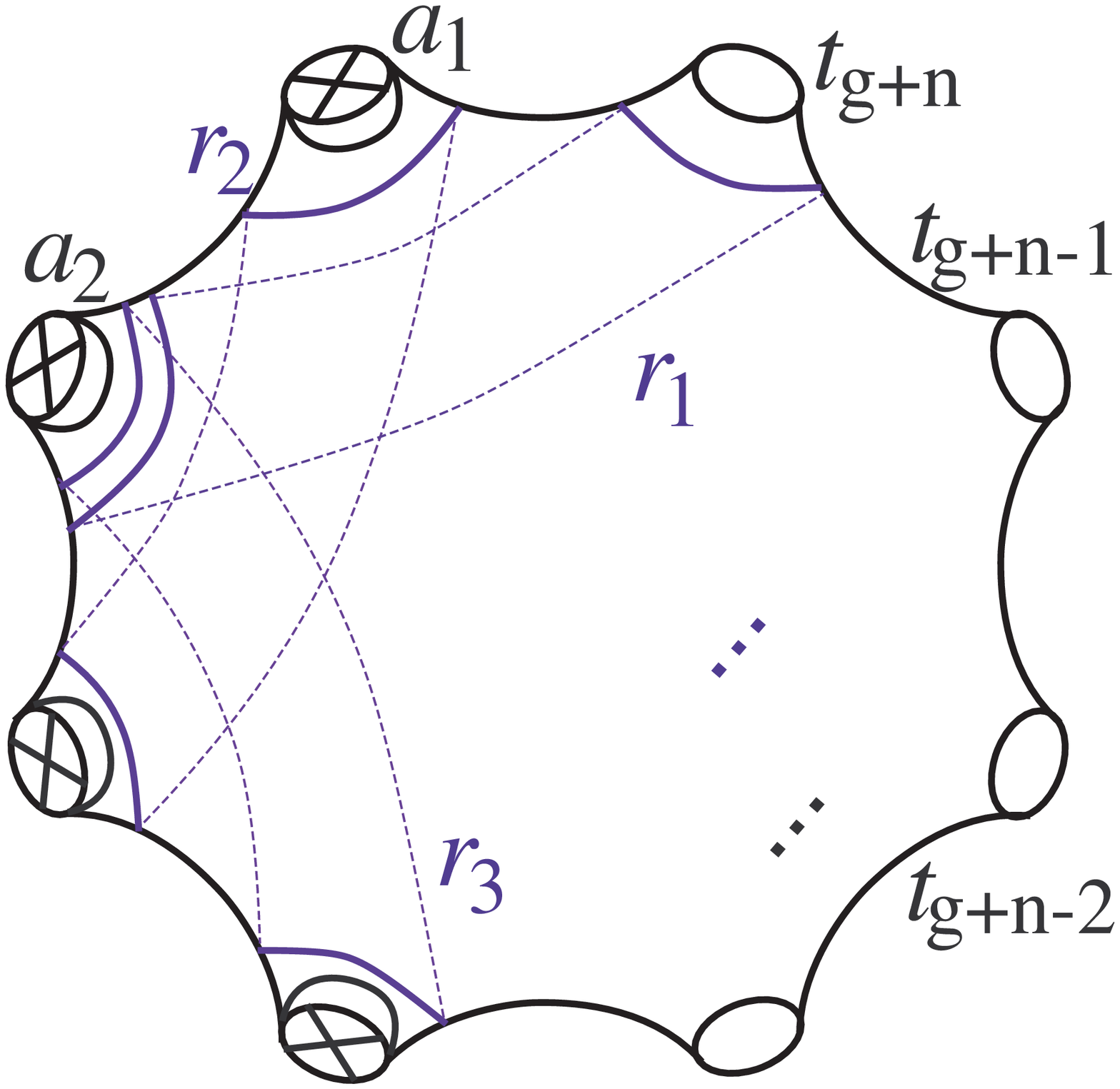} \hspace{0.2cm} 
		\epsfxsize=2in \epsfbox{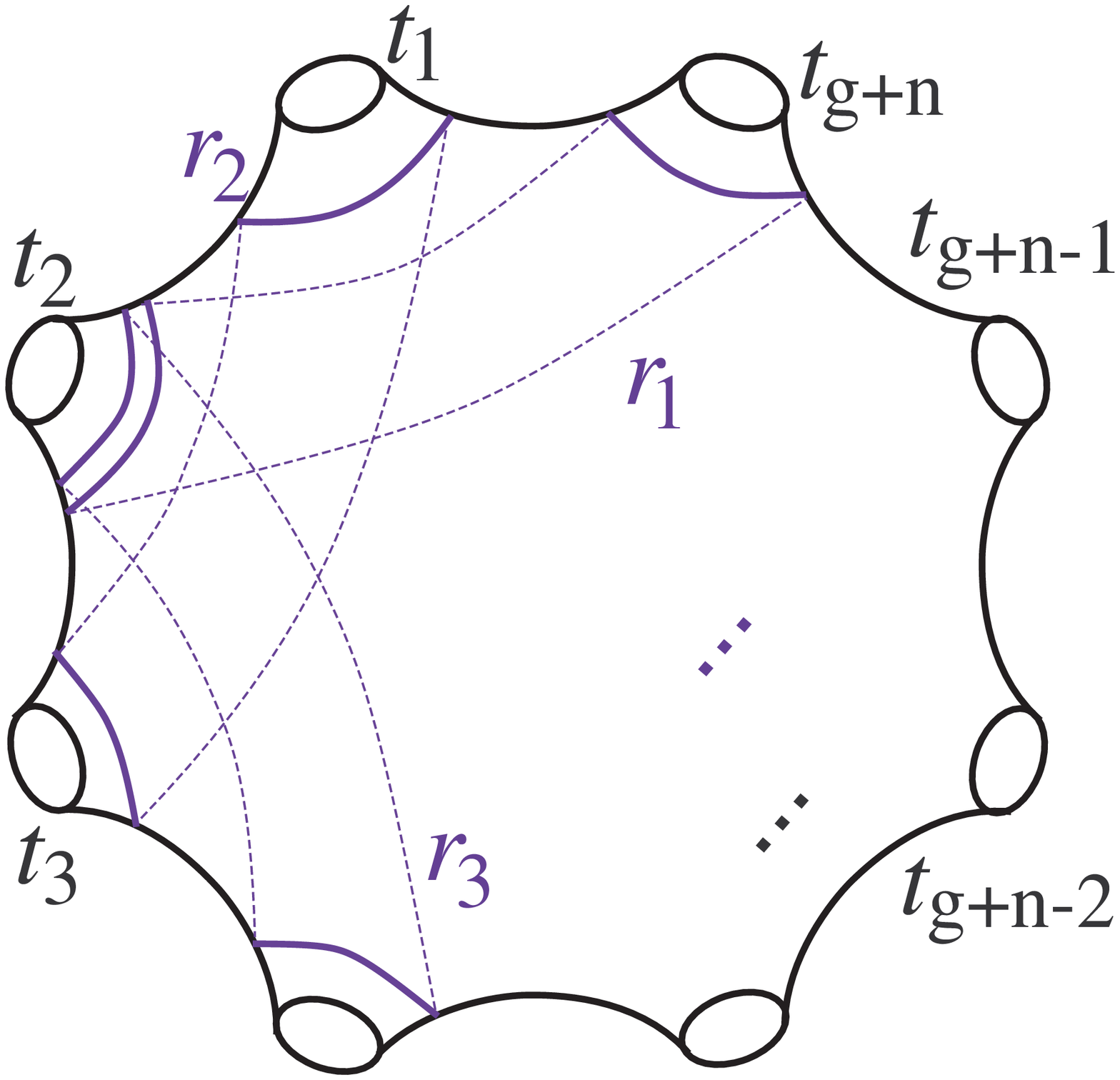}
		
			(iii)   \hspace{4.7cm} (iv)  
			
		 	\epsfxsize=2in \epsfbox{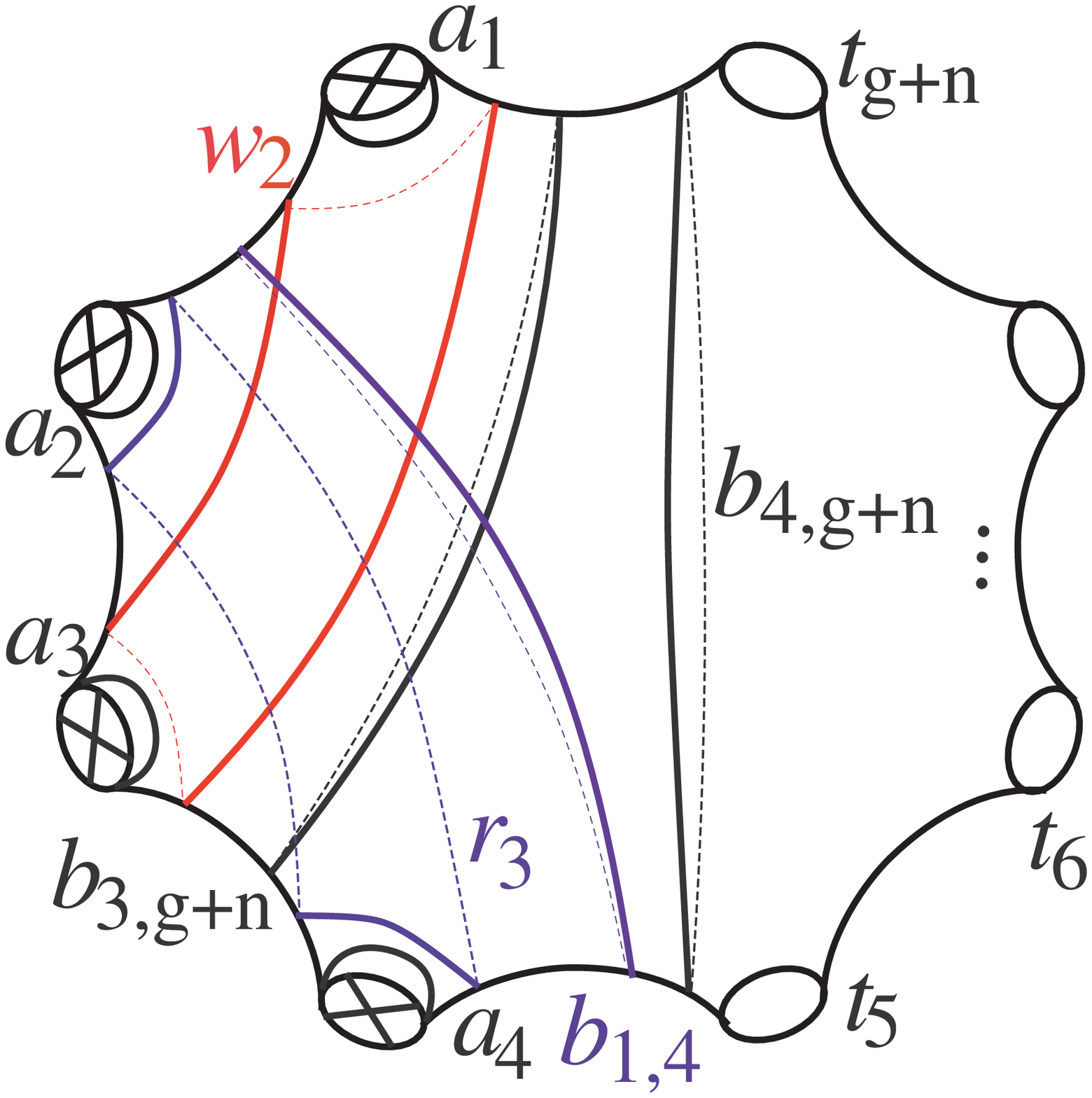} \hspace{0.2cm} 
		 \epsfxsize=2in \epsfbox{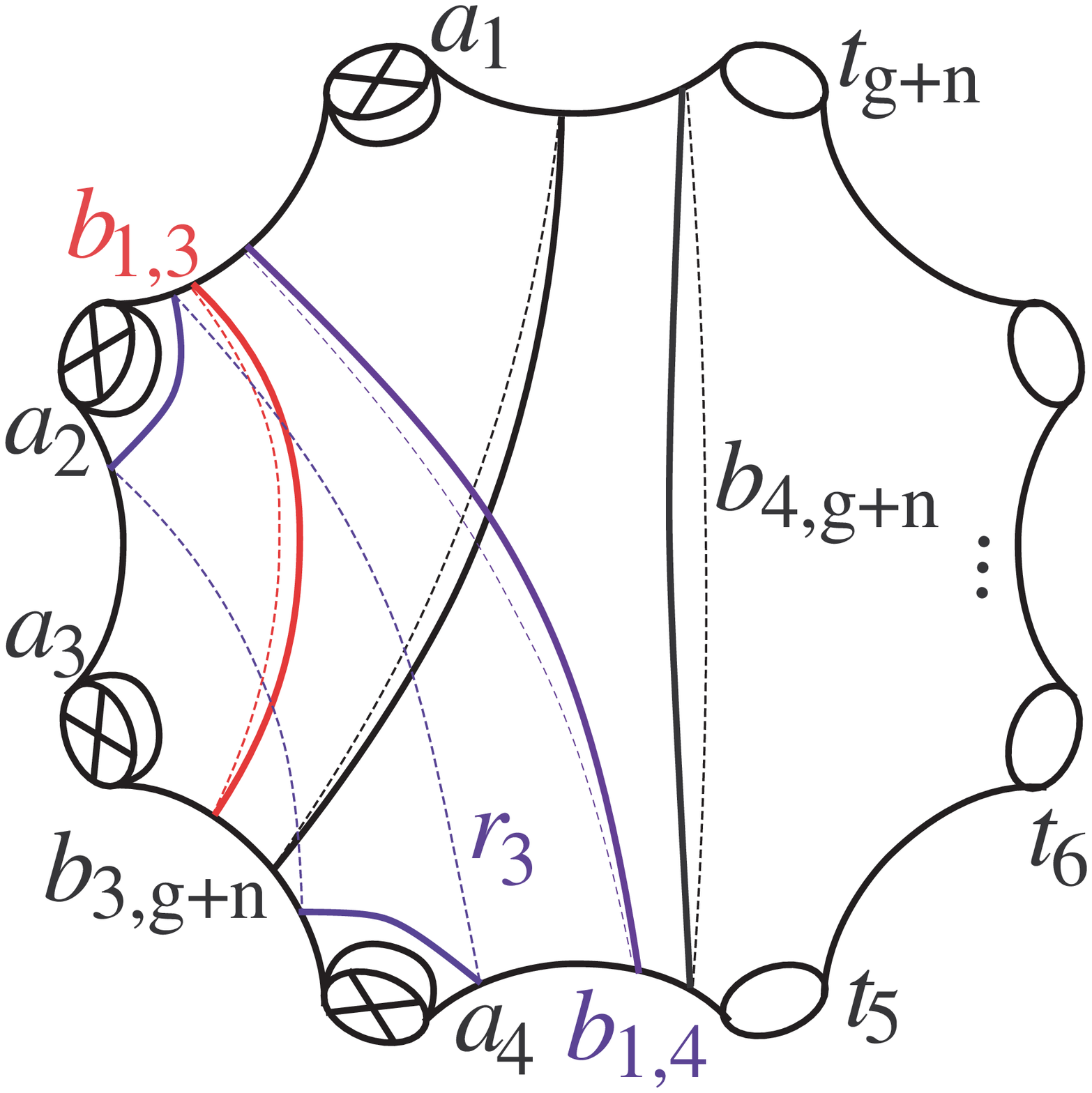}
		 
		 (v)   \hspace{4.7cm} (vi)  
		 
		 \epsfxsize=2in \epsfbox{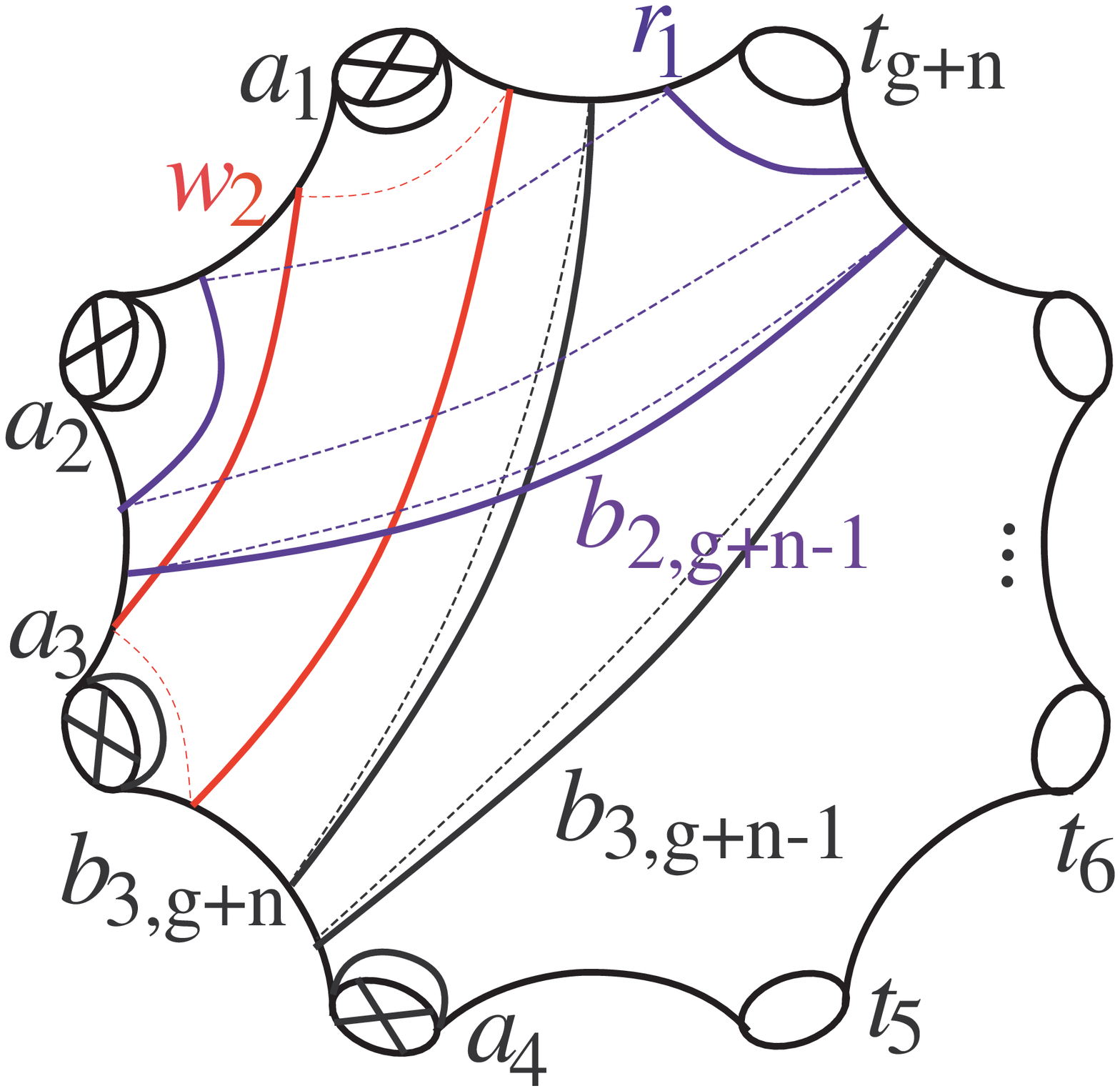} \hspace{0.2cm} 
		 \epsfxsize=2in \epsfbox{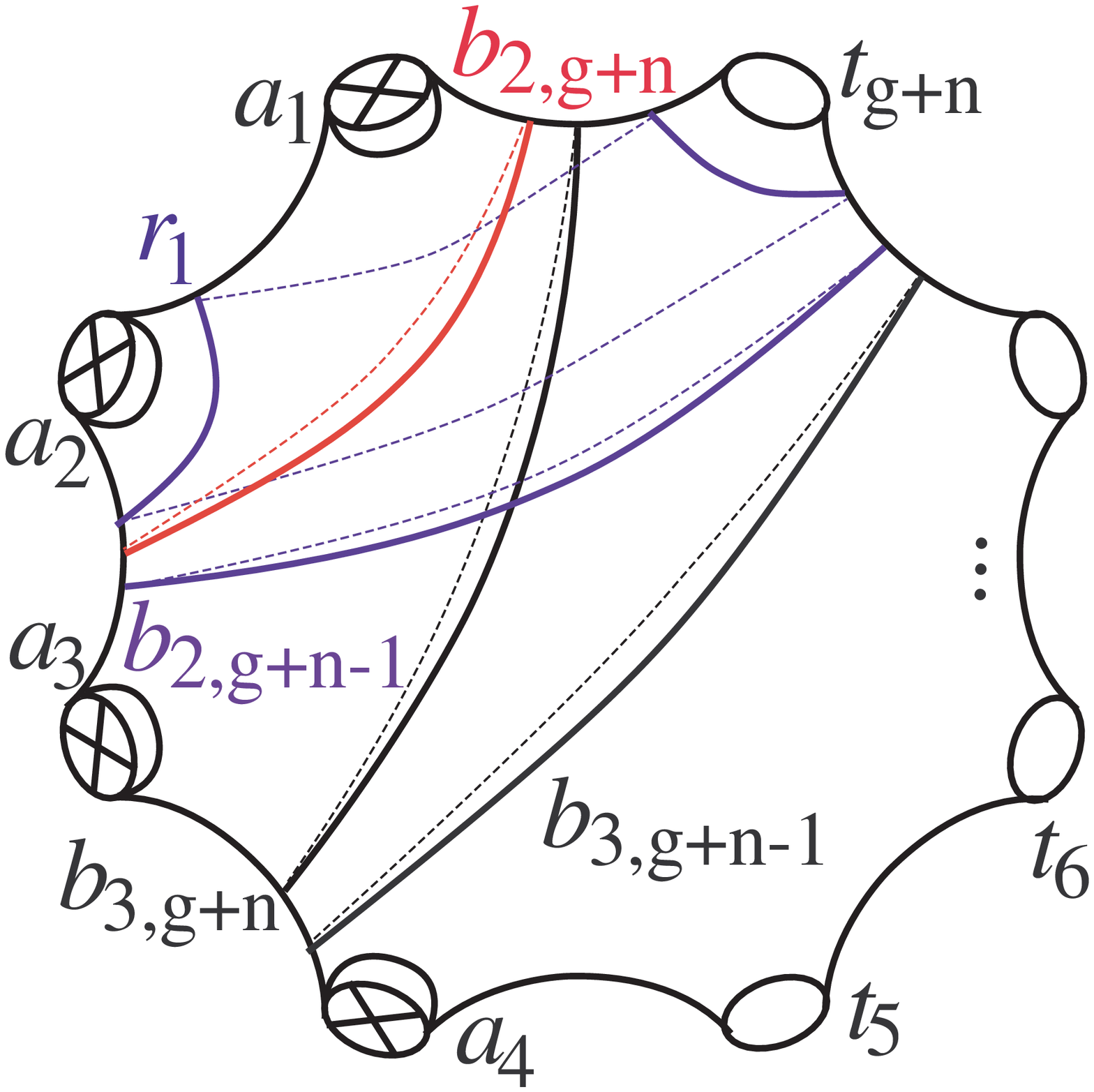}
		 
		 (vii)   \hspace{4.7cm} (viii)  
		\caption {Curves in $\mathcal{C}_2$} \label{fig-s11}
	\end{center}
\end{figure}

If $f: N \rightarrow N$ is a homeomorphism, then we will use the same notation for $f$ and $[f]$. We will also use the same notation for a subset $\mathcal{B}$  and the subcomplex of $\mathcal{C}(N)$ that is spanned by $\mathcal{B}$. Let $\mathcal{C}_1 = \mathcal{X}$ where $\mathcal{X}$ is defined as above. 

\begin{lemma} \label{inj} Suppose that $g+n \geq 5$. A superinjective simplicial  map $\lambda: \mathcal{C}_1 \rightarrow \mathcal{C}(N)$ is injective.\end{lemma}

\begin{proof} Let $x$ and $y$ be distinct elements in $\mathcal{C}_1$. We can find a curve $z \in \mathcal{C}_1$ such that $i([x], [z])=0$ and $i([y], [z]) \neq 0$. Since $\lambda$ is superinjective, we have $i(\lambda([x]), \lambda([z]))=0$ and $i(\lambda([y]), \lambda([z])) \neq 0$. So, $\lambda([x]) \neq \lambda([y])$. Hence, $\lambda$ is injective.\end{proof}

\begin{lemma} \label{C_1-super} If $g+n \geq 5$, then $\mathcal{C}_1$ is a finite superrigid set.\end{lemma}

\begin{proof} The proof follows from Theorem \ref{rigid} and Lemma \ref{inj}.\end{proof}\\

We will enlarge $\mathcal{C}_1$ to other superrigid sets in $\mathcal{C}(N)$. To make the notation simpler we will introduce new curves as follows: Let $w_i, r_i$ for $i= 1, 2, \cdots, g+n$ be as shown in Figure \ref{fig-s11} (i) and (iii). If we cut $N$ along all $a_i$ we get a sphere $S$ with $g+n$ boundary components. Let $t_i$ be the boundary component of $S$ that comes from cutting $N$ along $a_i$ for each $i$. Then $S$ has boundary components $t_1, t_2, \cdots, t_{g+n}$ and our curves $w_i, r_i$ are as shown in Figure \ref{fig-s11} (ii) and (iv) on $S$, (and hence they will have the same configuration in the complement of $\bigcup a_i$).
Let $\mathcal{C}_2 = \{w_1, w_2,  \cdots, w_{g+n}, r_1, r_2, \cdots, r_{g+n}\}$ where the curves are as shown in Figure \ref{fig-s11} (i), (iii), see also Figure \ref{Fig-2-a} (i), (ii). We will use pentagons in the proof of the following lemma. The pentagons are defined as follows: 
Let $\alpha_1, \alpha_2, \alpha_3, \alpha_4, \alpha_5$ be five distinct vertices in $\mathcal{C}(N)$. We will say that $(\alpha_1, \alpha_2, \alpha_3, \alpha_4, \alpha_5)$ is a pentagon in $\mathcal{C}(N)$ if $i(\alpha_i, \alpha_{i+1})=0$ for every $i=1, 2, 3, 4, 5$ and $i(\alpha_i, \alpha_{i+1}) \neq 0$ otherwise, where $\alpha_6=\alpha_1$.  

\begin{lemma} \label{1c} Suppose that $g+n \geq 5$. If $\lambda: \mathcal{C}_1 \cup \mathcal{C}_2 \rightarrow \mathcal{C}(N)$ is a superinjective simplicial map, then
	
	$i(\lambda([w_1]), \lambda([b_{1,g+n-1}])) =2$, $i(\lambda([w_1]), \lambda([b_{2,g+n}])) =2$, 
	
	$i(\lambda([w_2]), \lambda([b_{1,3}])) =2$, $i(\lambda([w_2]), \lambda([b_{2,g+n}])) =2$, 
	
	$i(\lambda([w_3]), \lambda([b_{1,3}])) =2$, $i(\lambda([w_3]), \lambda([b_{2,4}])) =2, \cdots,$ 
	
	$i(\lambda([w_{g+n}]), \lambda([b_{1,g+n-1}])) =2$, $i(\lambda([w_{g+n}]), \lambda([b_{g+n-2,g+n}])) =2$.\end{lemma}

\begin{proof} We will give the proof when $g=4$ and $n \geq 2$. The proofs for the other cases are  similar. It is easy to see that $\lambda$ is injective (see Lemma \ref{inj}). Since $\lambda$ is injective, by Theorem \ref{rigid} there exists a homeomorphism $h: N \rightarrow N$ such that $h([x]) = \lambda([x])$ for every $x$ in $\mathcal{C}_1$. 
  
We will first show that $i(\lambda([w_2]), \lambda([b_{1,3}])) =2$. 
Consider the curves given in Figure \ref{fig-s11} (v) and (vi). Let $M$ be the five holed sphere bounded by $a_1, a_2, a_3, a_4, b_{4,g+n}$. The curves 
$(b_{3, g+n}, b_{1,3}, b_{1,4}, r_3, w_2)$ form a pentagon in $C(M)$. For each pair of curves $x, y \in \{b_{3, g+n}, b_{1,3}, b_{1,4}, r_3, w_2\}$ we know that $i([x], [y]) = 0$ if and only if $i(\lambda([x]), \lambda([y])) = 0$ since $\lambda$ is superinjective.  If $b'_{3, g+n}, b'_{1,3}, b'_{1,4}, r'_3, w'_2, a'_1, a'_2, a'_3, a'_4, b'_{4,g+n}$ are minimally intersecting representatives of $\lambda([b_{3, g+n}]), \lambda([b_{1,3}]), \lambda([b_{1,4}]), \lambda([r_3]), \lambda([w_2]), \lambda([a_1]),  \lambda([a_2]),$ $\lambda([a_3]),\lambda([a_4]), \lambda([b_{4,g+n}])$ respectively, then $(b'_{3, g+n}, b'_{1,3}, b'_{1,4}, r'_3, w'_2)$ form a pentagon in the five holed sphere bounded by $a'_1, a'_2, a'_3, a'_4, b'_{4,g+n}$. Since $h([x]) = \lambda([x])$ for every $x$ in $\mathcal{C}_1$ and $(b'_{3, g+n}, b'_{1,3}, b'_{1,4}, r'_3, w'_2)$ form a pentagon in the five holed sphere bounded by $a'_1, a'_2, a'_3, a'_4, b'_{4,g+n}$, by using 
Korkmaz's Theorem 3.2 in \cite{K1}, we get $i(\lambda[w_2]), \lambda([b_{1,3}]) =2$. By using similar arguments we get all the intersections in the statement (to see $i(\lambda[w_2]), \lambda([b_{2,g+n}]) =2$ for example, use the pentagon formed by the curves $w_2, r_1, $ $b_{2,g+n-1}, b_{2,g+n}, b_{3,g+n}$ in the five holed sphere bounded by $a_1, a_2, a_3, t_{g+n},  b_{3,g+n-1}$ given in Figure \ref{fig-s11} (vii) and (viii)).\end{proof}
 
\begin{lemma} \label{1d} Suppose that $g+n \geq 5$. If $\lambda: \mathcal{C}_1 \cup \mathcal{C}_2 \rightarrow \mathcal{C}(N)$ is a superinjective simplicial map, then 
	
	 $i(\lambda([r_1]), \lambda([b_{1,g+n-1}])) =2$, $i(\lambda([r_1]), \lambda([b_{2,g+n}])) =2$, 
	 
	 $i(\lambda([r_2]), \lambda([b_{1,3}])) =2$, $i(\lambda([r_2]), \lambda([b_{2,g+n}])) =2$, 
	 
	 $i(\lambda([r_3]), \lambda([b_{1,3}])) =2$, $i(\lambda([r_3]), \lambda([b_{2,4}])) =2, \cdots,$ 
	 
	 $i(\lambda([r_{g+n}]), \lambda([b_{1,g+n-1}])) =2$, $i(\lambda([r_{g+n}]), \lambda([b_{g+n-2,g+n}])) =2$.   \end{lemma}
 
\begin{proof} The proof is similar to the proof given in Lemma \ref{1c}.\end{proof}
  
\begin{lemma} \label{C_2} If $g+n \geq 5$, then $\mathcal{C}_1 \cup \mathcal{C}_2$ is a finite superrigid set with trivial pointwise stabilizer.\end{lemma}

\begin{proof} Let $\lambda: \mathcal{C}_1 \cup \mathcal{C}_2 \rightarrow \mathcal{C}(N)$ be a superinjective simplicial map. It is easy to see that $\lambda$ is injective (see the proof of Lemma \ref{inj}). Since $\lambda$ is injective, by Theorem \ref{rigid}, there exists a homeomorphism $h: N \rightarrow N$ such that $h([x]) = \lambda([x])$ for every $x$ in $\mathcal{C}_1$. We will show that $h([x]) = \lambda([x])$ for every $x$ in $\mathcal{C}_2$ as well.
	
There exists a homeomorphism $\phi : N \rightarrow N$ of order two (reflection through the plane of the paper in Figure \ref{fig-s11} (i)) such that the map $\phi_{*}$ induced by $\phi$ on $\mathcal{C}_1$ sends the isotopy class of each curve in $\mathcal{C}_1$ to itself and switches $r_i$ and $w_i$. By Lemma \ref{1c} and Lemma \ref{1d} we know that $i(\lambda([w_2]), \lambda([b_{1,3}])) =2$, $i(\lambda([w_2]), \lambda([b_{2,g+n}])) =2$, and $i(\lambda([r_2]), \lambda([b_{1,3}])) =2$, $i(\lambda([r_2]), \lambda([b_{2,g+n}])) =2$. The curves $r_2$ and $w_2$ are the only nontrivial curves up to isotopy disjoint from $a_1, a_2, a_3, b_{3,g+n}$ and intersects each of $b_{1,3}, b_{2,g+n}$ nontrivially twice in the four holed sphere cut by $a_1, a_2, a_3, b_{3,g+n}$. 
Since we know that $h([x]) = \lambda([x])$ for all the curves $a_1, a_2, a_3, b_{3,g+n}$ and $\lambda$ preserves these properties, by replacing $\lambda$ with $\lambda \circ \phi_{*}$ if necessary, we can assume that we have $h([w_2]) = \lambda([w_2])$. This implies that $h([r_2]) = \lambda([r_2])$ since $\lambda$ is injective. Then by using Lemma \ref{1c} and Lemma \ref{1d} and $\lambda$ preserves disjointness, we see that $h([w_i]) = \lambda([w_i])$ and $h([r_i]) = \lambda([r_i])$ for each $i= 1, 2, \cdots, g+n$. So, $\mathcal{C}_1 \cup \mathcal{C}_2$ is a finite superrigid set. It is easy to see that the pointwise stabilizer of $\mathcal{C}_1$ has order two, generated by the reflection through the plane of the paper, and the pointwise stabilizer of $\mathcal{C}_1 \cup \mathcal{C}_2$ is trivial.\end{proof}\\

\begin{figure}[t]
	\begin{center} 
		\epsfxsize=2.3in \epsfbox{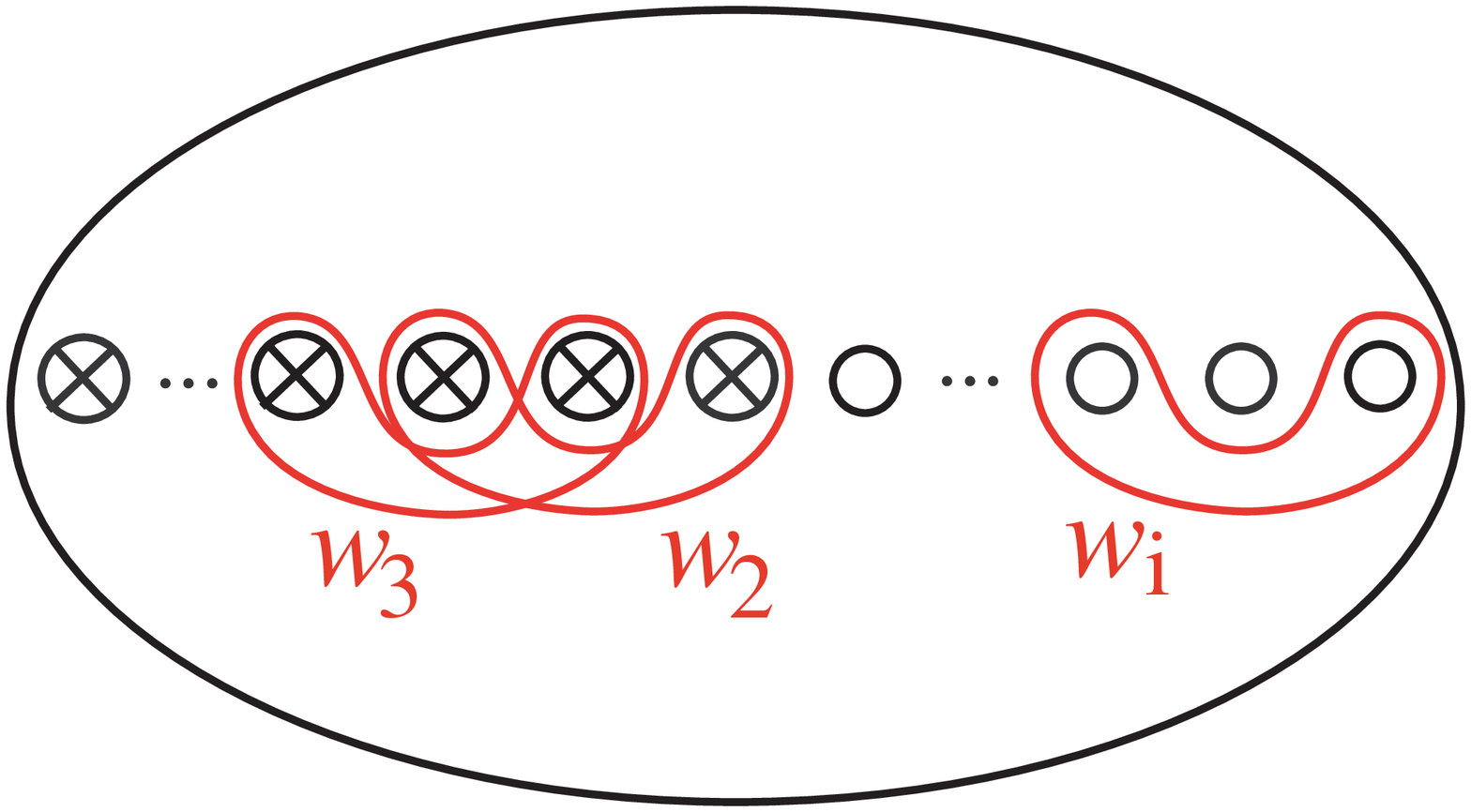} \hspace{0.1cm} 	\epsfxsize=2.3in \epsfbox{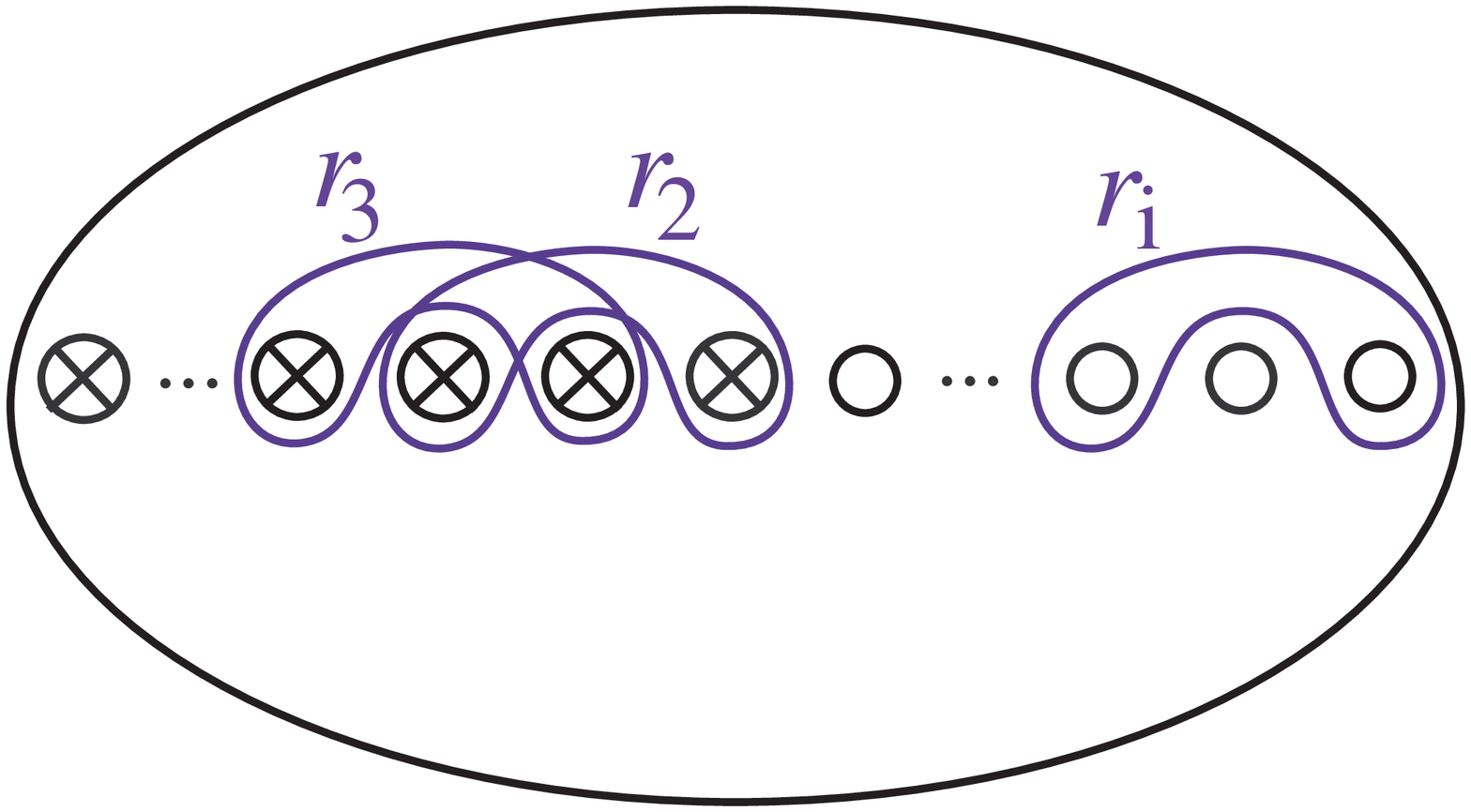} 
		
		(i)  \hspace{5cm}   (ii) 
		
		\epsfxsize=2.3in \epsfbox{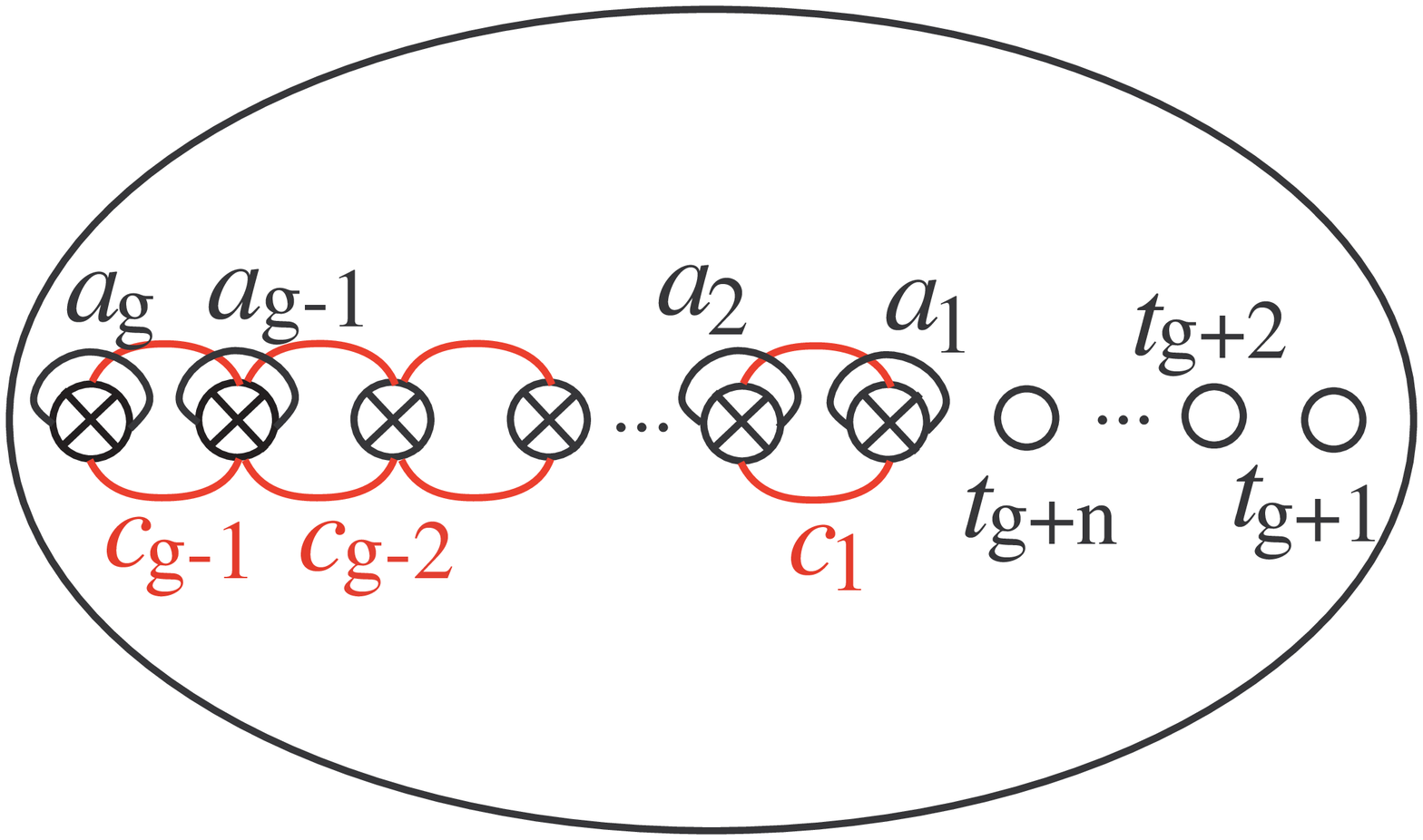} \hspace{0.1cm} 	\epsfxsize=2.3in \epsfbox{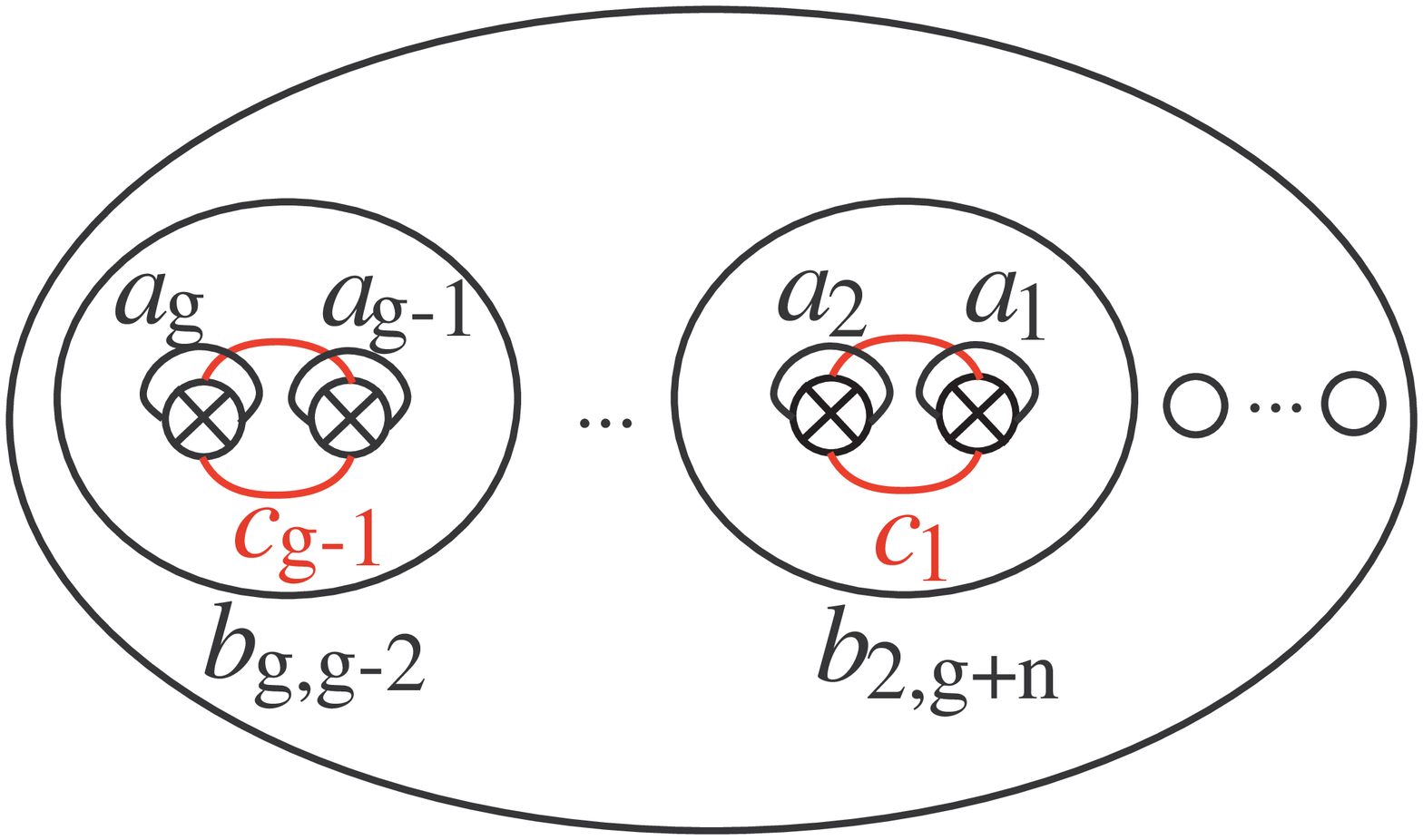} 
		
		(iii)  \hspace{5cm}   (iv) 
		
		\caption{Curves in $\mathcal{C}_2$ }
		\label{Fig-2-a}
	\end{center}
\end{figure}

In figure \ref{Fig-2-a} we show the curves $r_i, w_i$.  Let $t_i$ for $i=g+1, g+2, \cdots, g+n$ be the boundary components of $N$ as shown in Figure \ref{Fig-2-a} (iii). We will always assume that we have the ordering of the  crosscaps, $a_i$'s, and the boundary components $t_j$'s as shown in Figure \ref{Fig-2-a} without writing $a_i$ and $t_j$. 

Let $c_i$ for $i=1, 2, \cdots,g-1$ be the curves shown in Figure \ref{Fig-2-a} (iii) and (iv). The curve $c_1$ is, up to isotopy, the unique 2-sided curve in the Klein bottle with one hole bounded by $b_{2,g+n}$, and the curve $c_i$ when $i \geq 2$ is, up to isotopy, the unique 2-sided curve in the Klein bottle with one hole bounded by $b_{i+1,i-1}$.

\begin{figure}
	\begin{center} 
		\epsfxsize=1.58in \epsfbox{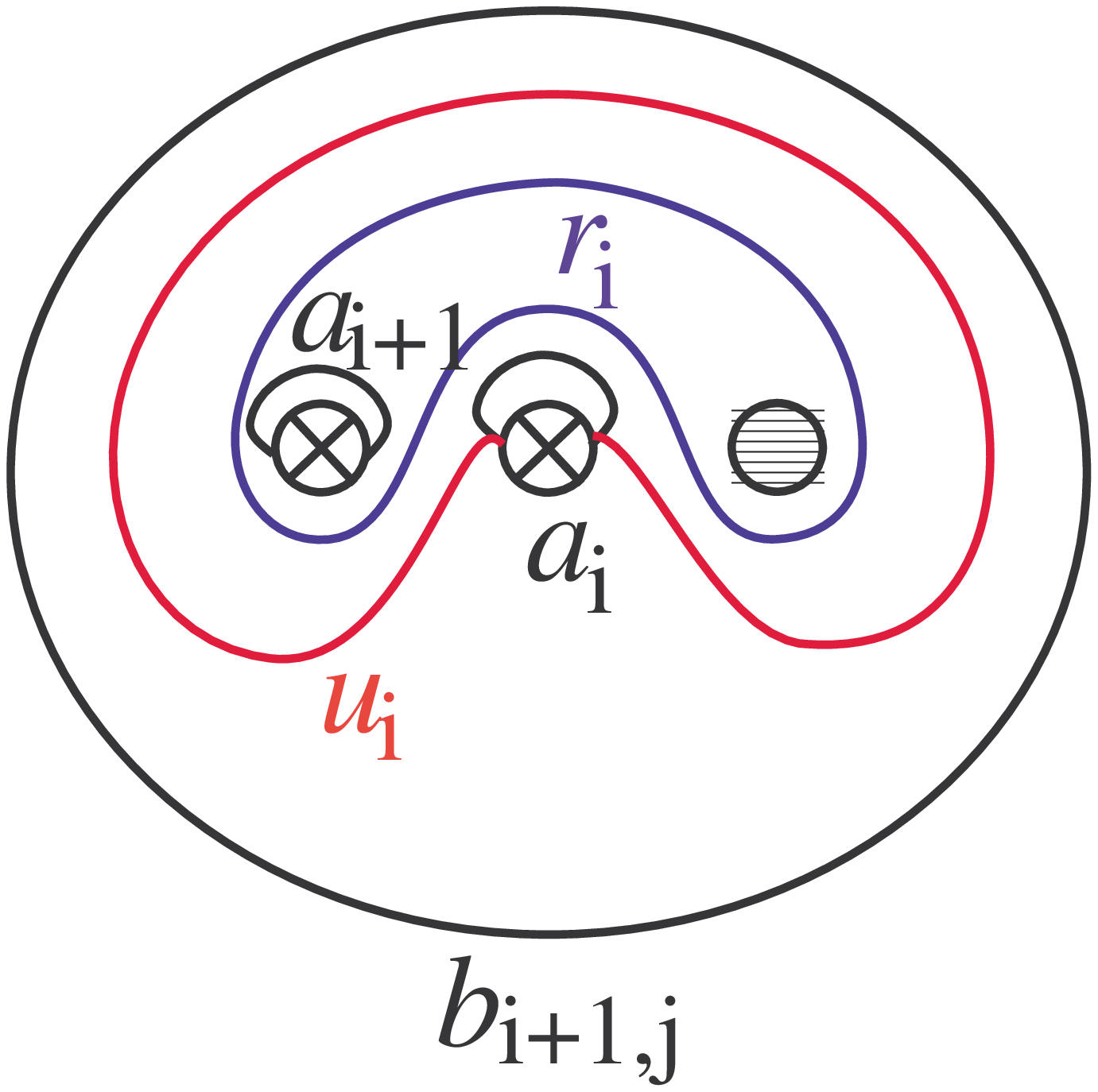} \hspace{0.1cm} 	\epsfxsize=1.58in \epsfbox{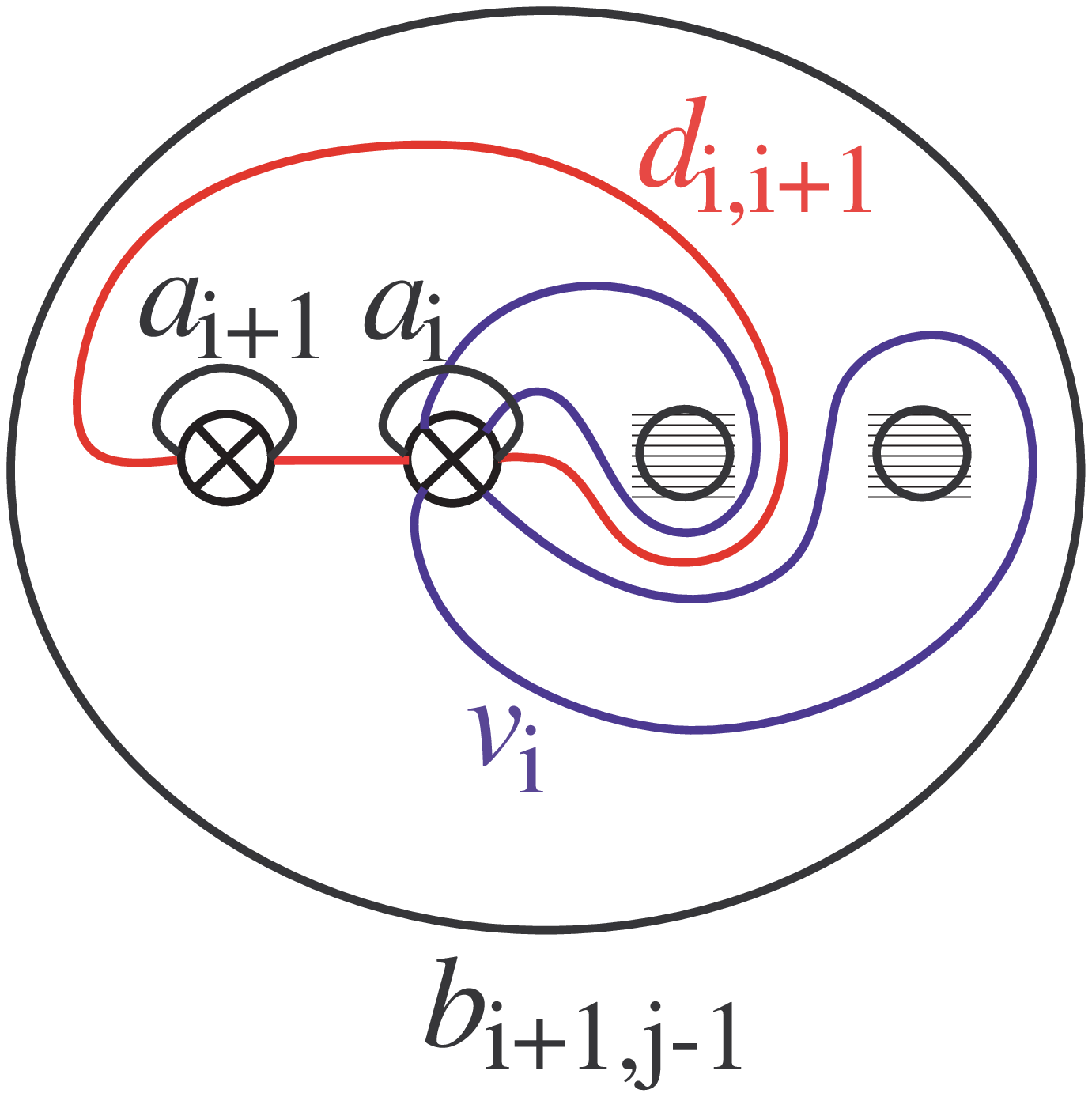}  \hspace{0.1cm} 	\epsfxsize=1.58in \epsfbox{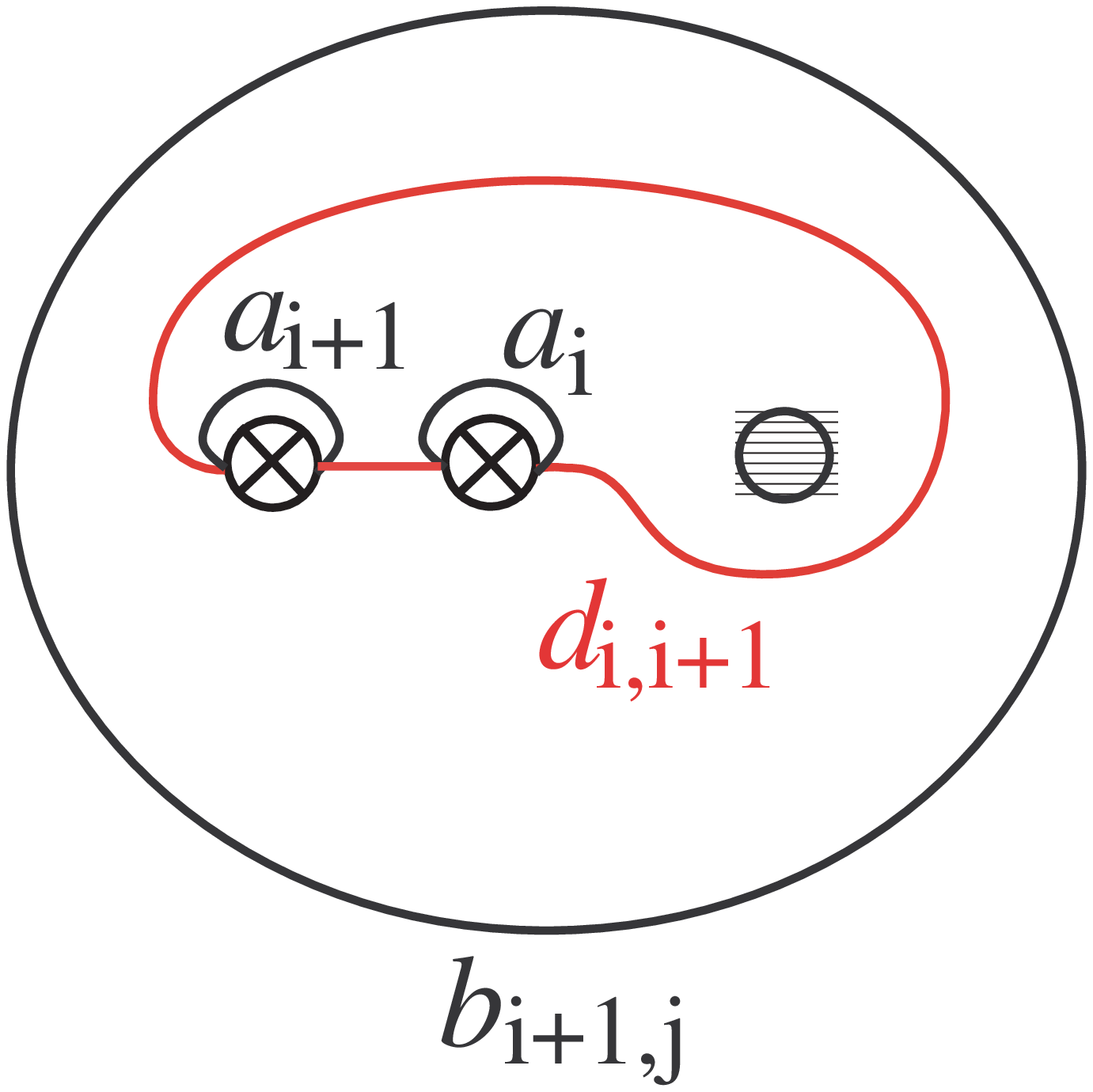}  \hspace{0.1cm}  
		
		(i)  \hspace{3.8cm}   (ii)  \hspace{3.8cm}   (iii) 
		
		\vspace{0.1cm}
		
		\epsfxsize=1.58in \epsfbox{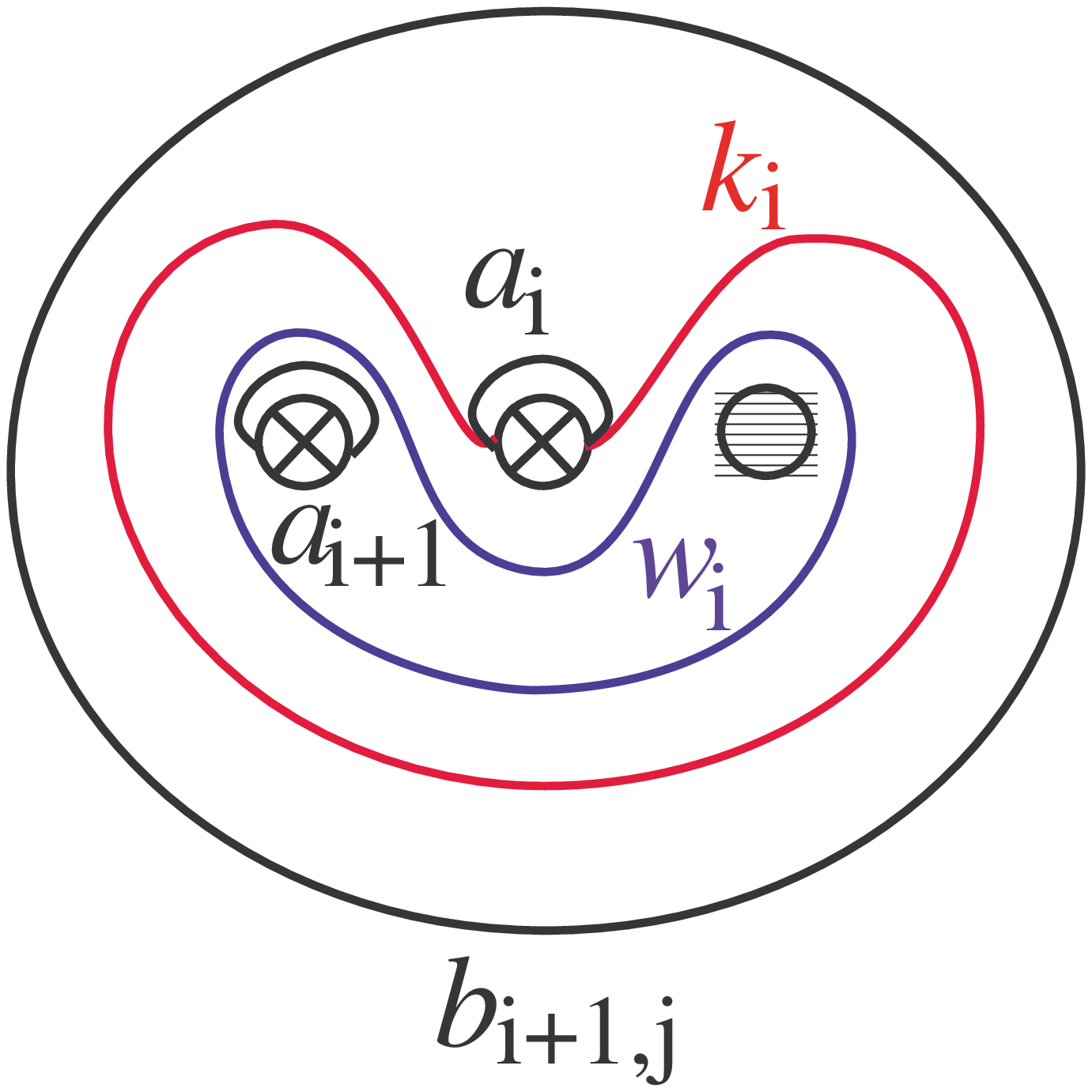} \hspace{0.1cm} 	
		\epsfxsize=1.58in \epsfbox{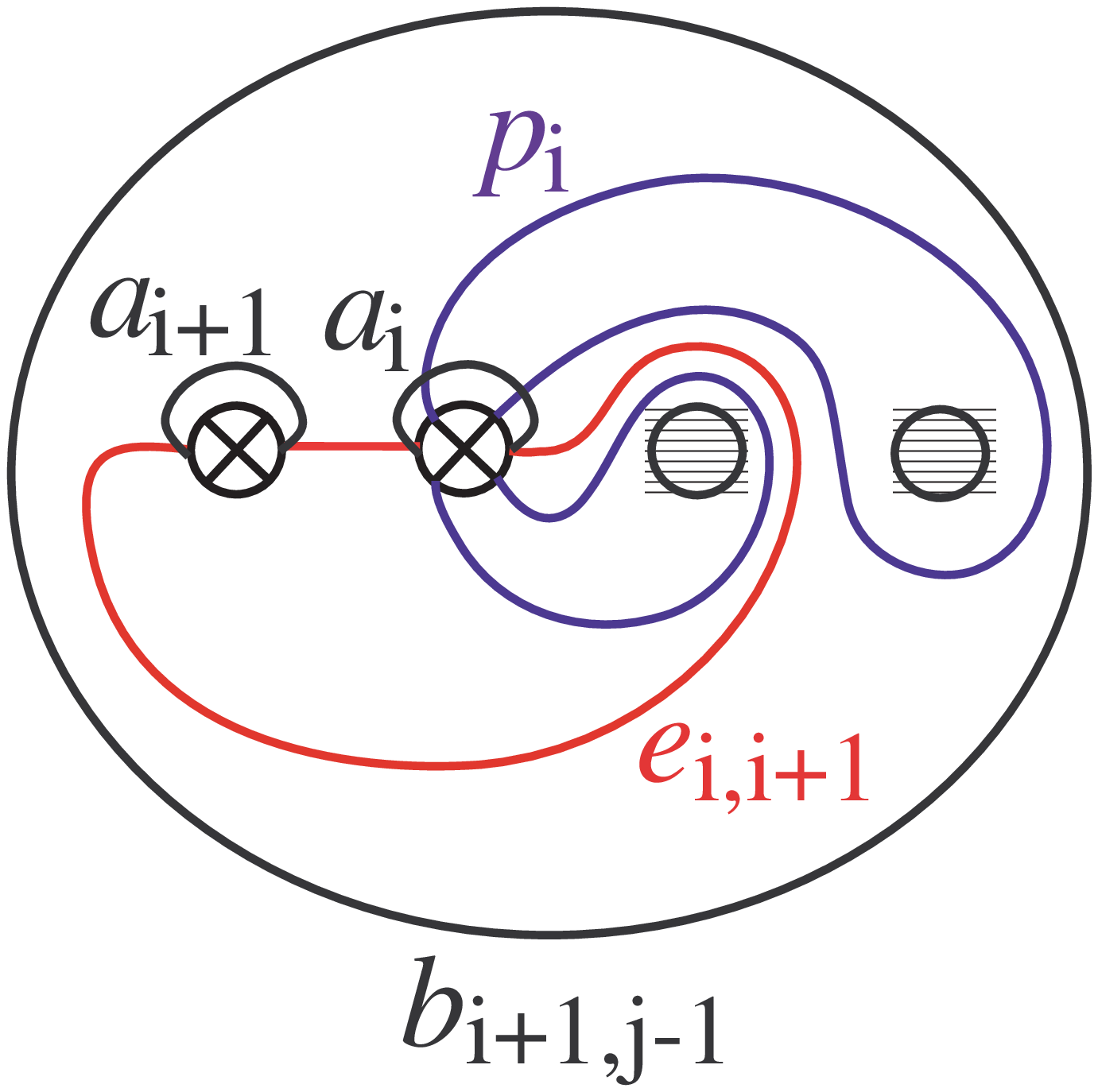}  \hspace{0.1cm} 
		\epsfxsize=1.58in \epsfbox{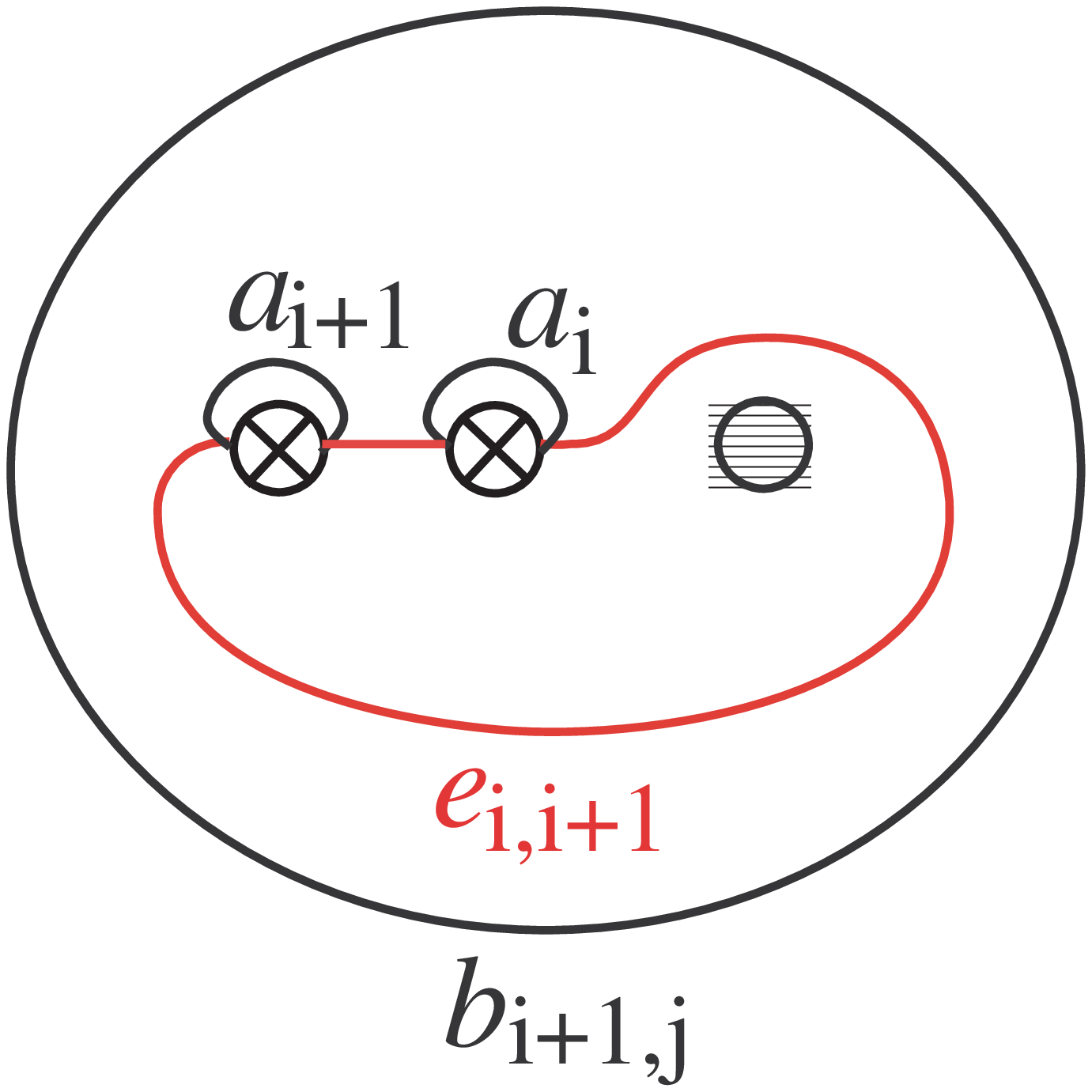} 
		
		(iv)  \hspace{3.8cm}   (v)  \hspace{3.8cm}   (vi) 
		
		\vspace{0.1cm}
		
		\epsfxsize=1.58in \epsfbox{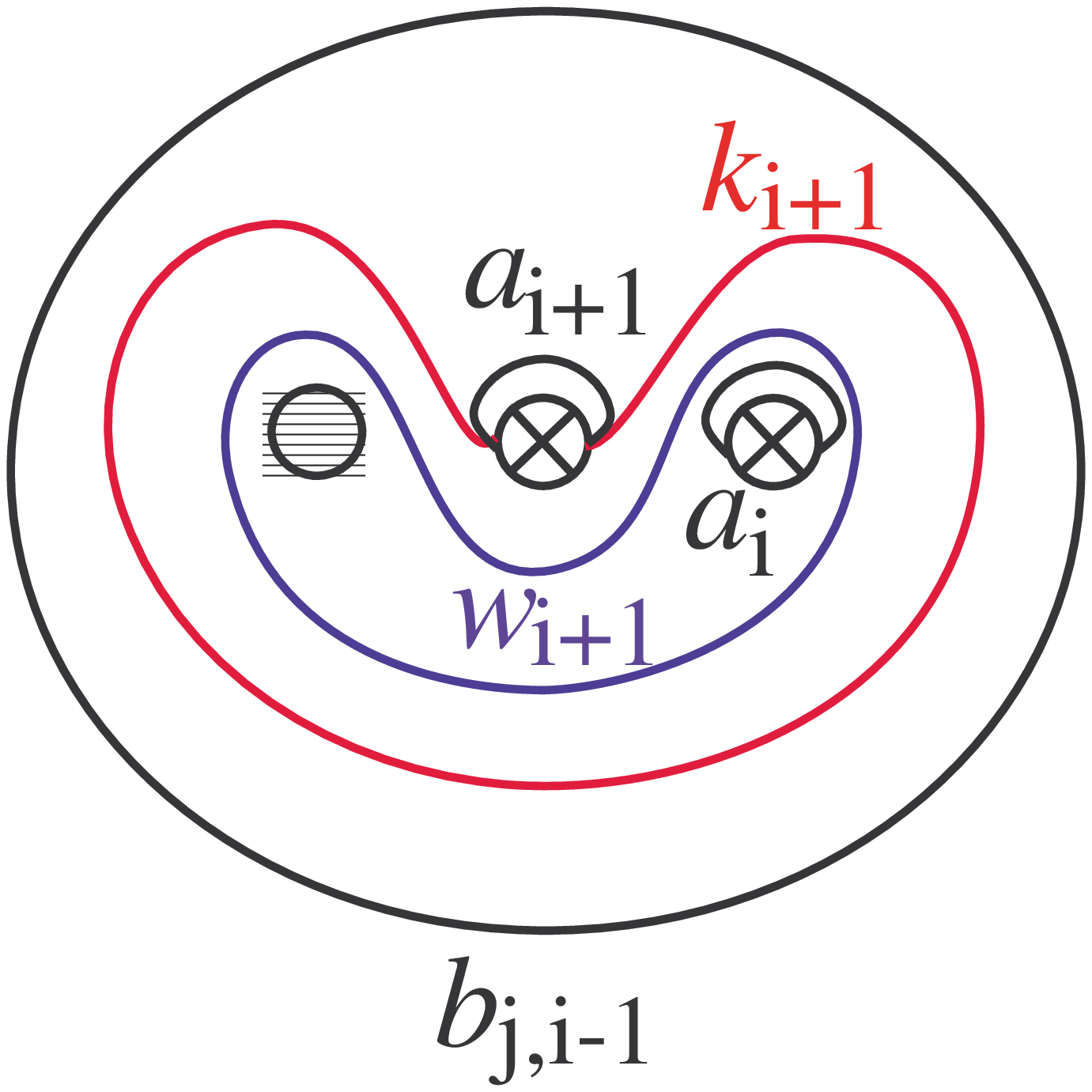} \hspace{0.1cm} 	
		\epsfxsize=1.58in \epsfbox{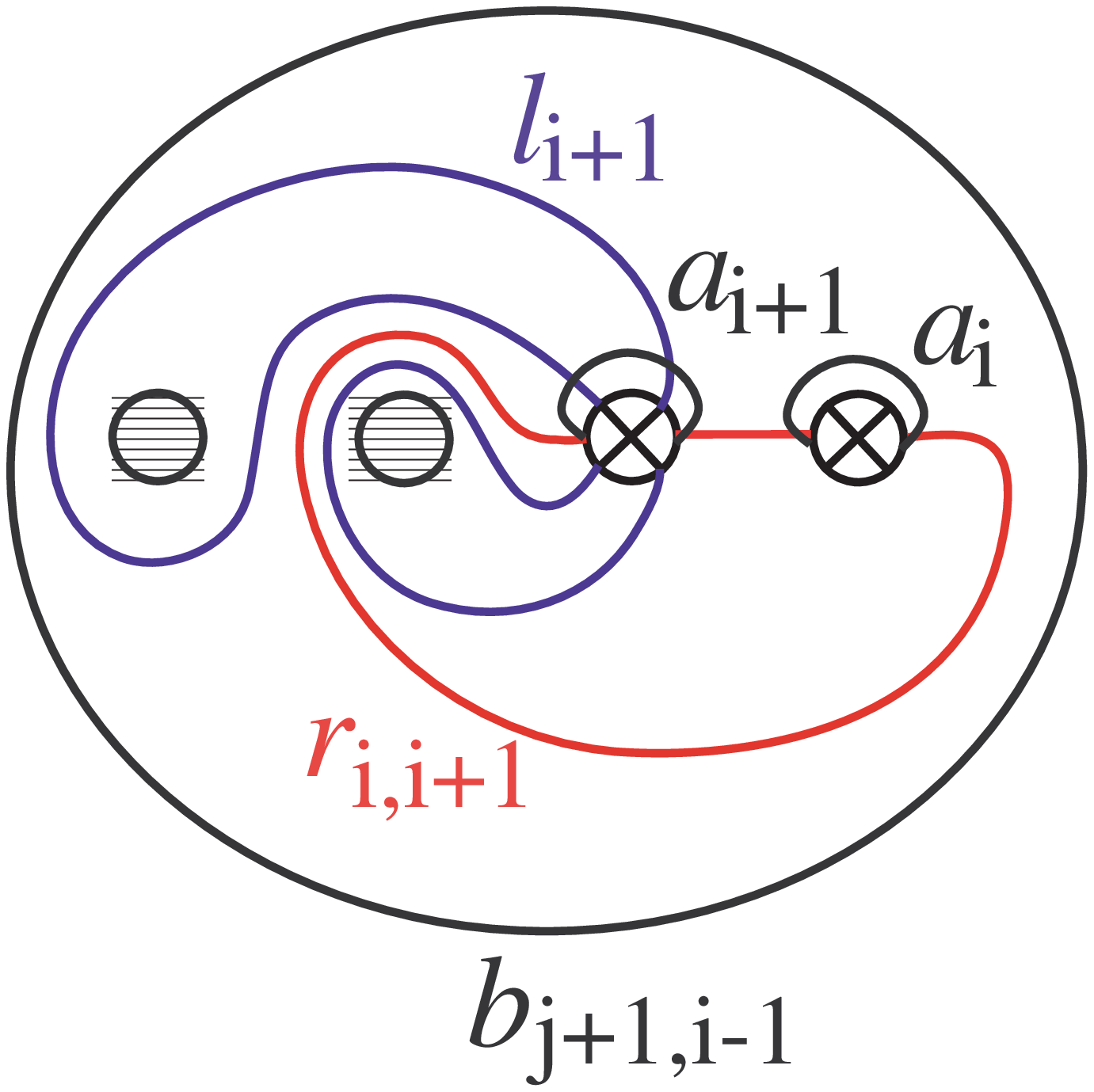}  \hspace{0.1cm} 
		\epsfxsize=1.58in \epsfbox{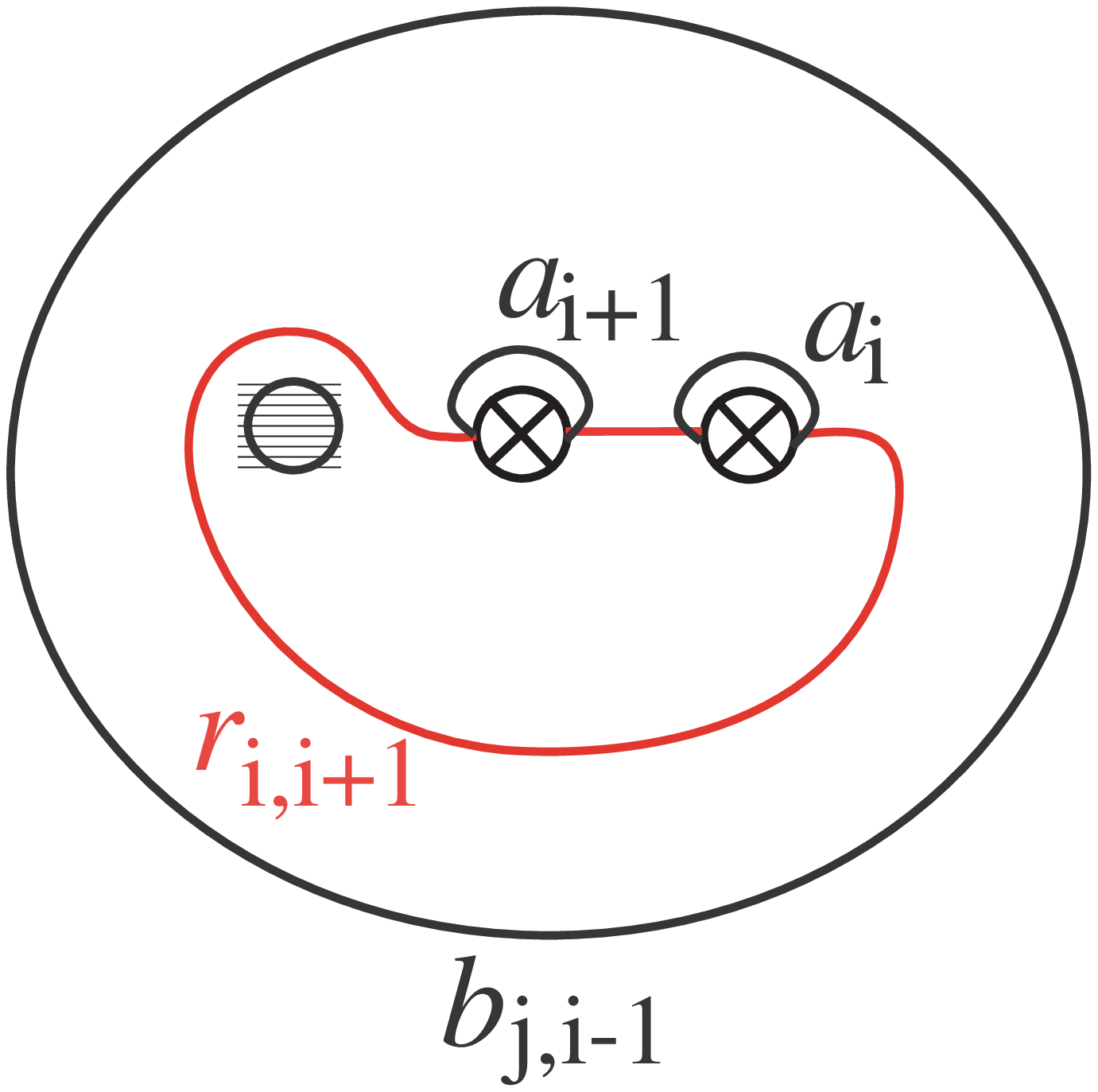} 
		
		(vii)  \hspace{3.8cm}   (viii)  \hspace{3.8cm}   (ix) 	
		
		\vspace{0.1cm}
		
		\epsfxsize=1.58in \epsfbox{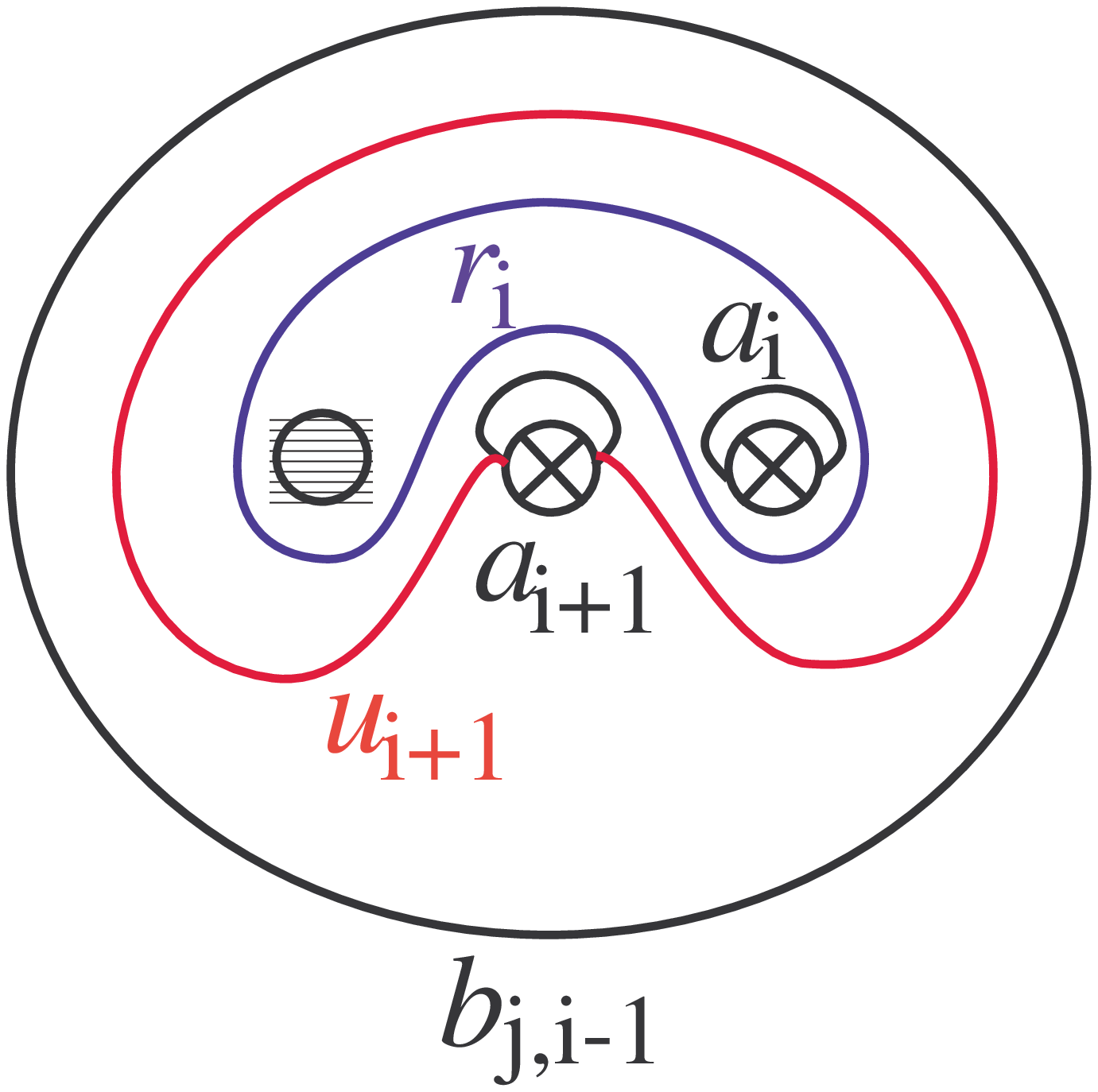} \hspace{0.1cm} 	
		\epsfxsize=1.58in \epsfbox{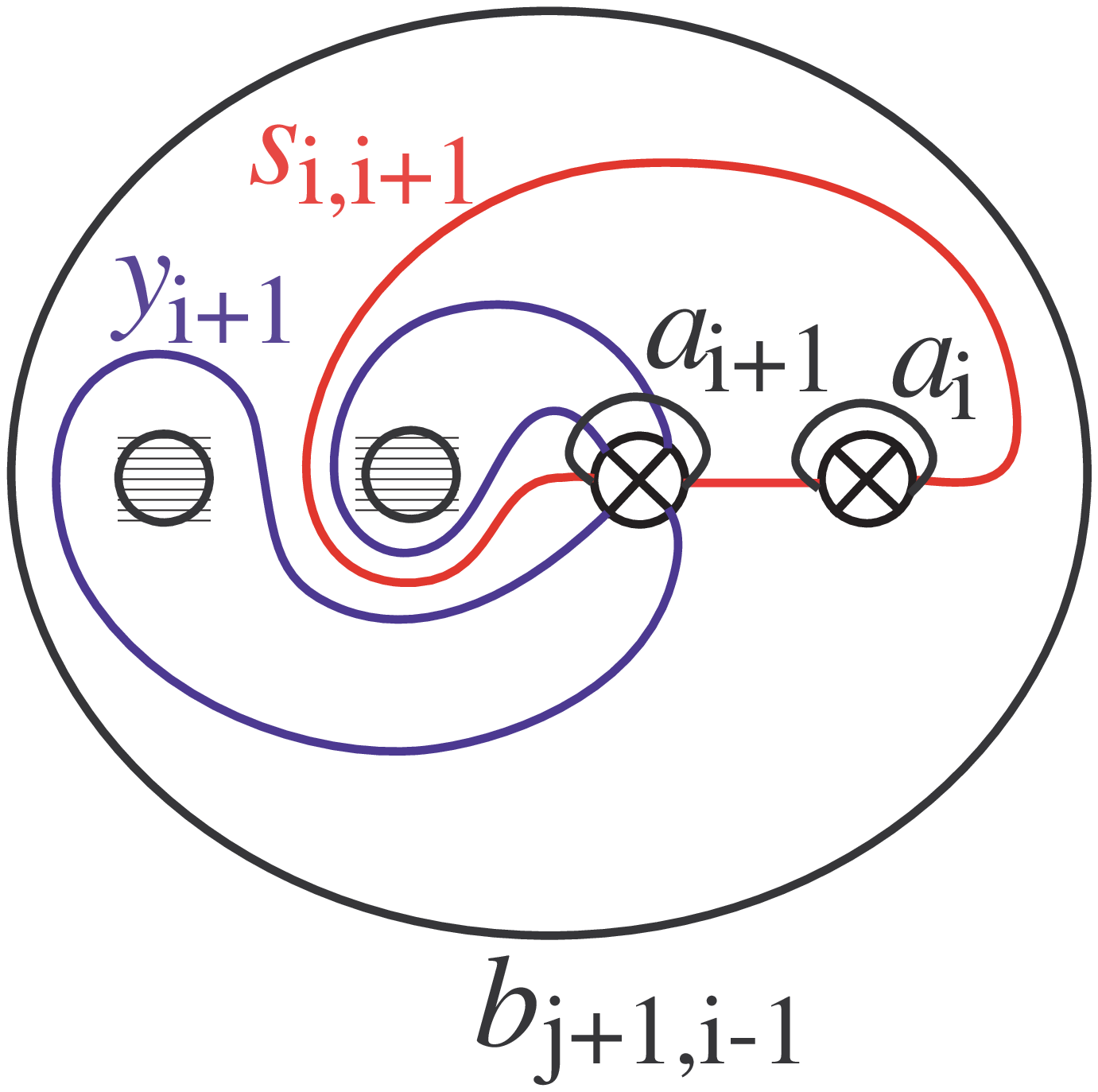}  \hspace{0.1cm} 
		\epsfxsize=1.58in \epsfbox{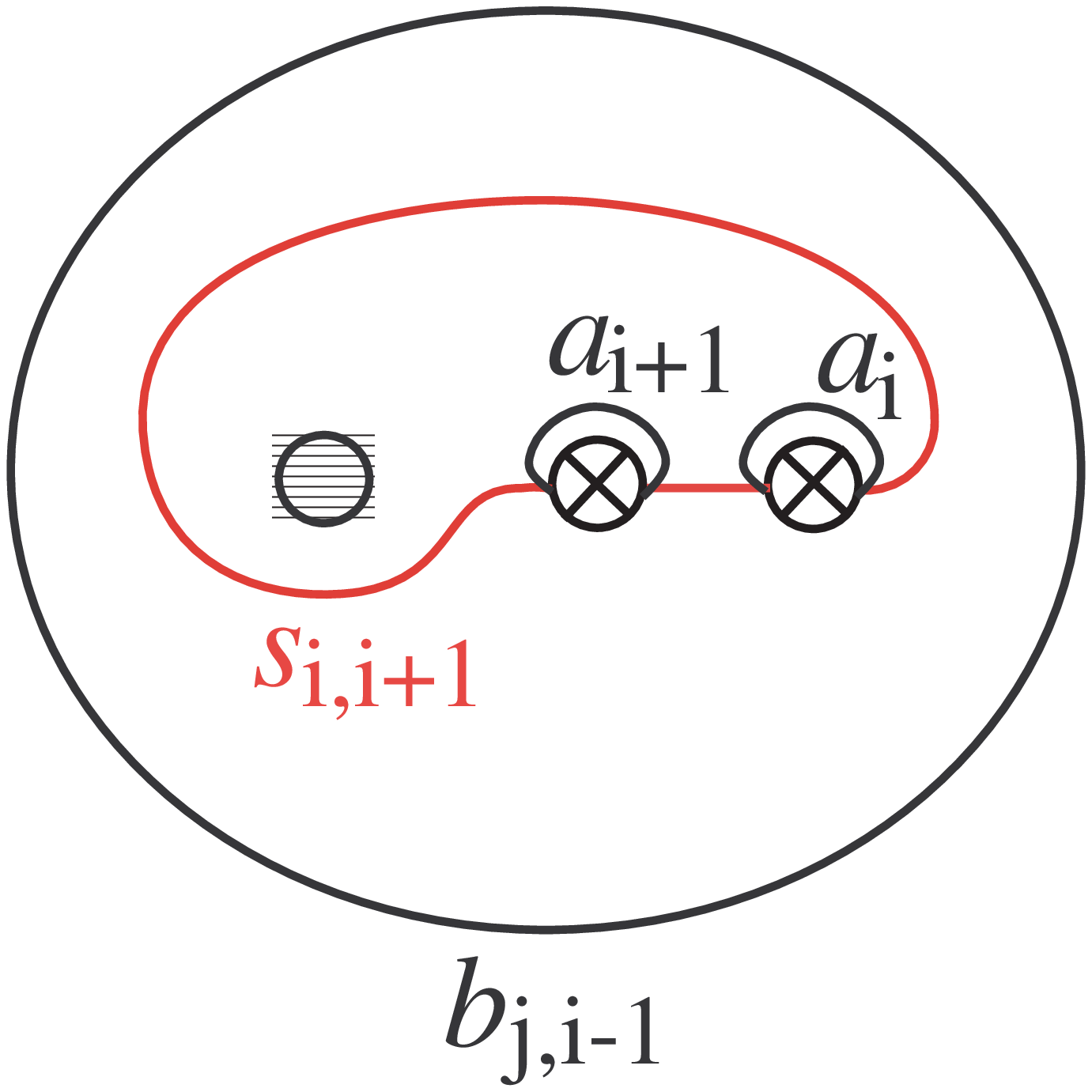} 
		
		(x)  \hspace{3.8cm}   (xi)  \hspace{3.8cm}   (xii) 	
		
		\caption{Curves in $\mathcal{C}_3$ and $\mathcal{C}_4$}\label{Fig-3-aa}
	\end{center}
\end{figure}

Let $d_{i,i+1}$ be the curve shown in Figure \ref{Fig-3-aa} (iii). In the Figure \ref{Fig-3-aa} (i) and (iii) we see a disk $D$ filled with stripes. Let $R$ be the boundary of $D$. By filling the disk $D$ with stripes we will mean that (i) $R$ bounds a  M\"{o}bius band which has $a_{i-1}$ as the core if $2 \leq i \leq g-1$, (ii) 
$R$ bounds a  M\"{o}bius band which has $a_{g}$ as the core if $i=1$ and $N$ is closed, and (iii) $R$ is isotopic to $t_{g+n}$ if $i=1$ and there is a boundary component of $N$. In Figure \ref{Fig-3-aa} we assume that the indices $i$ and $j$ are related as $|j-i|=2 \ (mod \ ( g+n))$.  
Let $u_i, v_i$ be the curves shown in Figure \ref{Fig-3-aa} (i) and (ii). Again in our notation filling disks with stripes will mean that the boundary of each disk either bounds a  M\"{o}bius band or is isotopic to a boundary component of $N$ depending on the position of $a_i$, when we consider the cyclic ordering of the $a_i$'s and $t_j$'s as shown in Figure \ref{fig-s11} (i). For $g+n \geq 5$ 
let $\mathcal{C}_3 = \{c_1, c_2,  \cdots, c_{g-1}, d_{1,2}, d_{2,3}, \cdots, d_{g-1,g}, u_1, u_2, \cdots, u_{g-1}, v_1, v_2, \cdots, v_{g-1}\}$. 

\begin{lemma} \label{C_3} If $g+n \geq 5$, then $\mathcal{C}_1 \cup \mathcal{C}_2 \cup \mathcal{C}_3$ is a finite superrigid set.\end{lemma}

\begin{proof} Let $\lambda: \mathcal{C}_1 \cup \mathcal{C}_2 \cup \mathcal{C}_3  \rightarrow \mathcal{C}(N)$ be a superinjective simplicial map. By Lemma \ref{C_2}, there exists a homeomorphism $h: N \rightarrow N$ such that $h([x]) = \lambda([x])$ for every $x$ in $\mathcal{C}_1 \cup \mathcal{C}_2$. We will show that $h([x]) = \lambda([x])$ for every $x$ in $\mathcal{C}_3$ as well. We will give the proof when $n \geq 2$. The proofs for the other cases when $n \leq 1$ are similar.

Claim 1: $h([c_i]) = \lambda([c_i])$ for all $i =1, 2, \cdots, g-1$. 
	
Proof of Claim 1: The curve $b_{2,g+n}$ is in $\mathcal{C}_1$ and it bounds a Klein bottle with one hole containing $c_1$. We complete $a_1, a_2, b_{2,g+n}$ to a top dimensional pants decomposition $P$ using the curves in $\mathcal{C}_1 \cup \mathcal{C}_2$ on $N$. Since $\lambda$ is superinjective and $c_1$ intersects only $a_1, a_2$ nontrivially in $P$, $\lambda([c_1])$ has to have geometric intersection nonzero with only $\lambda([a_1]), \lambda([a_2])$ in $\lambda([P])$. Since $c_1$ is a 2-sided curve and $\lambda$ is superinjective we see that $\lambda([c_1])$ is the isotopy class of a 2-sided curve on $N$. In a Klein bottle with one hole, up to isotopy there exists a unique 2-sided curve. Since $h([x]) = \lambda([x])$ for every $x$ in $\mathcal{C}_1 \cup \mathcal{C}_2$, by using the above information we see that $h([c_1]) = \lambda([c_1])$. Similarly, $h([c_i]) = \lambda([c_i])$ for all $i =2, 3, \cdots, g-1$. 
 
Claim 2: $h([u_i]) = \lambda([u_i])$ for all $i =1, 2, \cdots, g-1$.  

Proof of Claim 2: We will show that $h([u_1]) = \lambda([u_1])$. We complete $r_1, a_1, a_2, b_{2,g+n-1}$ to a top dimensional pants decomposition $P$ using the curves in $\mathcal{C}_1 \cup \mathcal{C}_2$ on $N$ such that $u_1$ is disjoint from all the curves in $P$ except for $a_1$. Since $\lambda$ is superinjective and $u_1$ intersects only $a_1$ nontrivially in $P$, $\lambda([u_1])$ has geometric intersection nonzero only with $\lambda([a_{1}])$ in $\lambda([P])$. Since $h([x]) = \lambda([x])$ for every $x$ in $\mathcal{C}_1 \cup \mathcal{C}_2$, and $\lambda([u_{1}]) \neq \lambda([a_{1}])$, we see that $h([u_1]) = \lambda([u_1])$ (since up to isotopy there are exactly two curves in the projective plane bounded by $a'_1$ and $b'_{2,g+n-1}$ which are disjoint  representatives of $\lambda([a_{1}]), \lambda([b_{2,g+n-1}])$ respectively). Similarly we get $h([u_{i}]) = \lambda([u_{i}])$ for all $i =2, 3, \cdots, g-2$.

Claim 3: $h([v_i]) = \lambda([v_i])$ for all $i =1, 2, \cdots, g-1$.

Proof of Claim 3: We will show that $h([v_1]) = \lambda([v_1])$. When we cut $N$ along the curve $b_{1,g+n-2}$ we get a nonorientable surface, $M_1$, of genus one with three boundary components, containing $a_{1,g+n-1}$. The curve $v_1$ is the unique curve up to isotopy disjoint from and nonisotopic to $u_1, a_{1,g+n-1}$ and $b_{1,g+n-2}$ and all the curves of $\mathcal{C}_1 \cup \mathcal{C}_2$ that are in the complement of $M_1$. Since $\lambda$ is superinjective and $h([x]) = \lambda([x])$ for all these curves, we see that $h([v_1]) = \lambda([v_1])$. When we cut $N$ along the curve $b_{2,g+n-1}$ we get a nonorientable surface, $M_2$, of genus two with two boundary components, containing $a_{2,g+n}$. The curve $v_2$ is the unique curve up to isotopy disjoint from and nonisotopic to $u_1, a_{2,g+n}, a_1$ and $b_{2,g+n-1}$ and all the curves of $\mathcal{C}_1 \cup \mathcal{C}_2$ that are in the complement of $M_2$. Since $\lambda$ is superinjective and $h([x]) = \lambda([x])$ for all these curves, we see that $h([v_2]) = \lambda([v_2])$.
When we cut $N$ along the curve $b_{3,g+n}$ we get a nonorientable surface, $M_3$, of genus three with one boundary component, containing $a_{3,1}$. 
The curve $v_3$ is the unique curve up to isotopy disjoint from and nonisotopic to $u_{2}, a_{3,1}, a_1, a_2$ and $b_{3,g+n}$ and all the curves of $\mathcal{C}_1 \cup \mathcal{C}_2$ that are in the complement of $M_3$. Since $\lambda$ is superinjective and $h([x]) = \lambda([x])$ for all these curves, we see that $h([v_3]) = \lambda([v_3])$. Similarly, we have $h([v_i]) = \lambda([v_i])$ for all $i =4, 5, \cdots, g-1$.

Claim 4: $h([d_{i,i+1}]) = \lambda([d_{i,i+1}])$ for all $i =1, 2, \cdots, g-1$.

Proof of Claim 4: When we cut $N$ along the curve $b_{2,g+n-1}$ we get a nonorientable surface, $M_4$, of genus two with two boundary components, which contains $c_1$. The curve $d_{1,2}$ is the unique curve up to isotopy disjoint from and nonisotopic to $c_1, b_{2,g+n-1}$ and $v_1$ and all the curves of $\mathcal{C}_1 \cup \mathcal{C}_2$ that are in the complement of $M_4$. Since $\lambda$ is superinjective and $h([x]) = \lambda([x])$ for all these curves, we see that $h([d_{1,2}]) = \lambda(d_{1,2})$. When we cut $N$ along the curve $b_{3,g+n}$ we get a nonorientable surface, $M_5$, of genus three with one boundary component, which contains $c_2$. The curve $d_{2,3}$ is the unique curve up to isotopy disjoint from and nonisotopic to $a_1, c_2, b_{3,g+n}$ and $v_2$ and all the curves of $\mathcal{C}_1 \cup \mathcal{C}_2$ that are in the complement of $M_5$. Since $\lambda$ is superinjective and $h([x]) = \lambda([x])$ for all these curves, we see that $h([d_{2,3}]) = \lambda(d_{2,3})$. Similarly, we have $h([d_{i,i+1}]) = \lambda([d_{i,i+1}])$ for all $i =3, 4, \cdots, g-1$.

We proved that $h([x]) = \lambda([x])$ for every $x$ in $\mathcal{C}_1 \cup \mathcal{C}_2 \cup \mathcal{C}_3$. Hence, $\mathcal{C}_1 \cup \mathcal{C}_2 \cup \mathcal{C}_3$ is a finite superrigid set.\end{proof}\\
 
Let $k_i, p_i, e_{i,i+1}$ be the curves shown in Figure \ref{Fig-3-aa} (iv), (v) and (vi). Let $l_{i+1}, r_{i,i+1}$ be the curves shown in Figure \ref{Fig-3-aa} (viii), (ix). Let $y_{i+1}, s_{i,i+1}$ be the curves shown in Figure \ref{Fig-3-aa} (xi), (xii). For $g+n \geq 5$,   
let $\mathcal{C}_4 = \{k_1, k_2, \cdots, k_{g-1}, p_1, p_2, \cdots, p_{g-1}, e_{1,2}, e_{2,3}, \cdots, e_{g-1,g},$ $l_2, l_3, \cdots, l_g, $ $r_{1,2}, r_{2,3}, \cdots,$ $r_{g-1,g}, y_2, y_3, \cdots, y_g, s_{1,2}, s_{2,3}, \cdots, s_{g-1,g} \}$.

\begin{lemma} \label{C_4}  If $g+n \geq 5$, then $\bigcup_{i=1} ^4 \mathcal{C}_i$ is a finite superrigid set.\end{lemma}

\begin{proof}The proof is similar to the proof of Lemma \ref{C_3}.\end{proof}\\

\begin{figure}
	\begin{center}
		\epsfxsize=2.5in \epsfbox{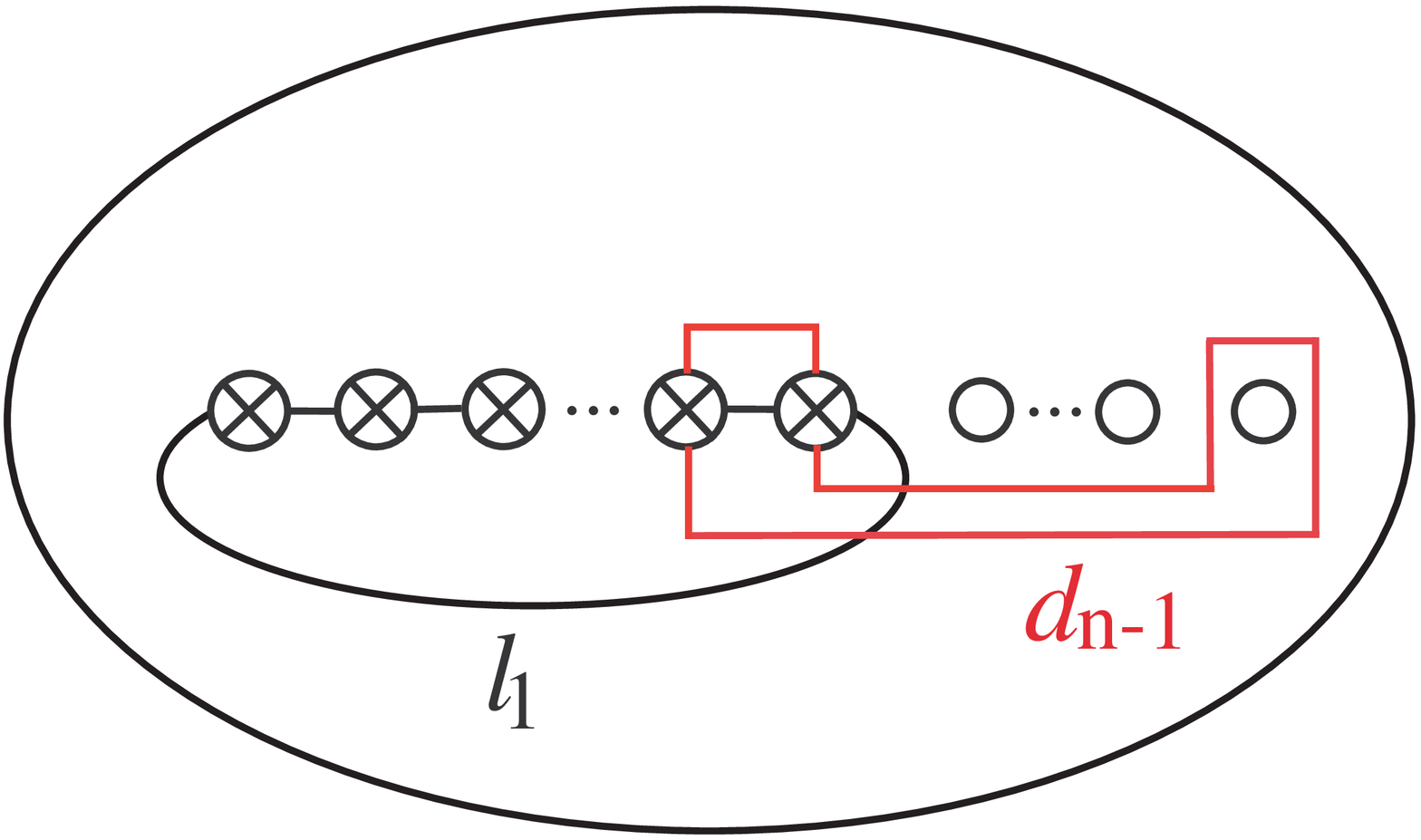} \hspace{0.1cm} 	\epsfxsize=2.5in \epsfbox{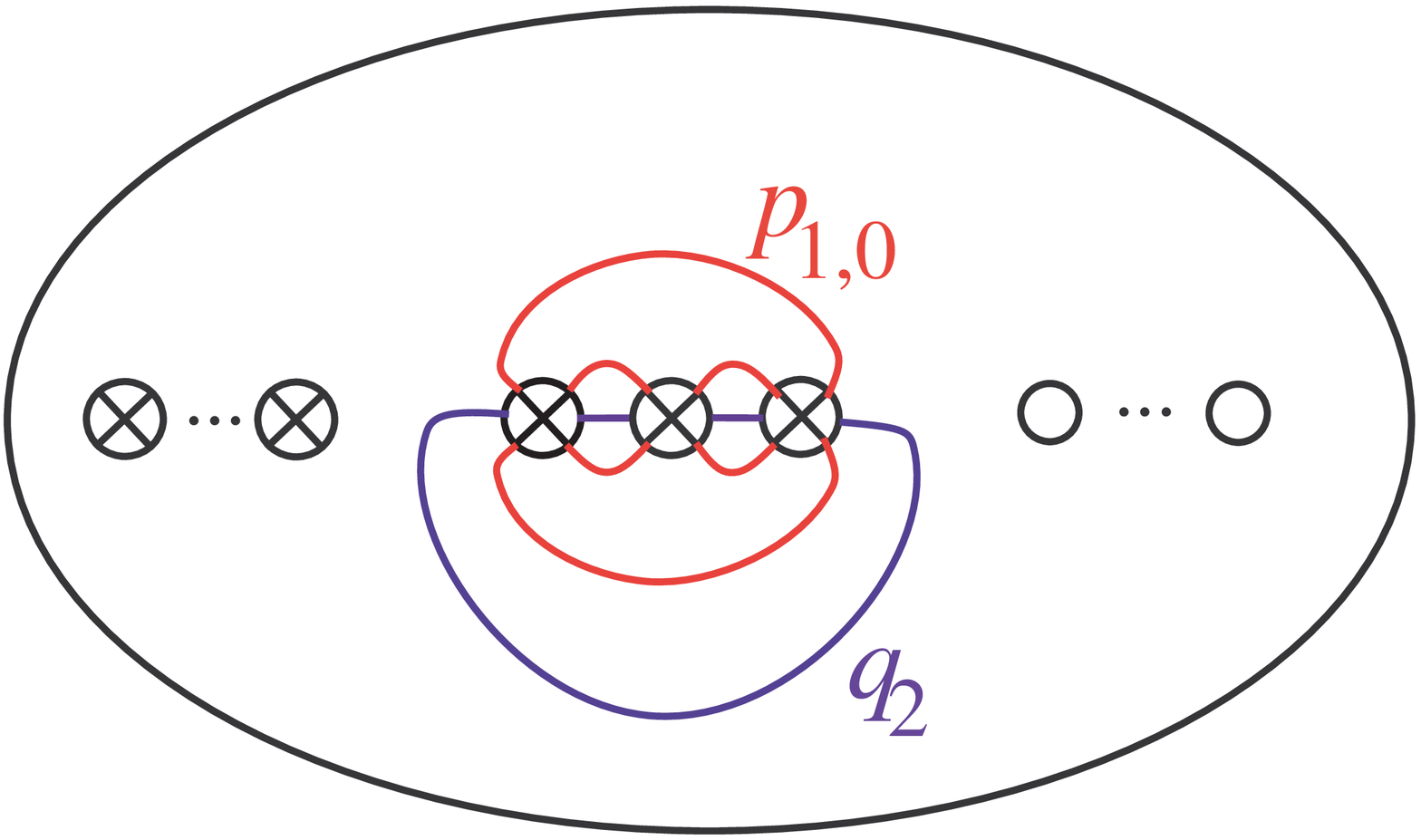} 
		
		(i)  \hspace{6cm}   (ii) 
		
		\epsfxsize=2.5in \epsfbox{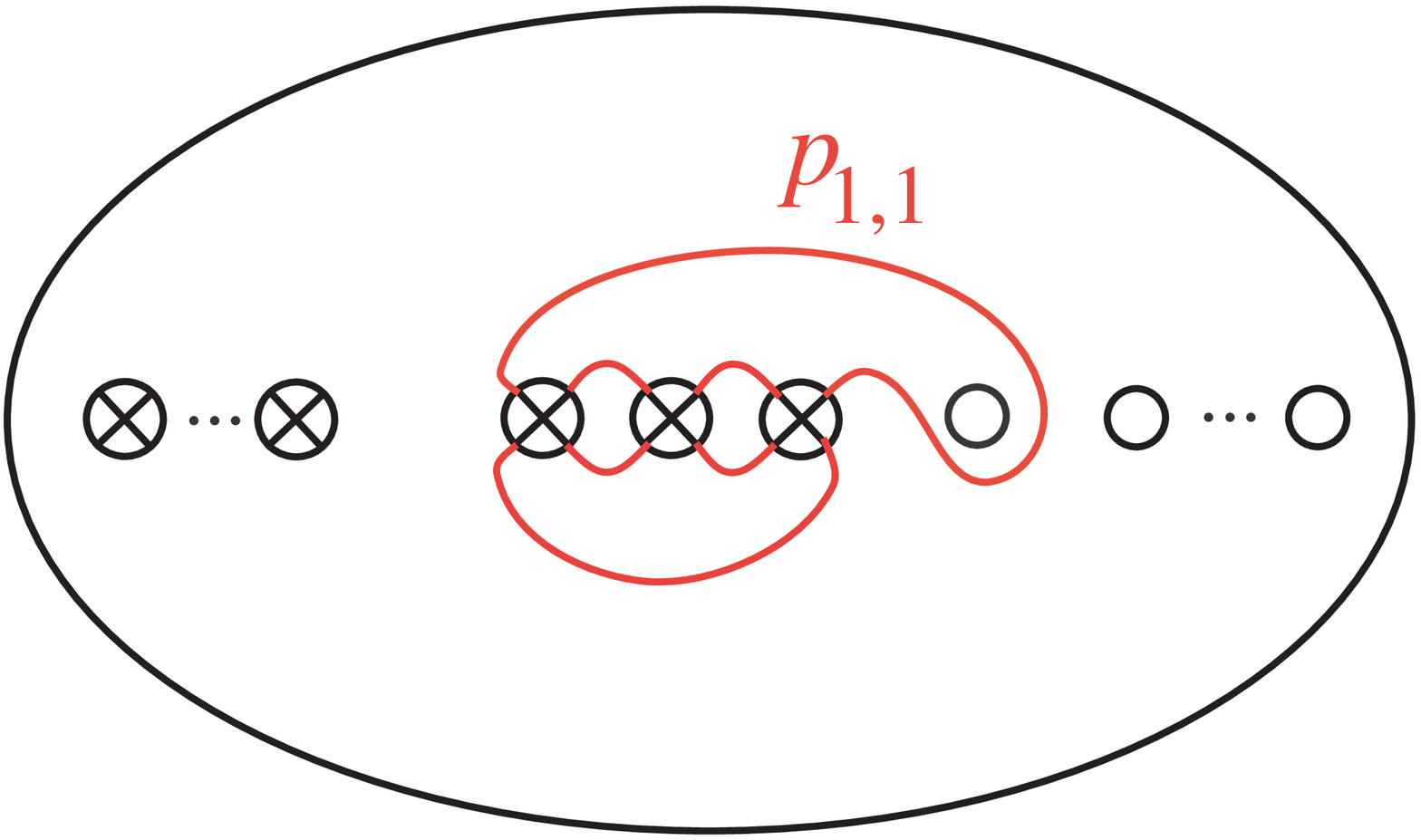} \hspace{0.1cm} 	\epsfxsize=2.5in \epsfbox{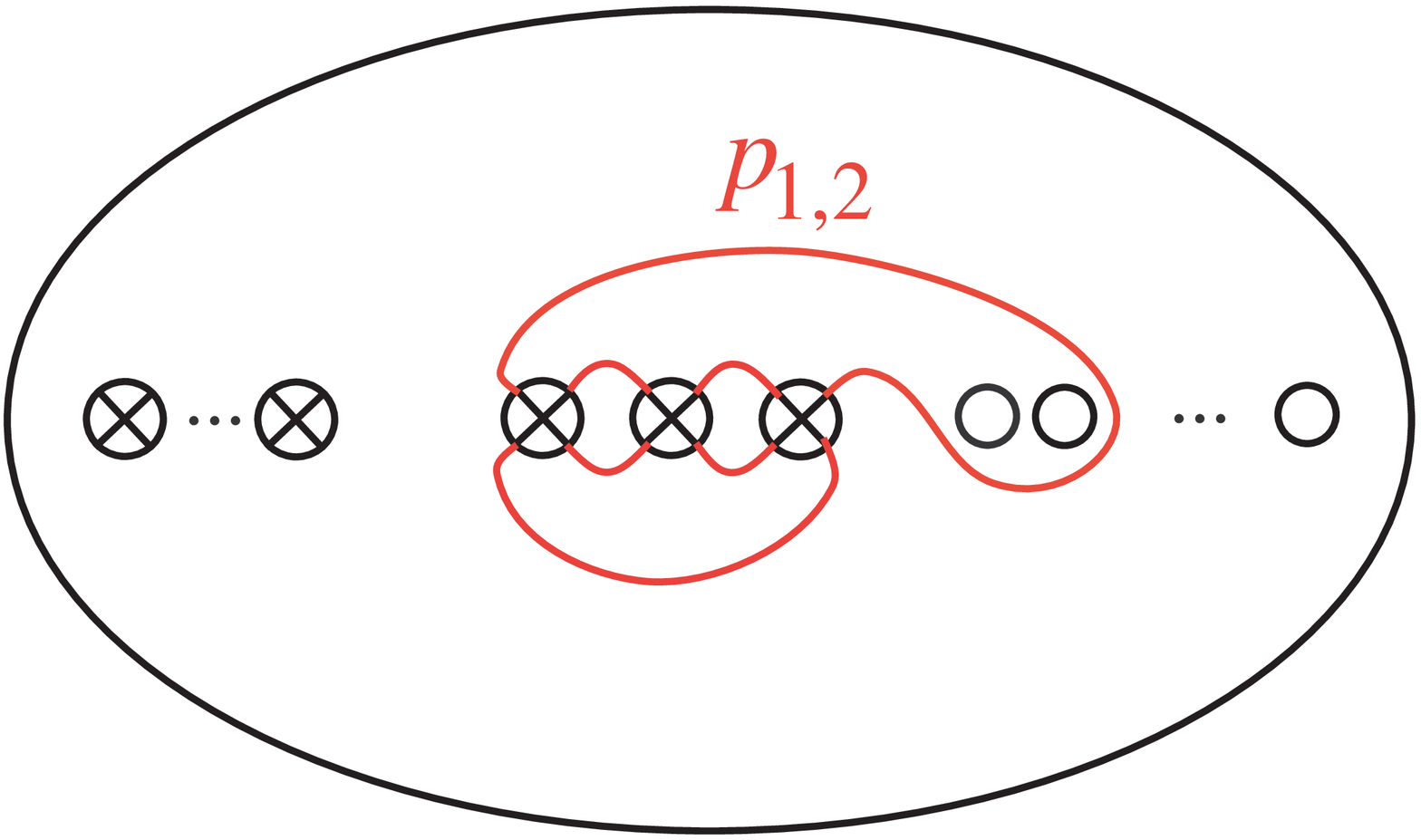}

		(iii)  \hspace{6cm}   (iv) 
		
		\epsfxsize=2.5in \epsfbox{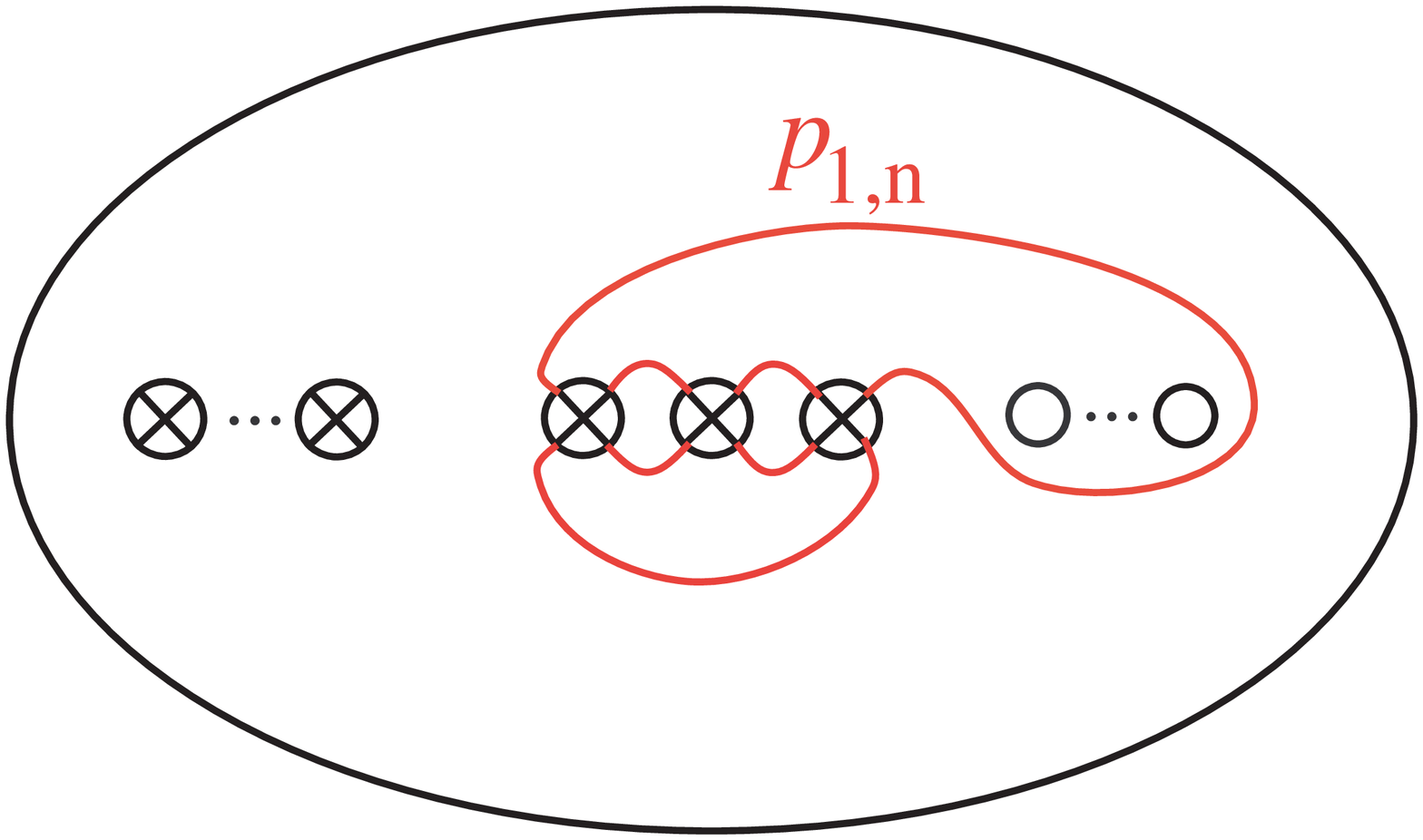} \hspace{0.1cm} 	\epsfxsize=2.5in \epsfbox{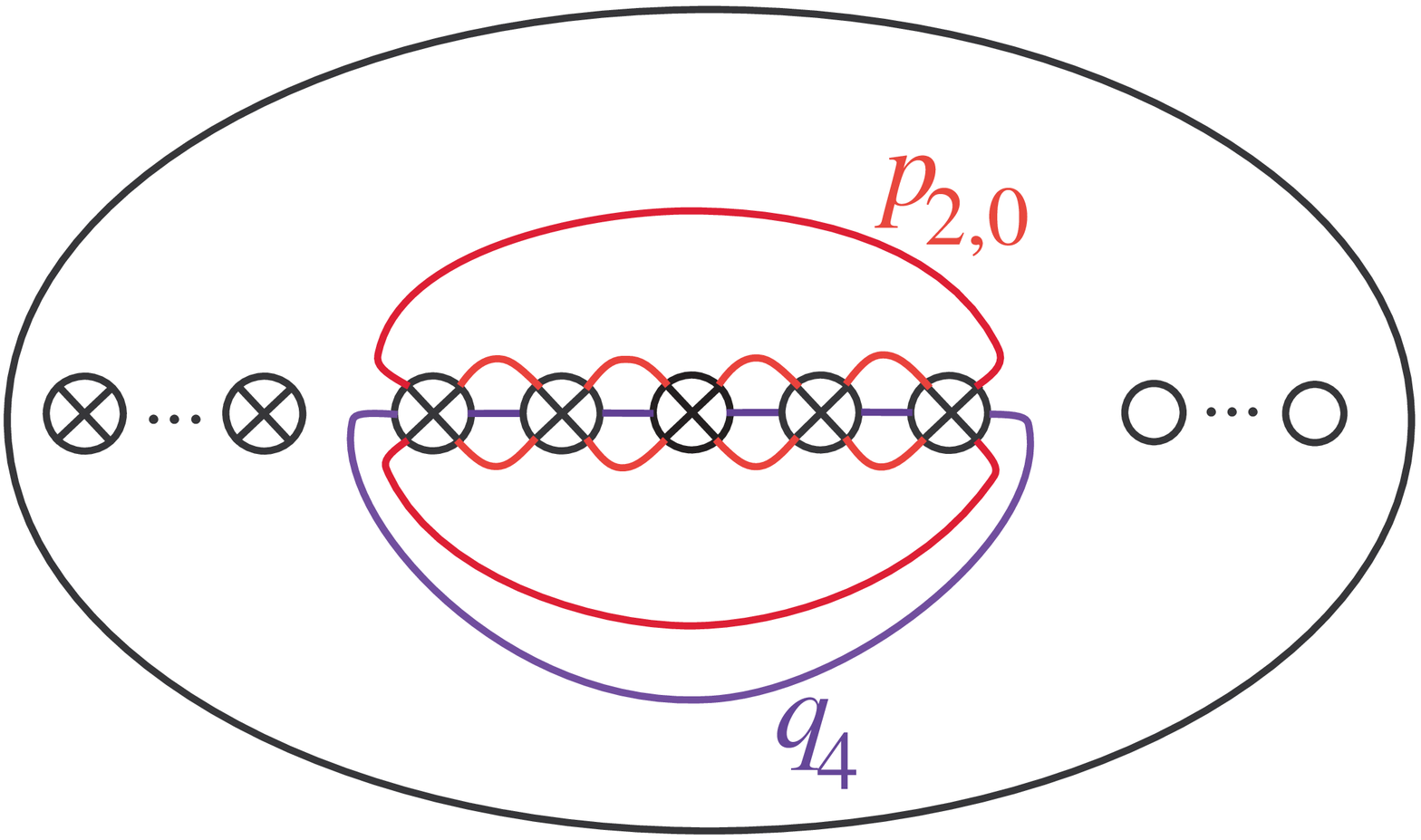} 
		
		(v)  \hspace{6cm}   (vi)
		
			\epsfxsize=2.5in \epsfbox{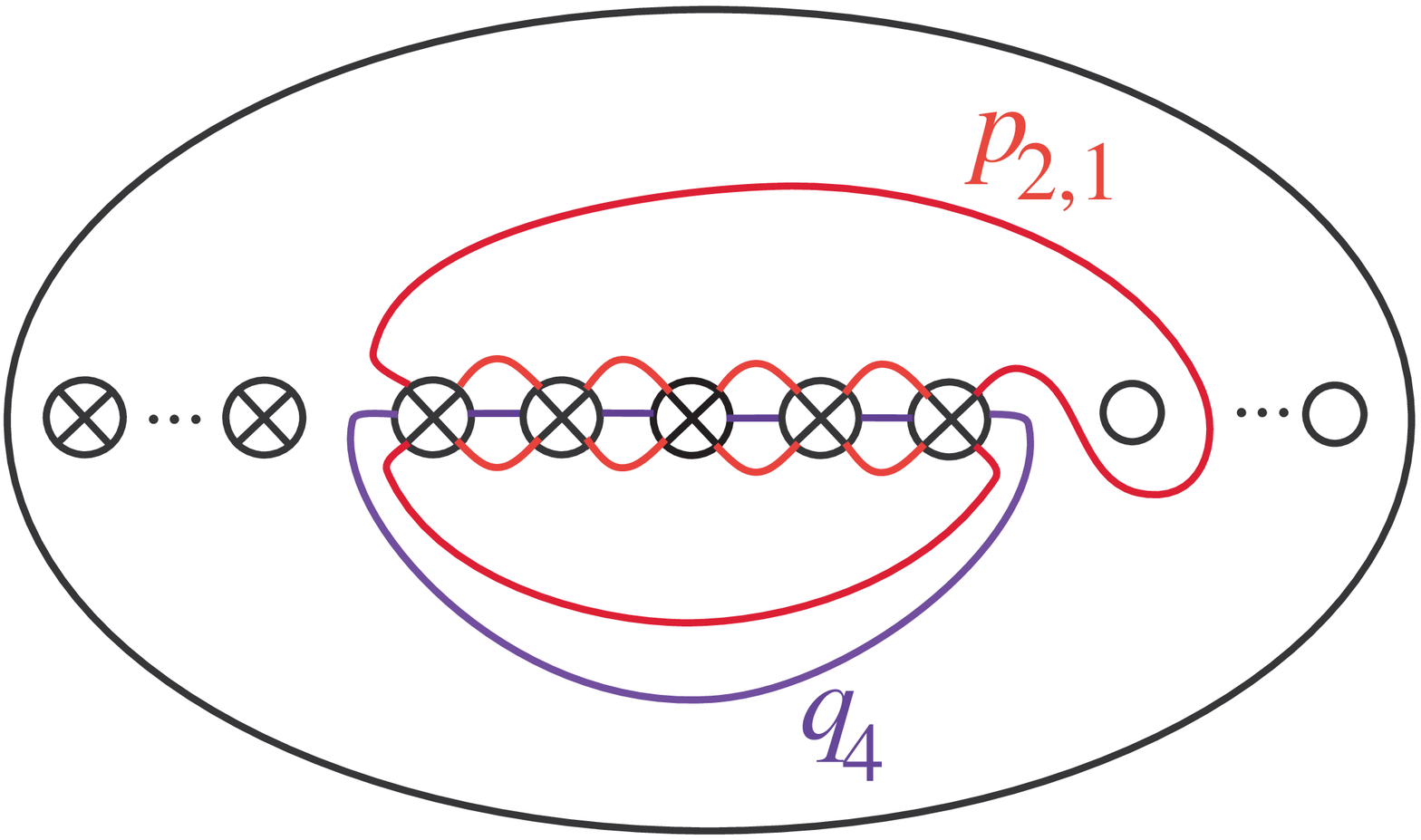} \hspace{0.1cm} 	\epsfxsize=2.5in \epsfbox{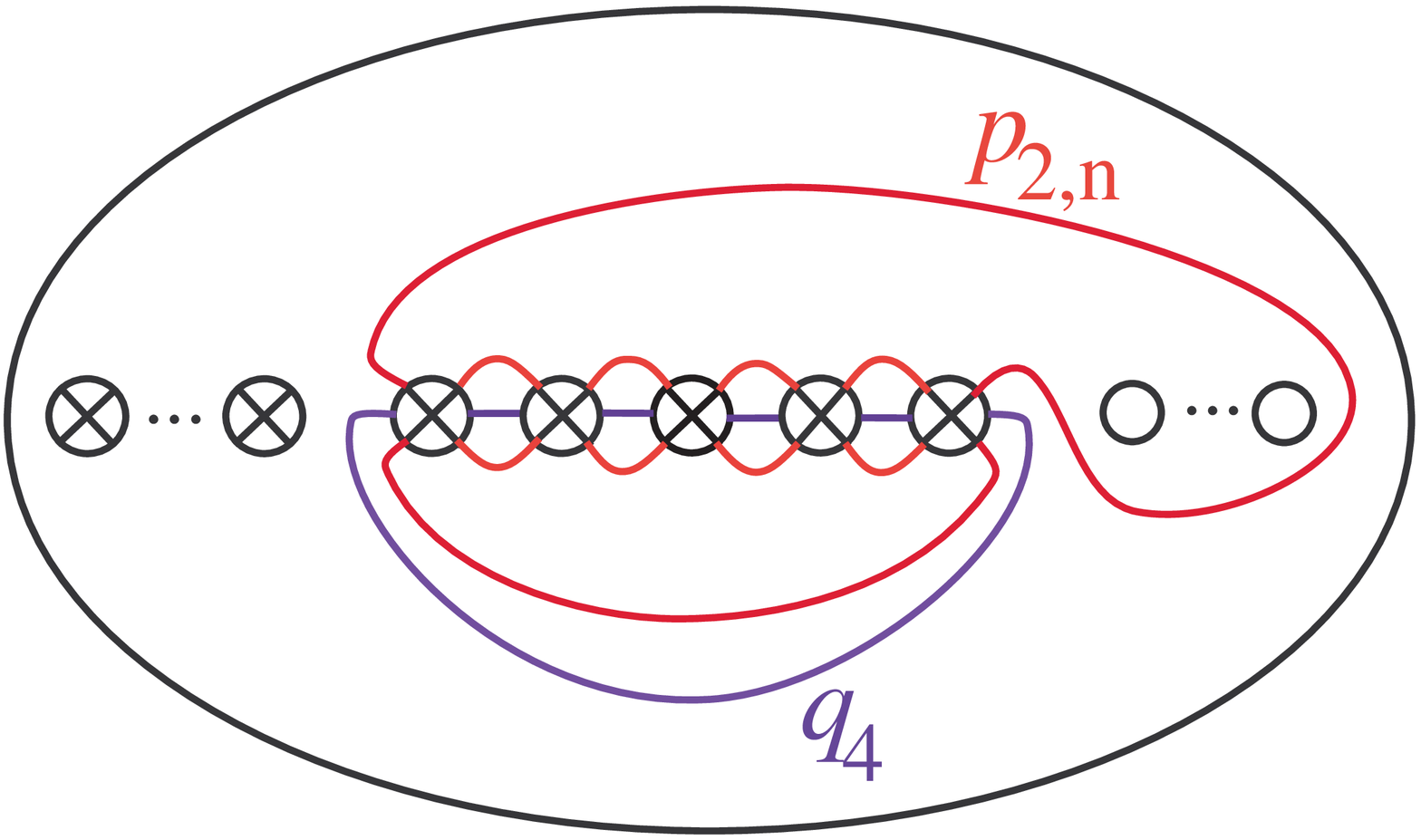} 
			
	(vii)  \hspace{6cm}   (viii)
			
		\caption{Curves in  $\mathcal{C}_5$ and $\mathcal{C}_6$}
		\label{Fig-7}
	\end{center}
\end{figure}

\begin{figure}
	\begin{center}
		\epsfxsize=2.5in \epsfbox{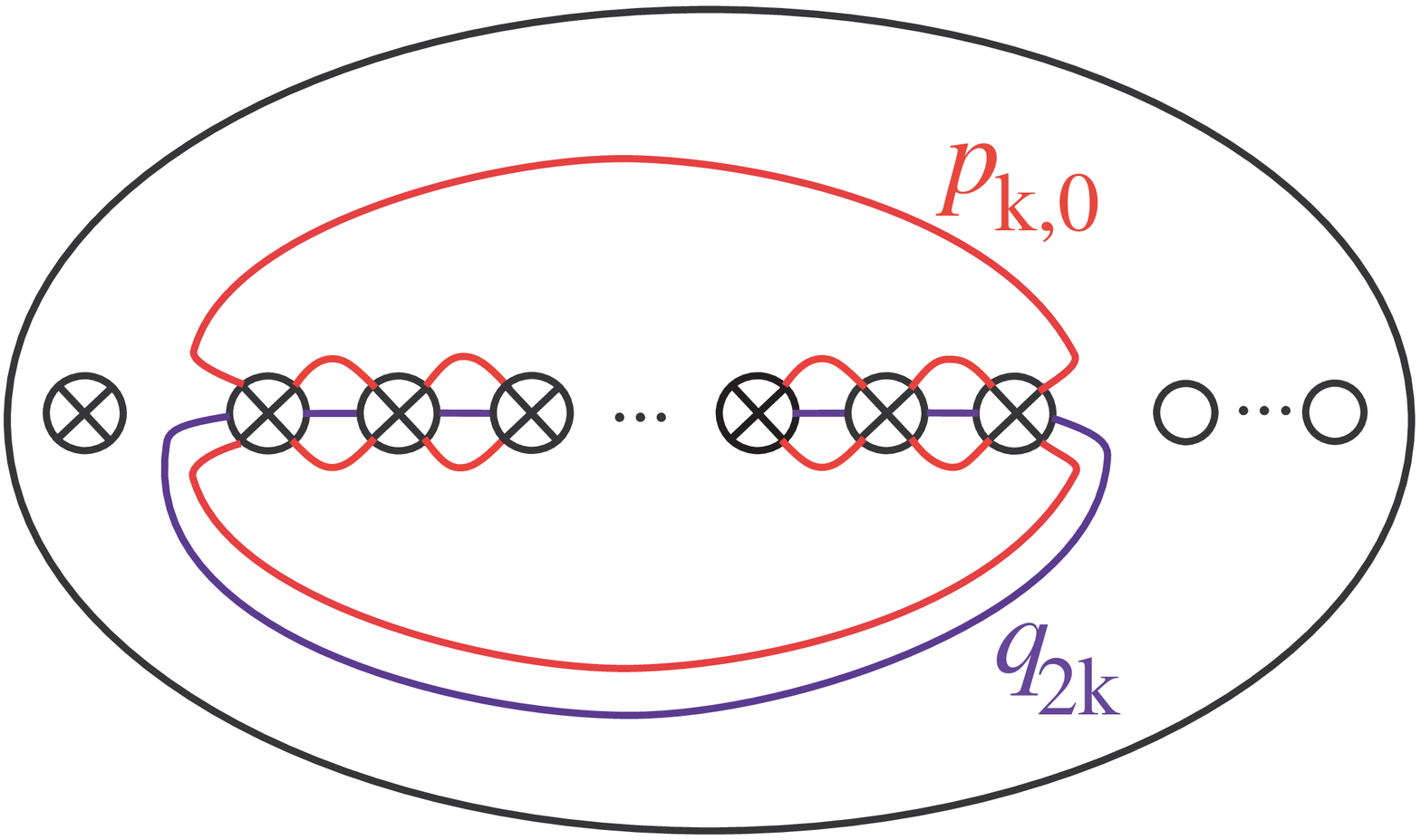} \hspace{0.1cm} 	\epsfxsize=2.5in \epsfbox{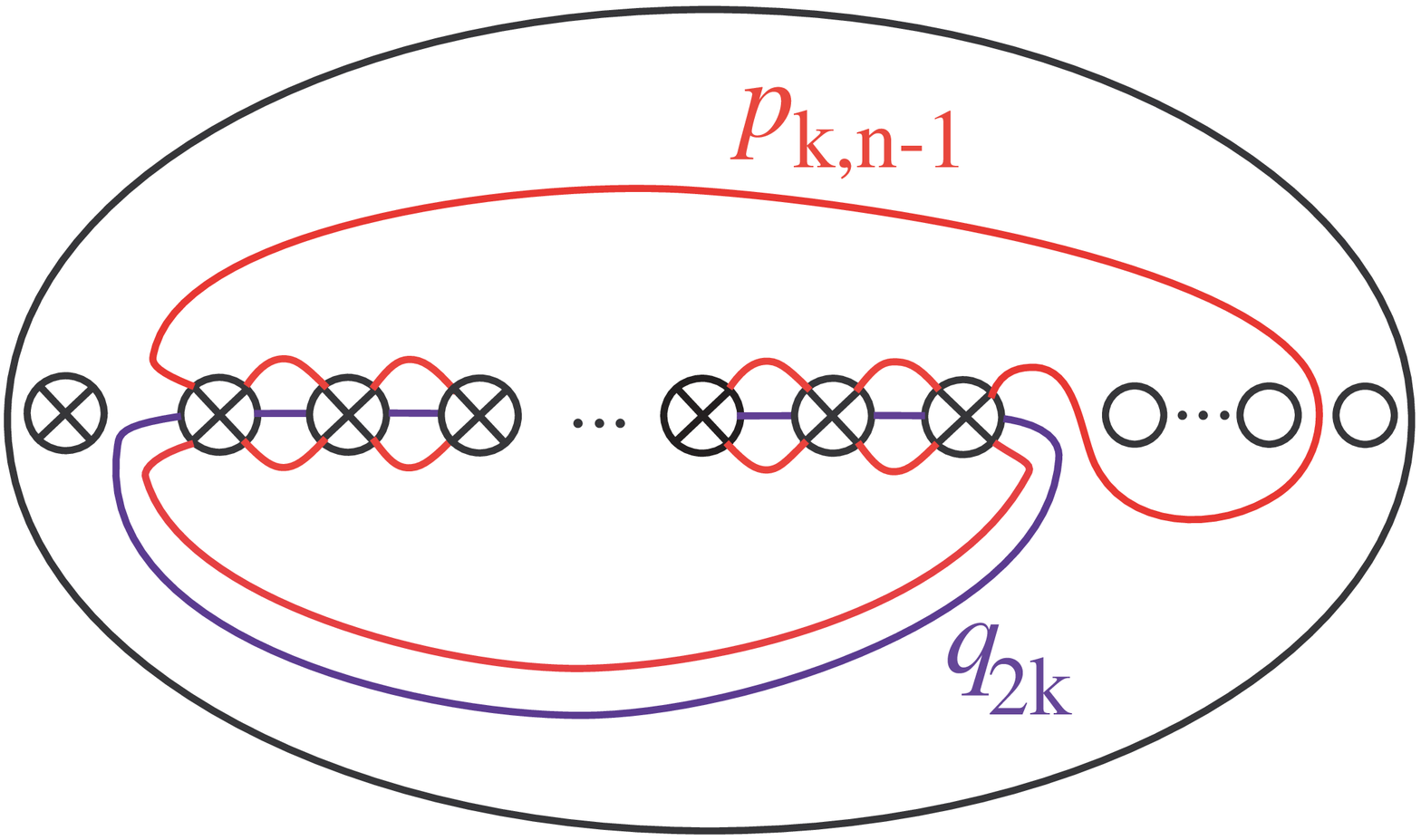} 
		
		(i)  \hspace{6cm}   (ii) 
		
		\epsfxsize=2.5in \epsfbox{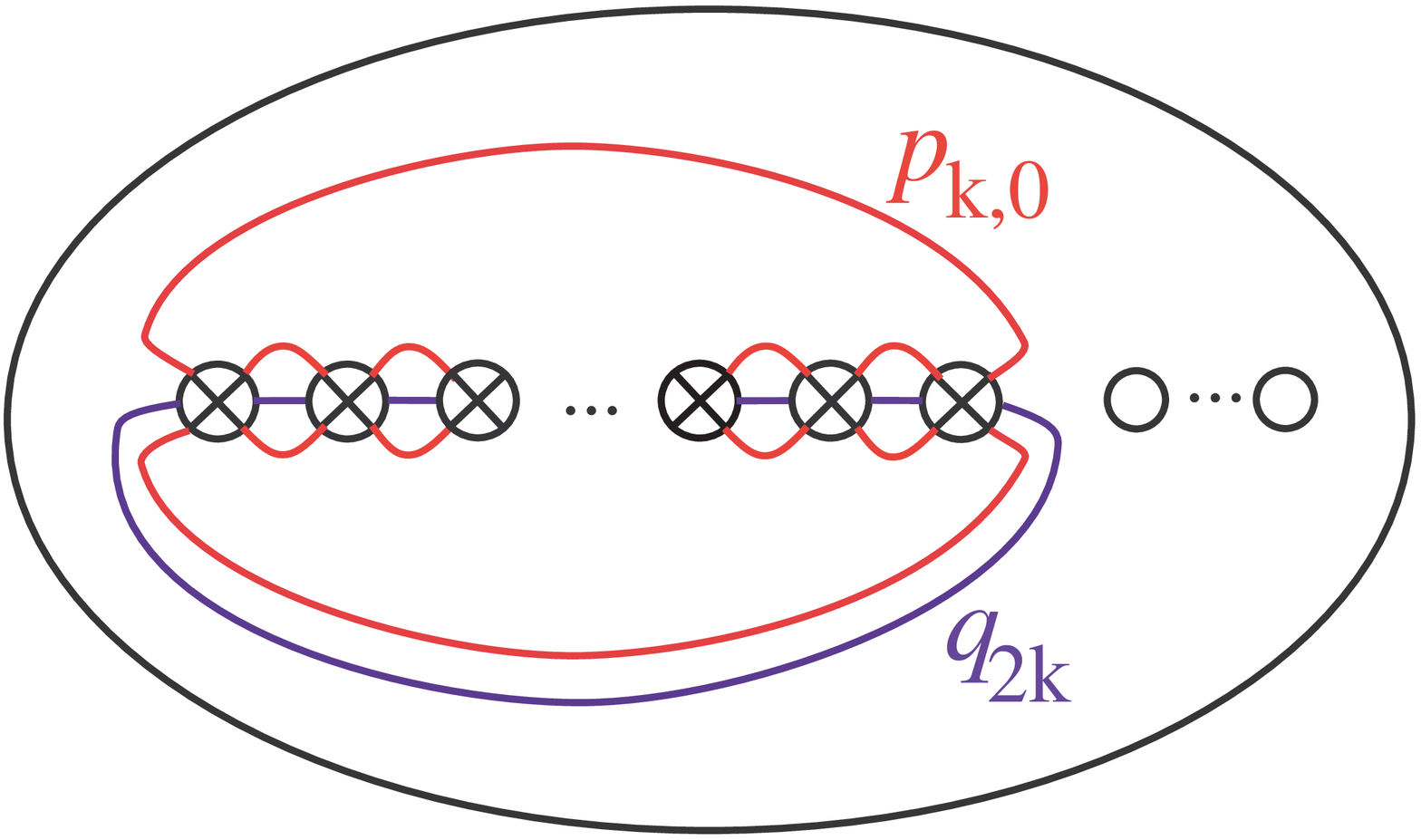} \hspace{0.1cm} 	\epsfxsize=2.5in \epsfbox{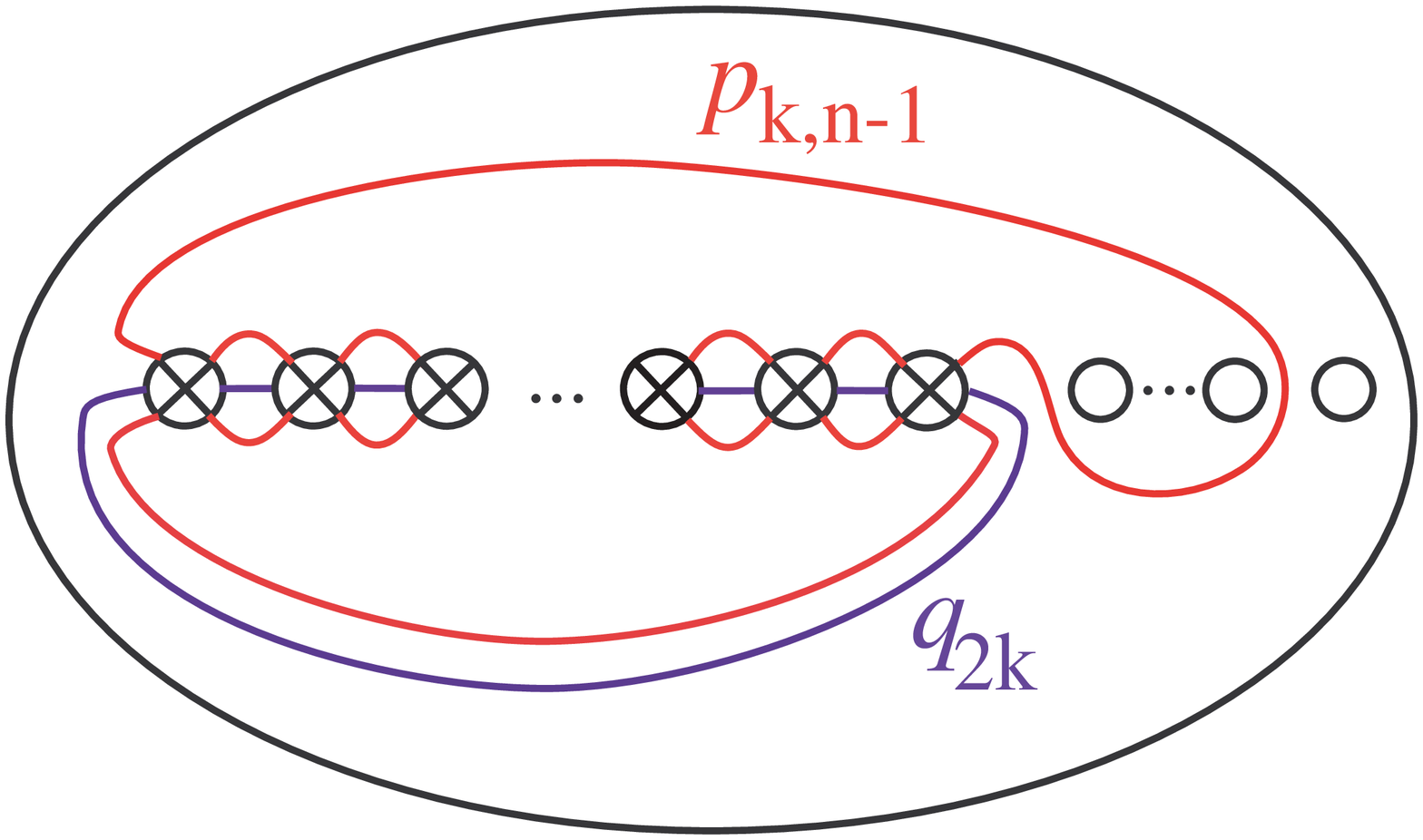} 
		
		(iii)  \hspace{6cm}   (iv) 
 
 	\epsfxsize=2.5in \epsfbox{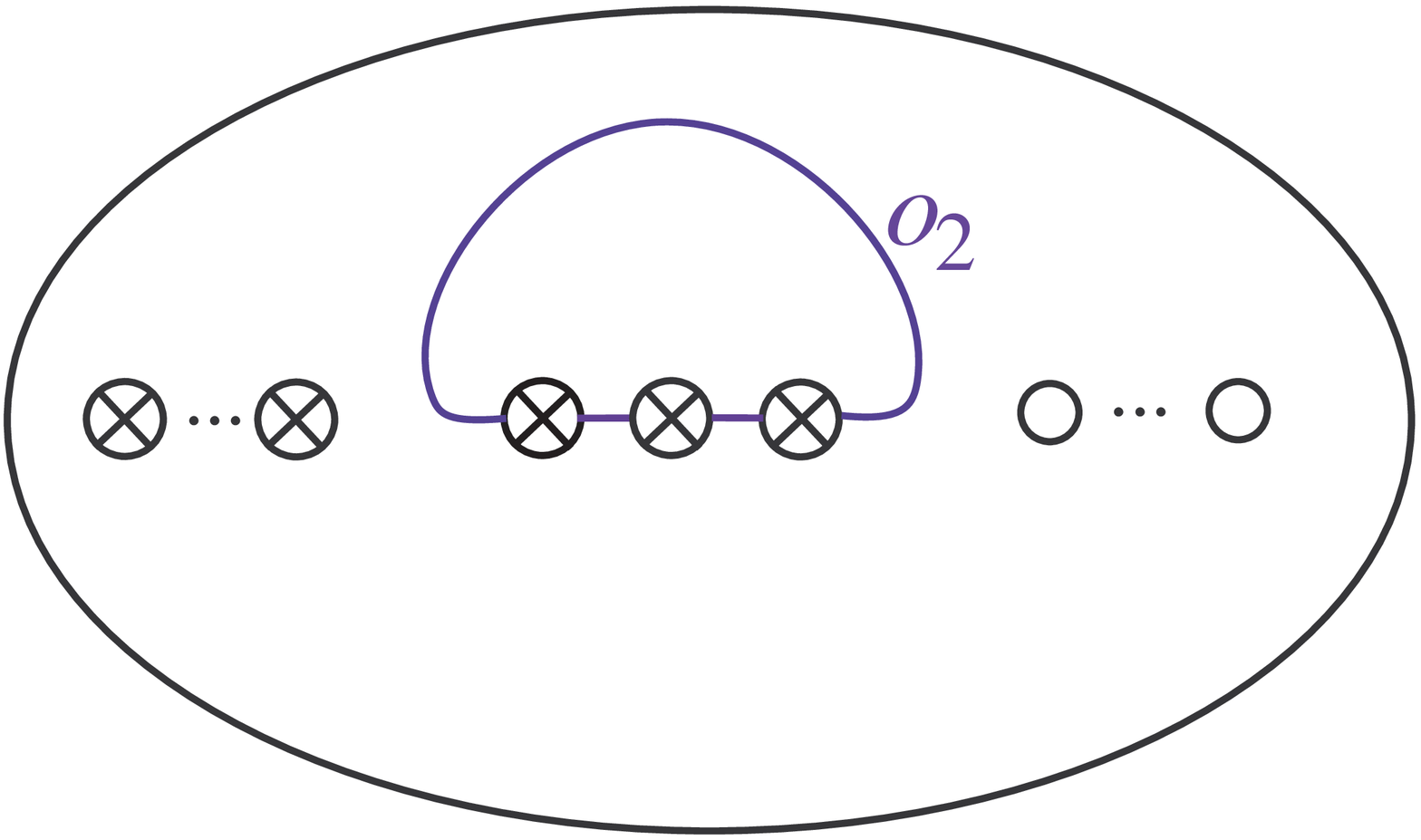} \hspace{0.1cm} 	\epsfxsize=2.5in \epsfbox{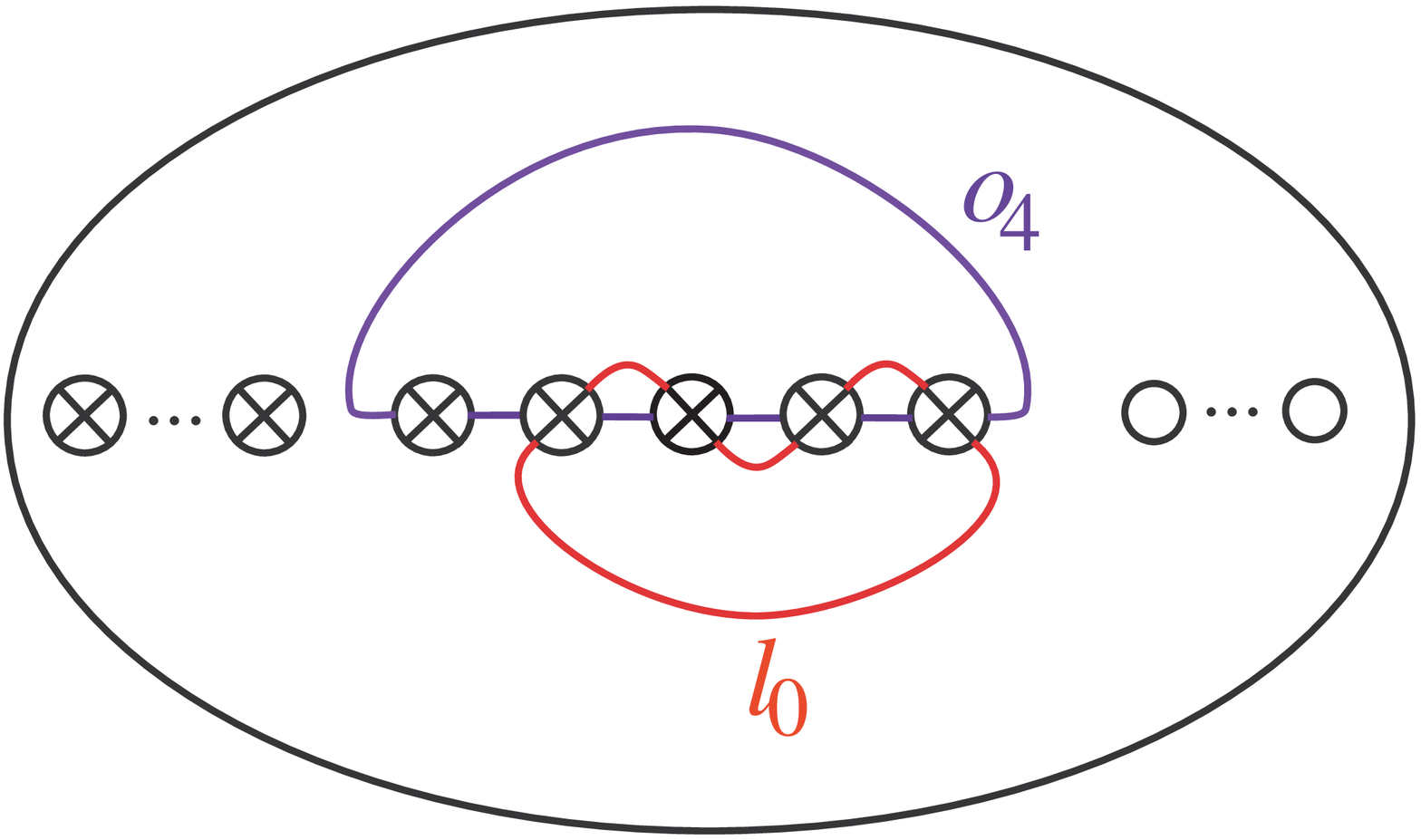} 
 
 (v)  \hspace{6cm}   (vi) 
 
 \epsfxsize=2.5in \epsfbox{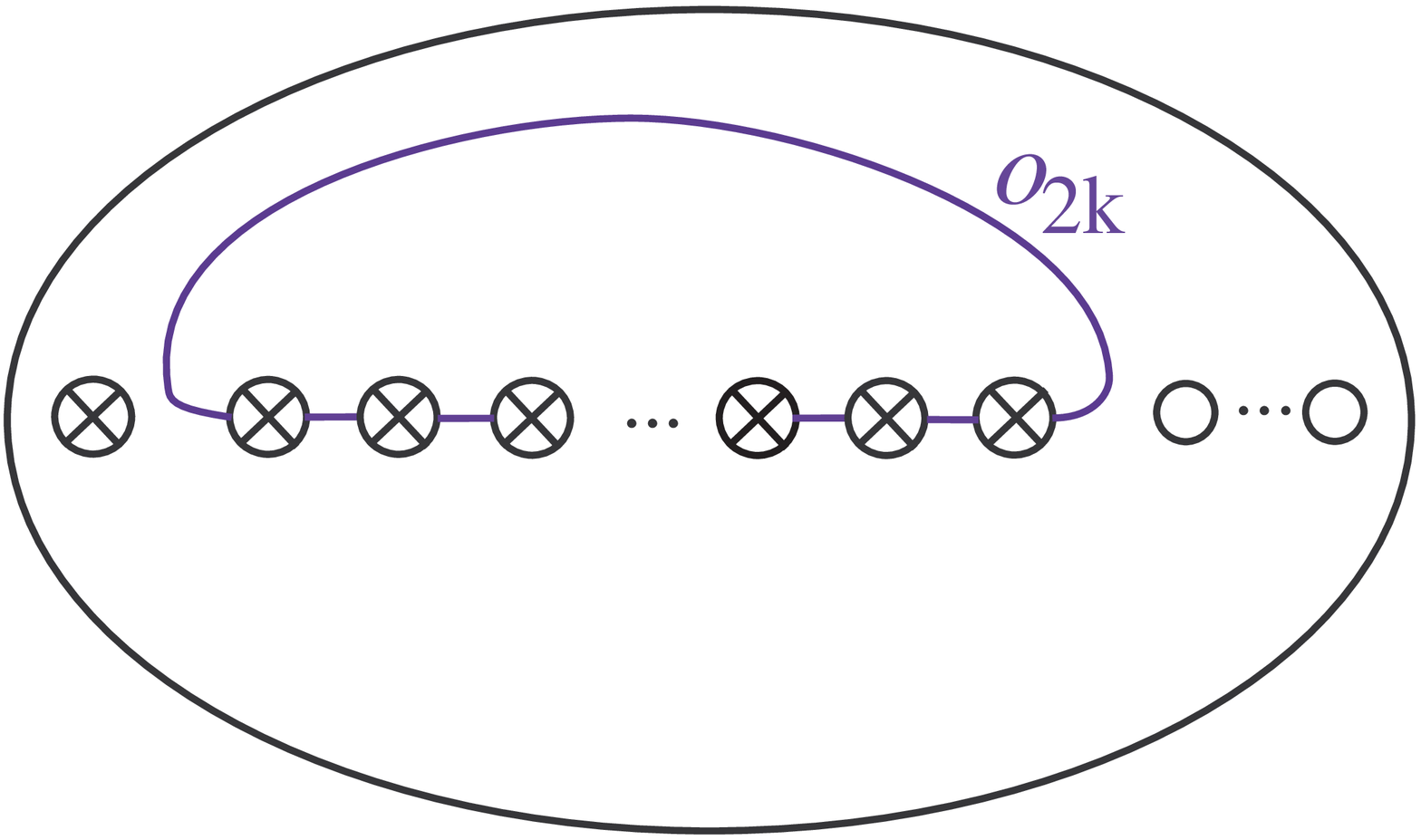} \hspace{0.1cm} 	\epsfxsize=2.5in \epsfbox{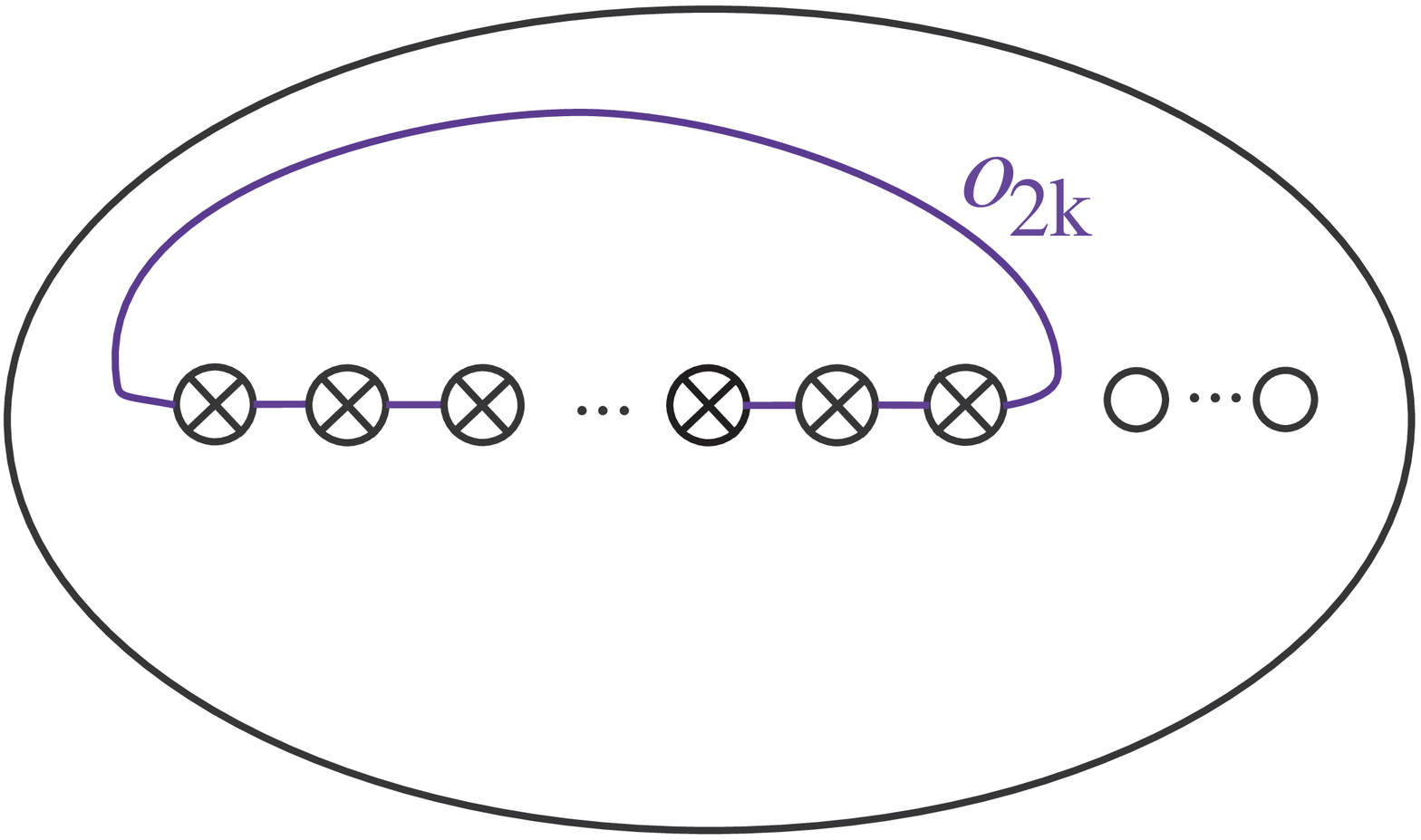} 
 
 (vii)  \hspace{6cm}   (viii) 
 	\caption{Curves in $\mathcal{C}_6$; (i), (ii) , (vii) for even genus surface; (iii), (iv), (viii) for odd genus surface }
 \label{Fig-8}
\end{center}
\end{figure}
 
Consider the curves $l_1$ and $d_{n-1}$ given in Figure \ref{Fig-7} (i) and $l_0$ given in Figure \ref{Fig-8} (vi). Let $\mathcal{C}_5 = \{l_0, l_1\}$ when $n \leq 1$ and let $\mathcal{C}_5 = \{l_1, d_{n-1}\}$ when $n \geq 2$. Note that $l_1$ has orientable complement on $N$. 
 
\begin{lemma} \label{C_5} If $g+n \geq 5$, then $\bigcup_{i=1} ^5\mathcal{C}_i$ is a finite superrigid set.\end{lemma}
 
\begin{proof} Let $\lambda : \bigcup_{i=1} ^5 \mathcal{C}_i \rightarrow \mathcal{C}(N)$ be a superinjective simplicial map. By Lemma \ref{C_4}, there exists a homeomorphism $h: N \rightarrow N$ such that $h([x]) = \lambda([x])$ for every $x$ in $\bigcup_{i=1} ^4 \mathcal{C}_i$. We will first show that $h([l_1]) = \lambda([l_1])$. (i) Suppose $N$ is a closed surface. Then $l_1$ is the unique nontrivial curve up to isotopy disjoint from each of $c_1, c_2, \cdots, c_{g-1}$.  Since $\lambda$ is superinjective and preserves this property, we have  $h([l_1]) = \lambda([l_1])$. (ii) Suppose $N$ has boundary. The curve $l_1$ is the unique nontrivial curve up to isotopy disjoint from each of $c_1, c_2, \cdots, c_{g-1}, d_{1,2},  b_{g,g+n}, b_{g,g+2}, b_{g+1,g+3}, \cdots, b_{g+n-2,g+n}$. Since $\lambda$ is superinjective and preserves this property, we have $h([l_1]) = \lambda([l_1])$. 
	
Suppose $n \leq 1$. Then we have $g \geq 4$ since $g+n \geq 5$. The curve $l_0$ is the unique nontrivial curve up to isotopy disjoint from and nonisotopic to each of $c_1, c_2, c_3, d_{1,2}, b_{4,g+n}$ and all the curves of $\bigcup_{i=1} ^4 \mathcal{C}_i$ on the other side of $b_{4,g+n}$. Since $\lambda$ is superinjective and preserves these properties, we have $h([l_0]) = \lambda([l_0])$.	

Suppose $n \geq 2$. The curve $d_{n-1}$ is the unique curve up to isotopy disjoint from and nonisotopic to $b_{2,4}, b_{3,5}, \cdots, b_{g-2,g}, r_{g+1}, b_{g+1,g+3}, b_{g+2,g+4}, \cdots,  b_{g+n-2,g+n}, a_3, a_4, \cdots, a_g,$ 
$c_1,d_{1,2}$. Since $h([x]) = \lambda([x])$ for all these curves and $\lambda$ preserves these properties we have $h([d_{n-1}]) = \lambda([d_{n-1}])$.
So, $h([x]) = \lambda([x])$ for every $x$ in $\bigcup_{i=1} ^5 \mathcal{C}_i$. Hence, $\bigcup_{i=1} ^5 \mathcal{C}_i$ is a finite superrigid set.\end{proof}\\

For genus $g \geq 4$, $g=2k$ for some $k \in \mathbb{Z}^+$, and $g+n \geq 5$, we define $\mathcal{C}_6 = \{p_{i,j}, q_t, o_t: i= 1, 2, \cdots k, \ j= 0, 1, 2, \cdots, n, \ t=2, 4, \cdots, 2k\}$ where the curves $p_{i,j}, q_t, o_t$ are as shown in Figure \ref{Fig-7} and Figure \ref{Fig-8}. Note that $p_{i,j}$ is a separating curve that separates the surface into two pieces and one of them is an orientable surface of genus $i$ and $j+1$ boundary components. The boundary components of this piece are $p_{i,j}$ and $j$ boundary components of $N$. If genus $g = 2$ and $g+n \geq 5$, we let $\mathcal{C}_6 = \emptyset$.

For genus $g \geq 3$, $g=2k+1$ for some $k \in \mathbb{Z}^+$, and $g+n \geq 5$, we define $\mathcal{C}_6 = \{p_{i,j}, q_t, o_t: i= 1, 2, \cdots k, \ j= 0, 1, 2, \cdots, n-1, \ t=2, 4, \cdots, 2k\}$ where the curves $p_{i,j}, q_t, o_t$ are as shown in Figure \ref{Fig-7} and Figure \ref{Fig-8}. The curve $p_{i,j}$ is a separating curve that separates the surface into two pieces and one of them is an orientable surface of genus $i$ and $j+1$ boundary components. The boundary components of this piece are $p_{i,j}$ and $j$ boundary components of $N$. When $g=2k+1$ the curve $p_{k,n}$ corresponds to a trivial curve bounding a  M\"{o}bius band, that is why we don't include it in our set. If genus $g = 1$ and $g+n \geq 5$, we let $\mathcal{C}_6 = \emptyset$.
 
\begin{lemma} \label{C_6} If $g+n \geq 5$, then $\bigcup_{i=1} ^6 \mathcal{C}_i$ is a finite superrigid set with trivial pointwise stabilizer.\end{lemma}

\begin{proof} Let $\lambda : \bigcup_{i=1} ^6 \mathcal{C}_i \rightarrow \mathcal{C}(N)$ be a superinjective simplicial map. By Lemma \ref{C_5}, there exists a homeomorphism $h: N \rightarrow N$ such that $h([x]) = \lambda([x])$ for every $x$ in $\bigcup_{i=1} ^5\mathcal{C}_i$. We will show that $h([x]) = \lambda([x])$ for every $x$ in $\mathcal{C}_6$ as well. We will give the proof when $n \geq 1$. The proofs for the other cases are  similar. 
	
The curve $q_2$ is the unique curve up to isotopy disjoint from and nonisotopic to each of $c_1, c_2, d_{1,2}, b_{3,g+n}, b_{3,5}, b_{4,6}, \cdots, $ $b_{g+n-2,g+n}, a_4, a_5, \cdots, a_g$. Since $h([x]) = \lambda([x])$ for all these curves and $\lambda$ preserves these properties we 
have $h([q_2]) = \lambda([q_2])$. The curve $q_4$ is the unique curve up to isotopy disjoint from and nonisotopic to each of $c_1, c_2, c_3, c_4, d_{1,2}, b_{5,g+n}, b_{5,7},$ $b_{6,8}, \cdots, b_{g+n-2,g+n}, a_6, a_7, \cdots, a_g$. Since $h([x]) = \lambda([x])$ for all these curves and $\lambda$ preserves these properties we have $h([q_4]) = \lambda([q_4])$. Similarly, we get $h([q_{2k}]) = \lambda([q_{2k}])$.

The curve $o_2$ is the unique curve up to isotopy disjoint from and nonisotopic to each of $c_1, c_2, e_{1,2}, b_{3,g+n}, b_{3,5}, b_{4,6}, \cdots, $ $b_{g+n-2,g+n}, a_4, a_5, \cdots, a_g$. Since $h([x]) = \lambda([x])$ for all these curves and $\lambda$ preserves these properties we 
have $h([o_2]) = \lambda([o_2])$. The curve $o_4$ is the unique curve up to isotopy disjoint from and nonisotopic to each of $c_1, c_2, c_3, c_4, e_{1,2}, b_{5,g+n}, b_{5,7},$ $b_{6,8}, \cdots, b_{g+n-2,g+n}, a_6, a_7, \cdots, a_g$. Since $h([x]) = \lambda([x])$ for all these curves and $\lambda$ preserves these properties we have $h([o_4]) = \lambda([o_4])$. Similarly, we get $h([o_{2k}]) = \lambda([o_{2k}])$.

The curve $p_{1,0}$ is the unique curve up to isotopy disjoint from and nonisotopic to $c_1, c_2, q_2, b_{3,g+n}, b_{3,5}, b_{4,6}, \cdots, b_{g+n-2,g+n}, a_4, a_5, \cdots, a_g$. Since $h([x]) = \lambda([x])$ for all these curves and $\lambda$ preserves these properties we have $h([p_{1,0}]) = \lambda([p_{1,0}])$. The curve $p_{1,1}$ is the unique curve up to isotopy disjoint from and nonisotopic to $c_1, c_2, d_{1,2}, q_2, b_{3,g+n-1},$ 
$ b_{3,5}, b_{4,6}, \cdots,$ $ b_{g+n-3,g+n-1}, a_4, a_5, \cdots, a_g$. Since $h([x]) = \lambda([x])$ for all these curves and $\lambda$ preserves these properties we have $h([p_{1,1}]) = \lambda([p_{1,1}])$. The curve $p_{1,2}$ is the unique curve up to isotopy disjoint from and nonisotopic to each of $c_1, c_2, d_{1,2}, b_{g+n-2,g+n}, q_2, $ $b_{3,g+n-2}, b_{3,5}, b_{4,6}, \cdots, $ $b_{g+n-4,g+n-2}, a_4, a_5, \cdots, a_g$. Since $h([x]) = \lambda([x])$ for all these curves and $\lambda$ preserves these properties we have $h([p_{1,2}]) = \lambda([p_{1,2}])$. 
Similarly, when $g=2k$, we get the curve $p_{1,n}$ (note that when $g=2k+1$, we check all the way to $p_{1,n-1}$): $p_{1,n}$ is the unique curve up to isotopy disjoint from and nonisotopic to $c_1, c_2, d_{1,2}, b_{3,g}, q_2, b_{g+2,g}, b_{g+1,g+3}, \cdots, b_{g+n-2,g+n}, a_4, a_5, \cdots, a_g$. Since $h([x]) = \lambda([x])$ for all these curves and $\lambda$ preserves these properties we have $h([p_{1,n}]) = \lambda([p_{1,n}])$. Hence, $h([p_{1,j}]) = \lambda([p_{1,j}])$ for all the curves we consider in $\mathcal{C}_5$.  

The curve $p_{2,0}$ is the unique curve up to isotopy disjoint from and nonisotopic to $c_1, c_2, c_3, c_4, q_4, b_{5,g+n}, b_{5,7}, b_{6,8}, \cdots, b_{g+n-2,g+n}, a_6, a_7, \cdots, a_g$. Since $h([x]) = \lambda([x])$ for all these curves and $\lambda$ preserves these properties we have $h([p_{2,0}]) = \lambda([p_{2,0}])$. The curve $p_{2,1}$ is the unique curve up to isotopy disjoint from and nonisotopic to $c_1, c_2, c_3, c_4, d_{1,2}, q_4, b_{5,g+n-1},$ $b_{5,7}, b_{6,8}, \cdots,$ $ b_{g+n-3,g+n-1}, a_6, a_7, \cdots, a_g$. Since $h([x]) = \lambda([x])$ for all these curves and $\lambda$ preserves these properties we have $h([p_{2,1}]) = \lambda([p_{2,1}])$. Similarly, when $g=2k$, we get the curve $p_{2,n}$ (note that when $g=2k+1$, we check all the way to $p_{2,n-1}$): $p_{2,n}$ is the unique curve up to isotopy disjoint from and nonisotopic to $c_1, c_2, c_3, c_4, d_{1,2}, b_{5,g}, q_4, b_{g+2,g}, b_{g+1,g+3}, \cdots,$ 
$b_{g+n-2,g+n}, a_6, a_7, \cdots, a_g$. Since $h([x]) = \lambda([x])$ for all these curves and $\lambda$ preserves these properties we have $h([p_{2,n}]) = \lambda([p_{2,n}])$. Hence, $h([p_{2,j}]) = \lambda([p_{2,j}])$ for all the curves we consider in $\mathcal{C}_5$. 

Continuing this way, we get  $h([p_{i,j}]) = \lambda([p_{i,j}])$ for all the curves we consider in $\mathcal{C}_6$. So, $h([x]) = \lambda([x])$ for every $x$ in $\bigcup_{i=1} ^6 \mathcal{C}_i$. Hence, $\bigcup_{i=1} ^6 \mathcal{C}_i$ is a finite superrigid set. It is easy to see that the pointwise stabilizer of $\bigcup_{i=1} ^6 \mathcal{C}_i $ is trivial.\end{proof}\\
 
\begin{figure}[t]
	\begin{center}
		\epsfxsize=2.9in \epsfbox{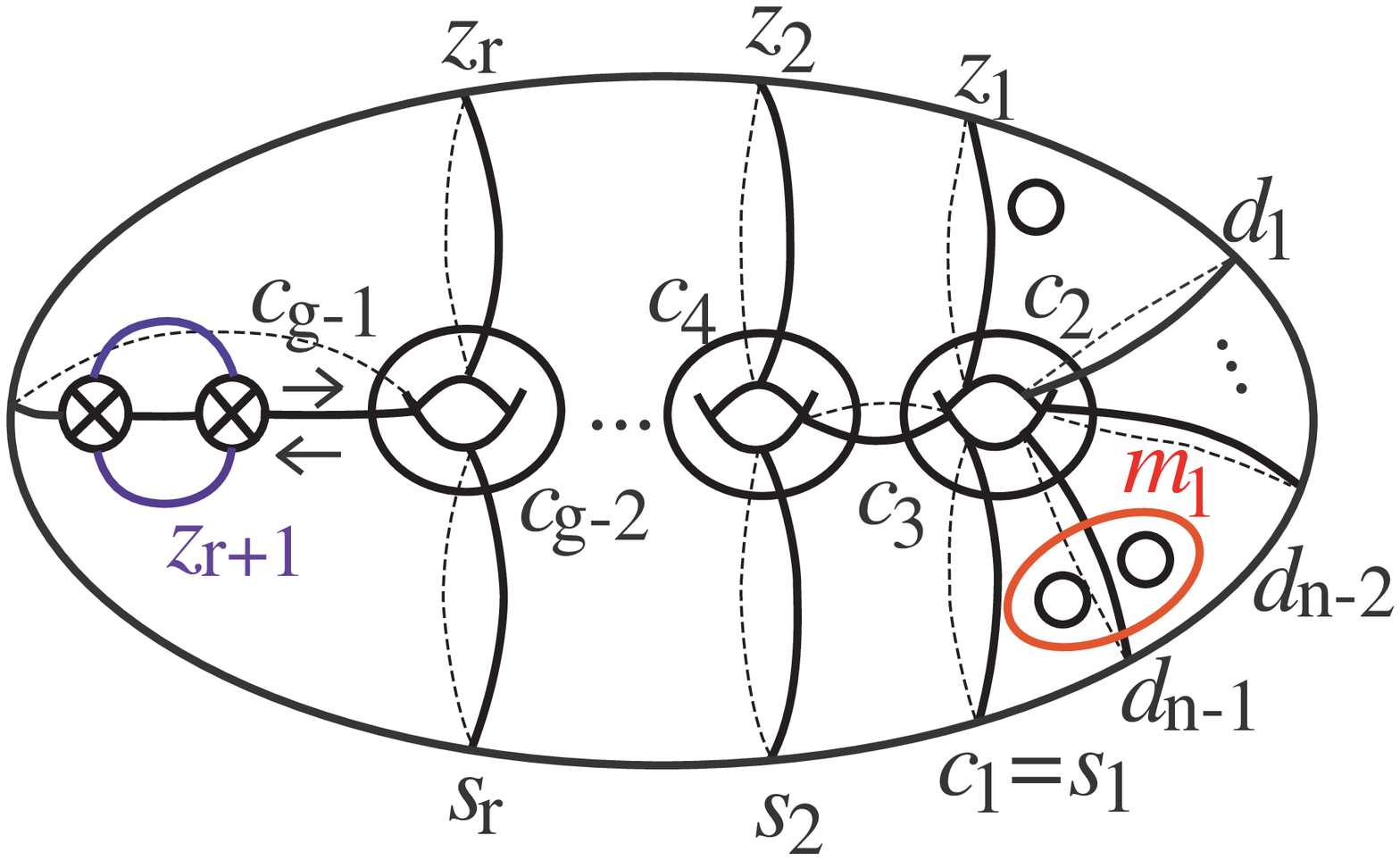} \hspace{-0.1cm} \epsfxsize=2.9in \epsfbox{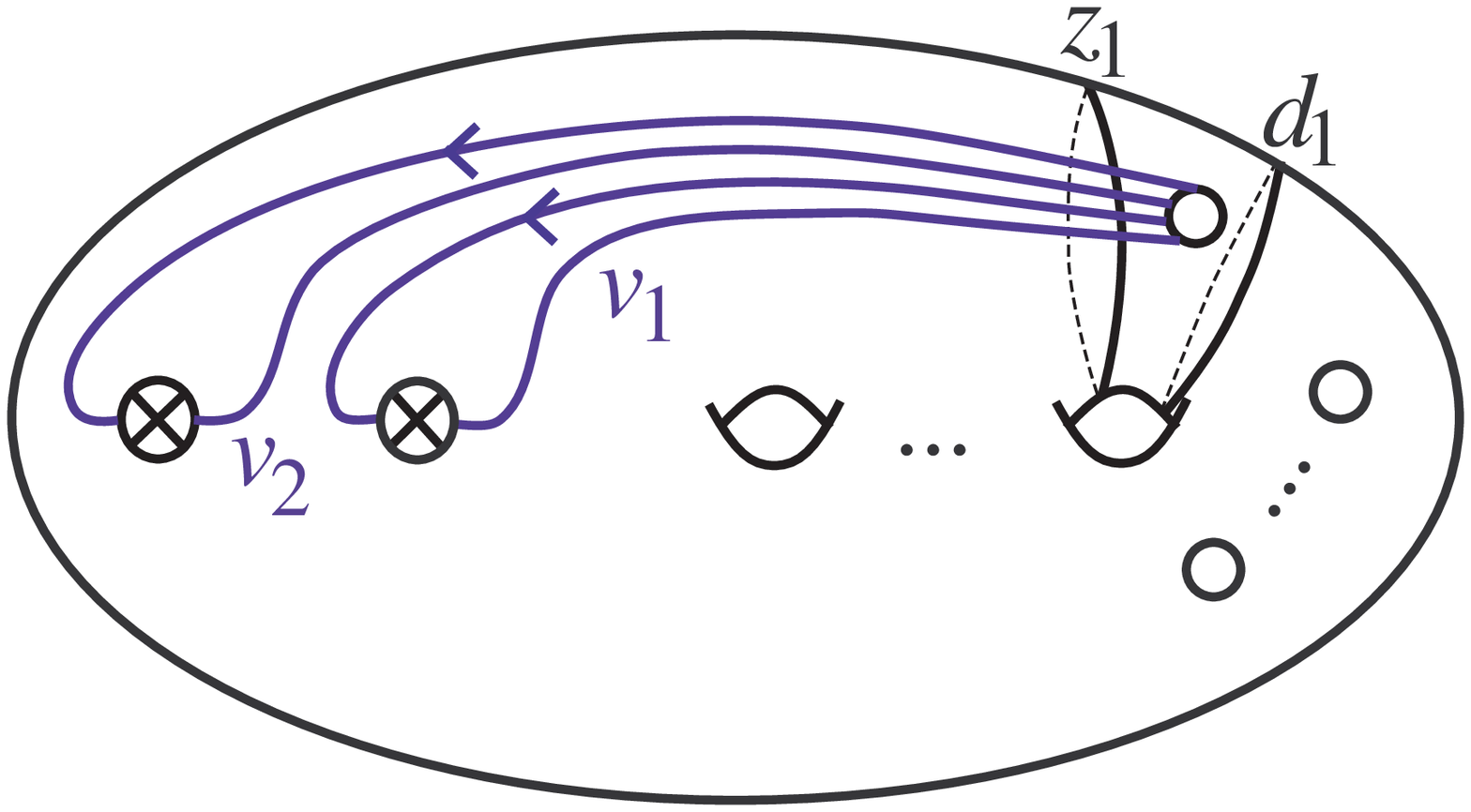}
		
		\hspace{-1cm} (i) \hspace{6.5cm} (ii)
		
		\hspace{0cm} \epsfxsize=3.1in \epsfbox{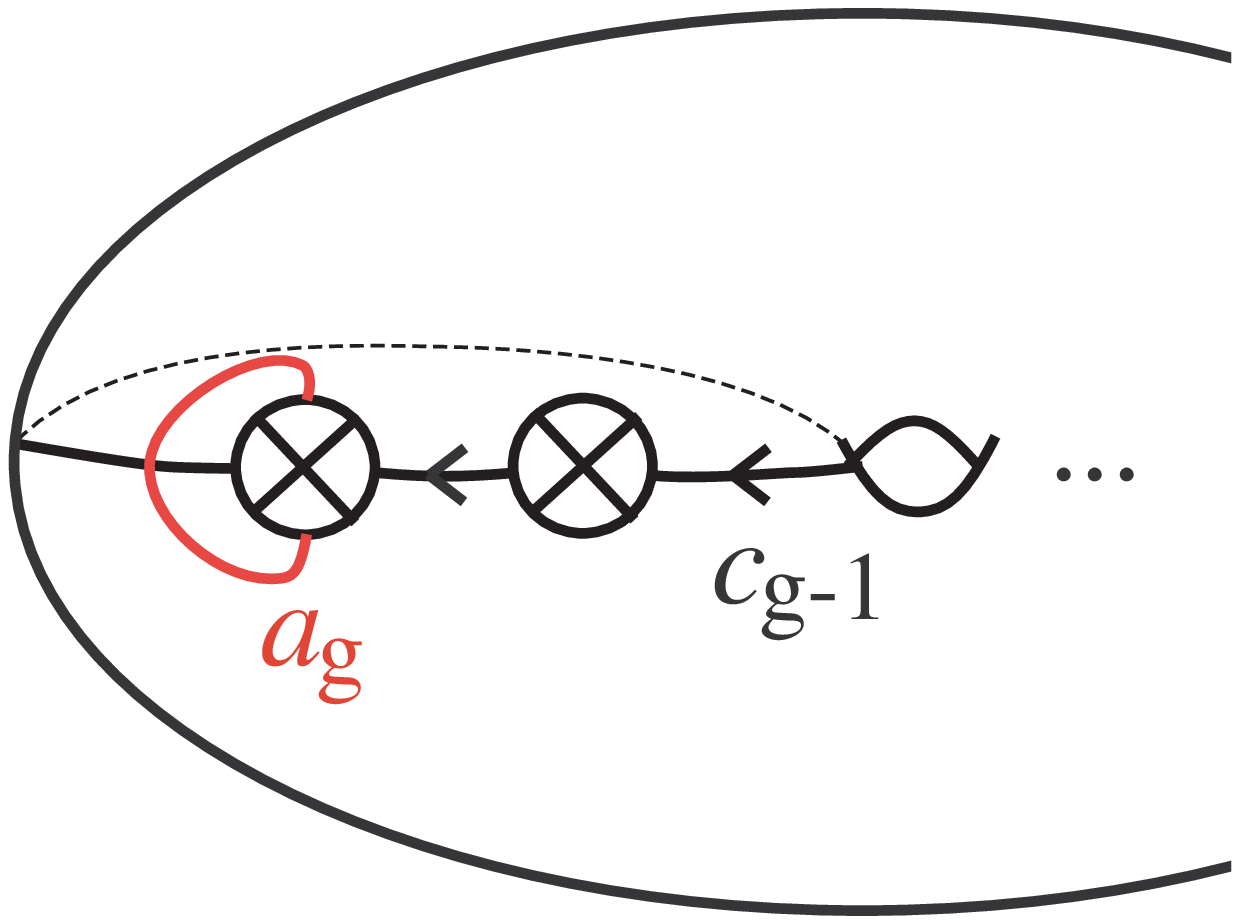} \hspace{-1.4cm} \epsfxsize=3.1in \epsfbox{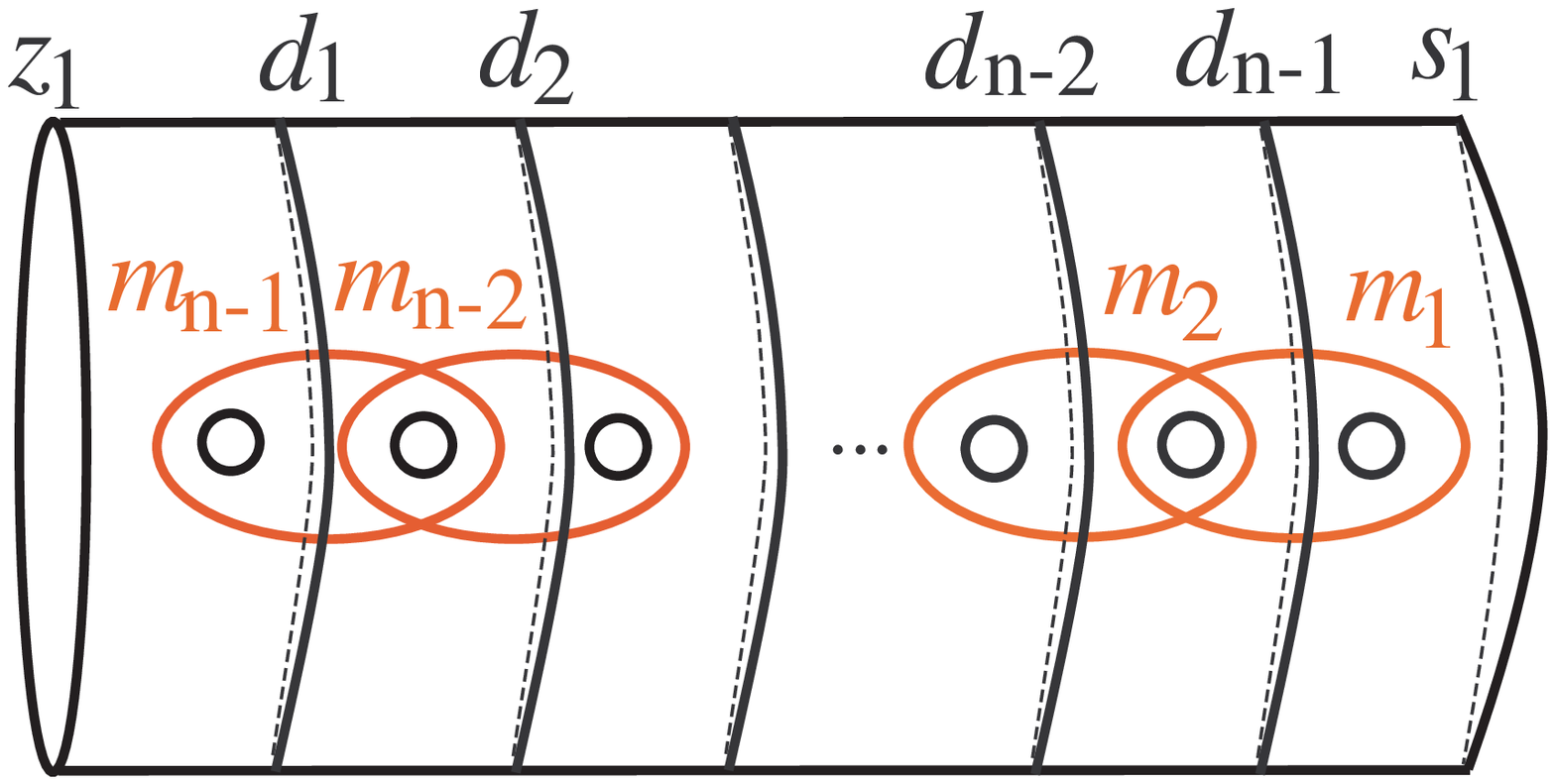}
		
		\hspace{-1cm} (iii) \hspace{6.5cm} (iv)
		
		\caption{Generators for even genus $g=2r+2$, $g \geq 4$} \label{Fig-10}
	\end{center}
\end{figure}
 
\begin{figure}
	\begin{center}
		\epsfxsize=2.7in \epsfbox{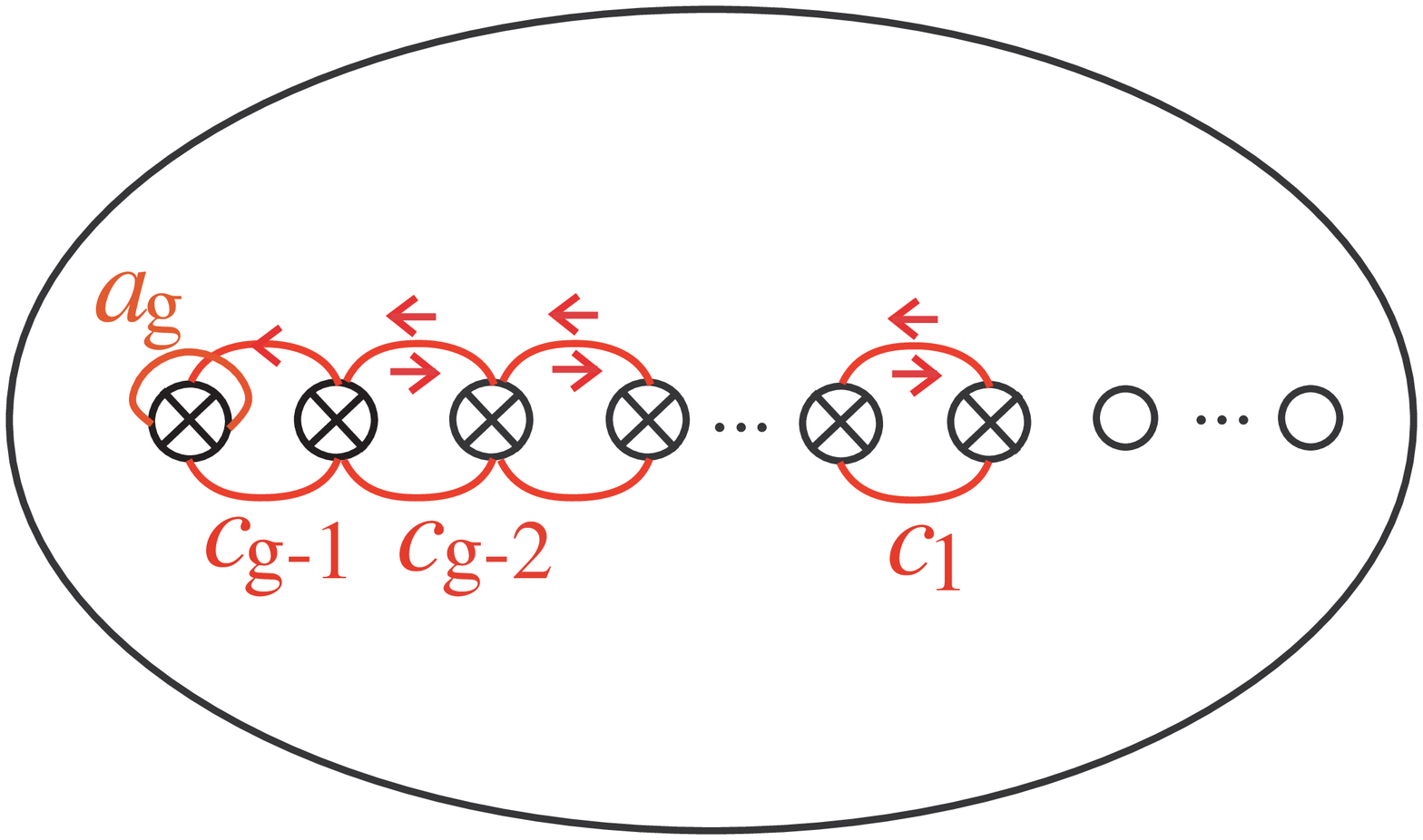}   \hspace{0.1cm} 
			\epsfxsize=2.7in \epsfbox{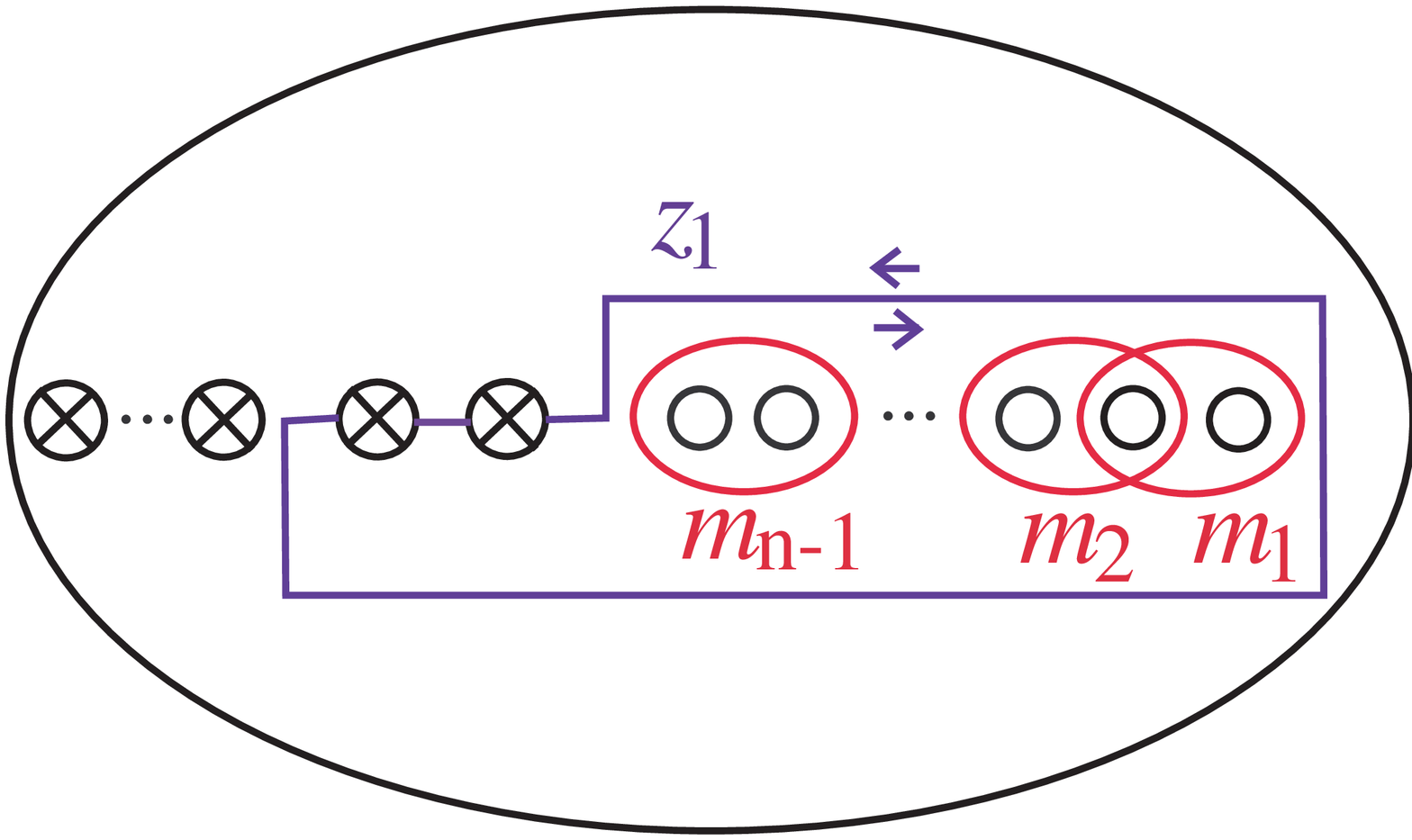}  
			
				(i)  \hspace{6cm}   (ii) 
				
				\epsfxsize=2.7in \epsfbox{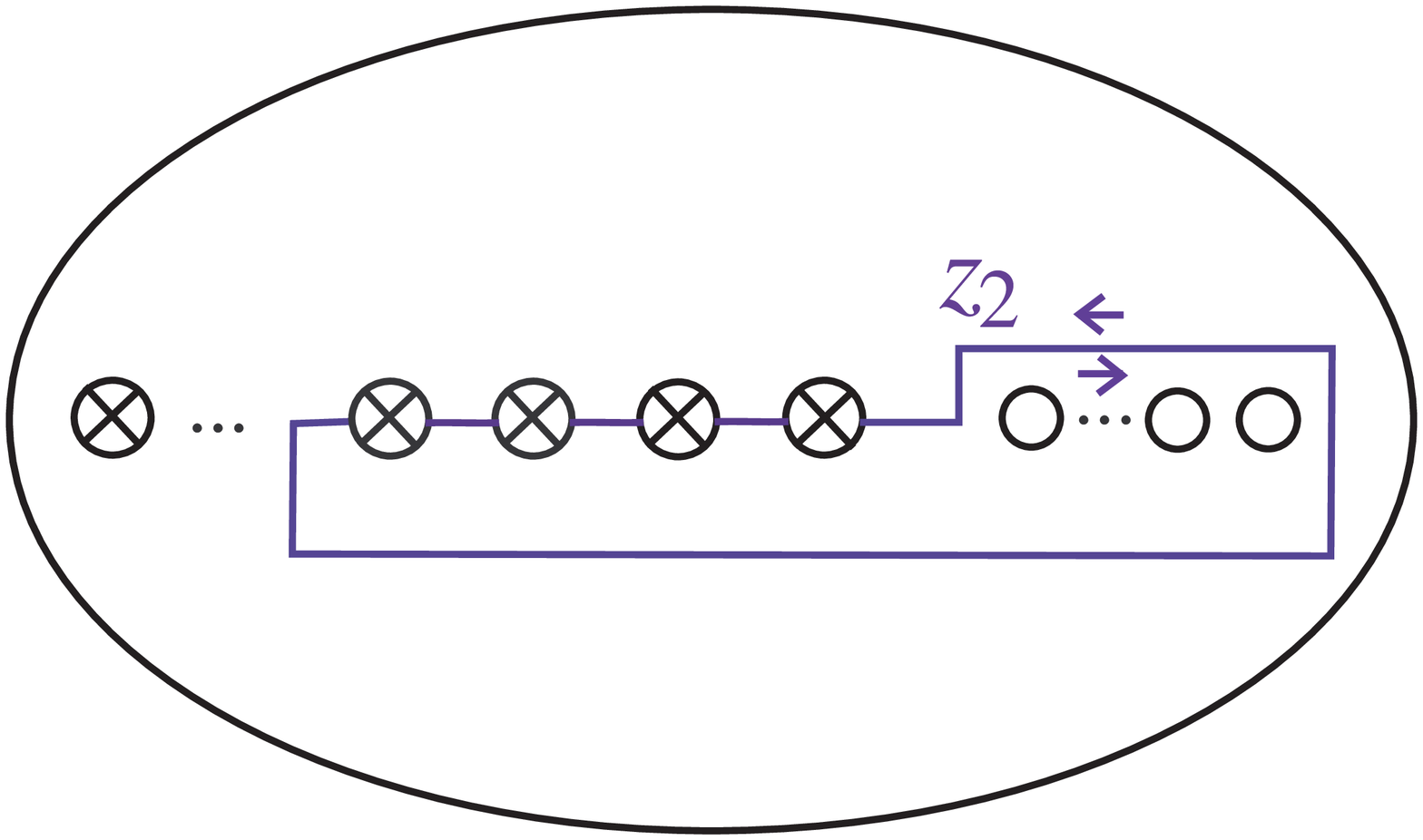}  \hspace{0.1cm}  	\epsfxsize=2.7in \epsfbox{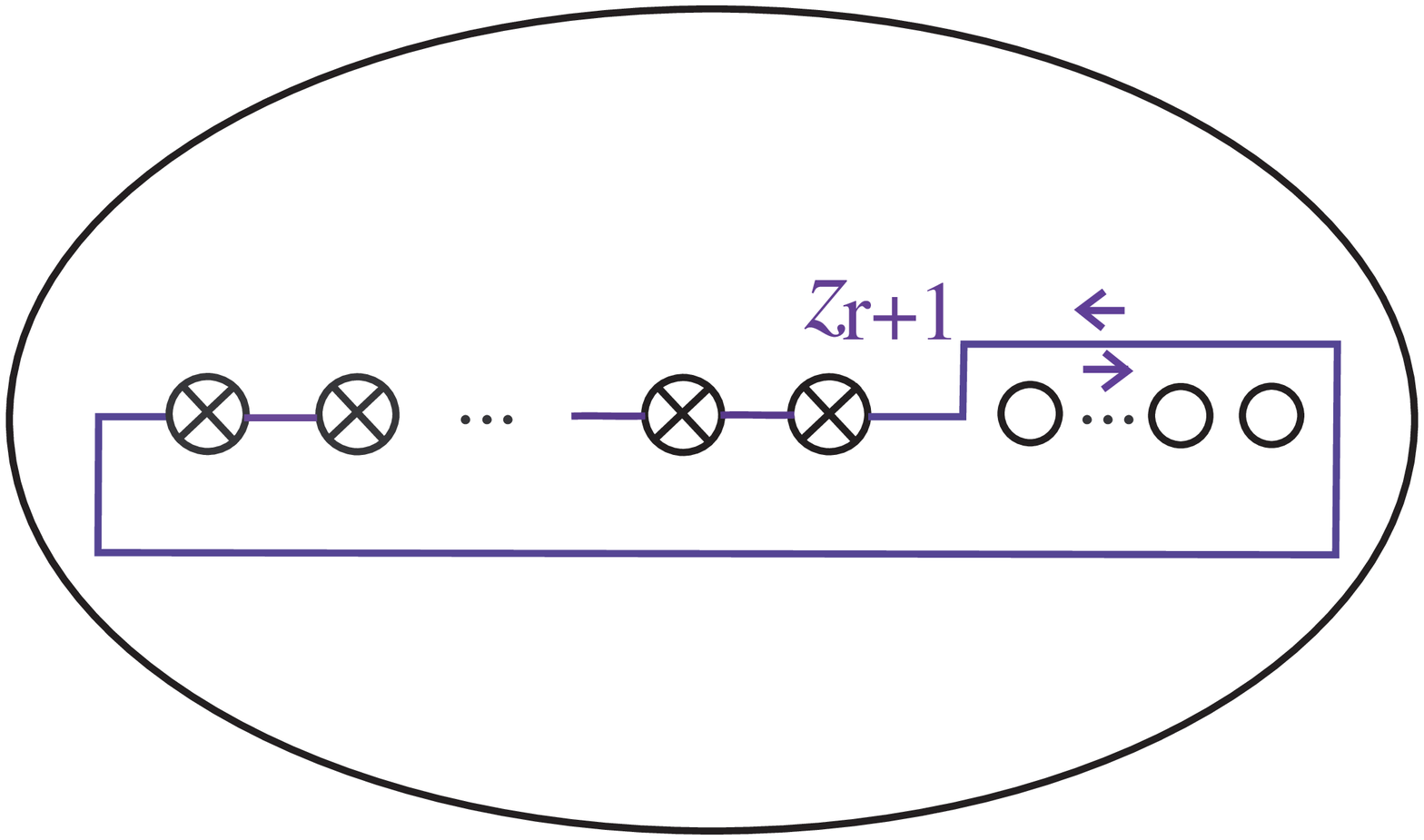}  
				
					(iii)  \hspace{6cm}   (iv) 
					
					\epsfxsize=2.7in \epsfbox{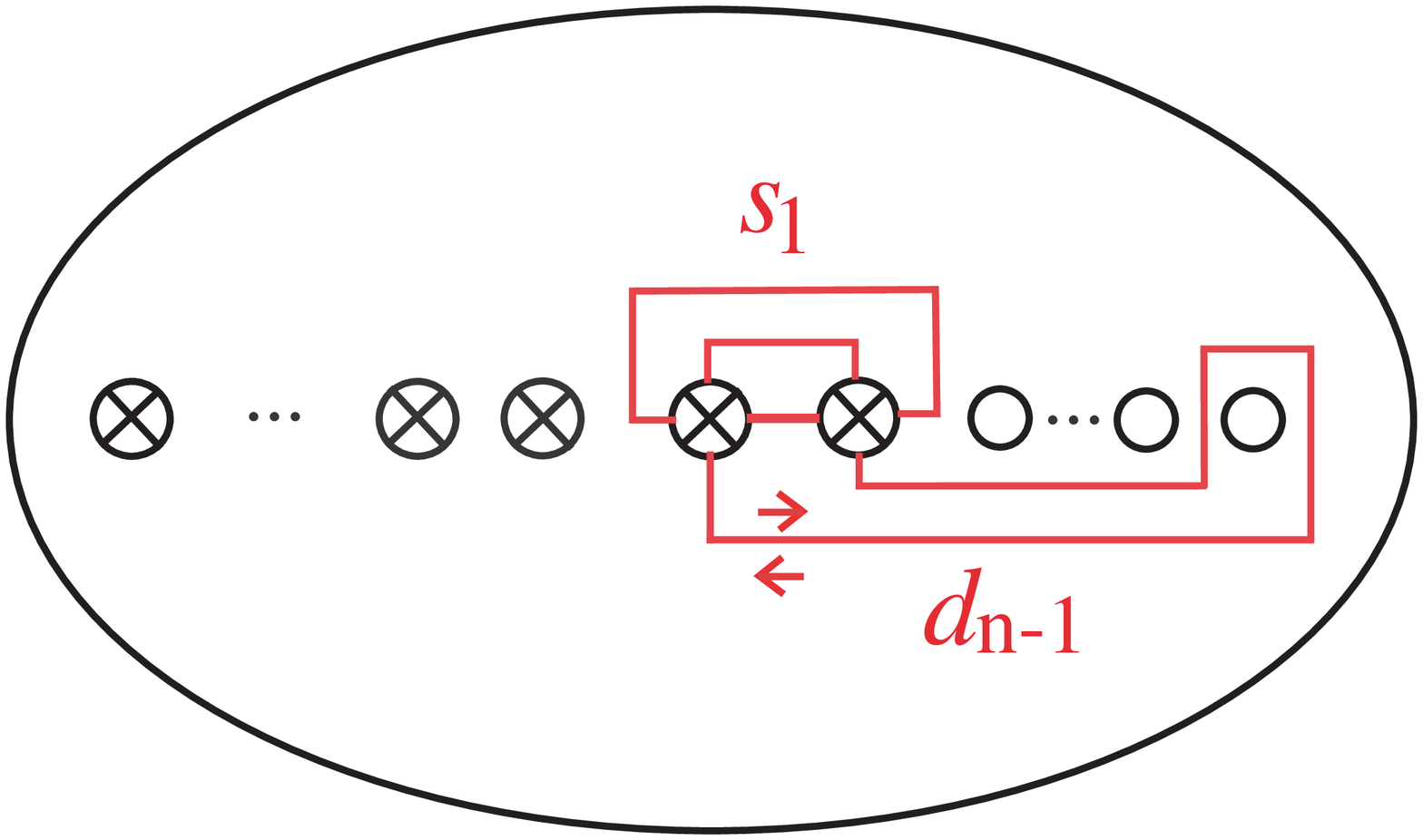}  \hspace{0.1cm}  	\epsfxsize=2.7in \epsfbox{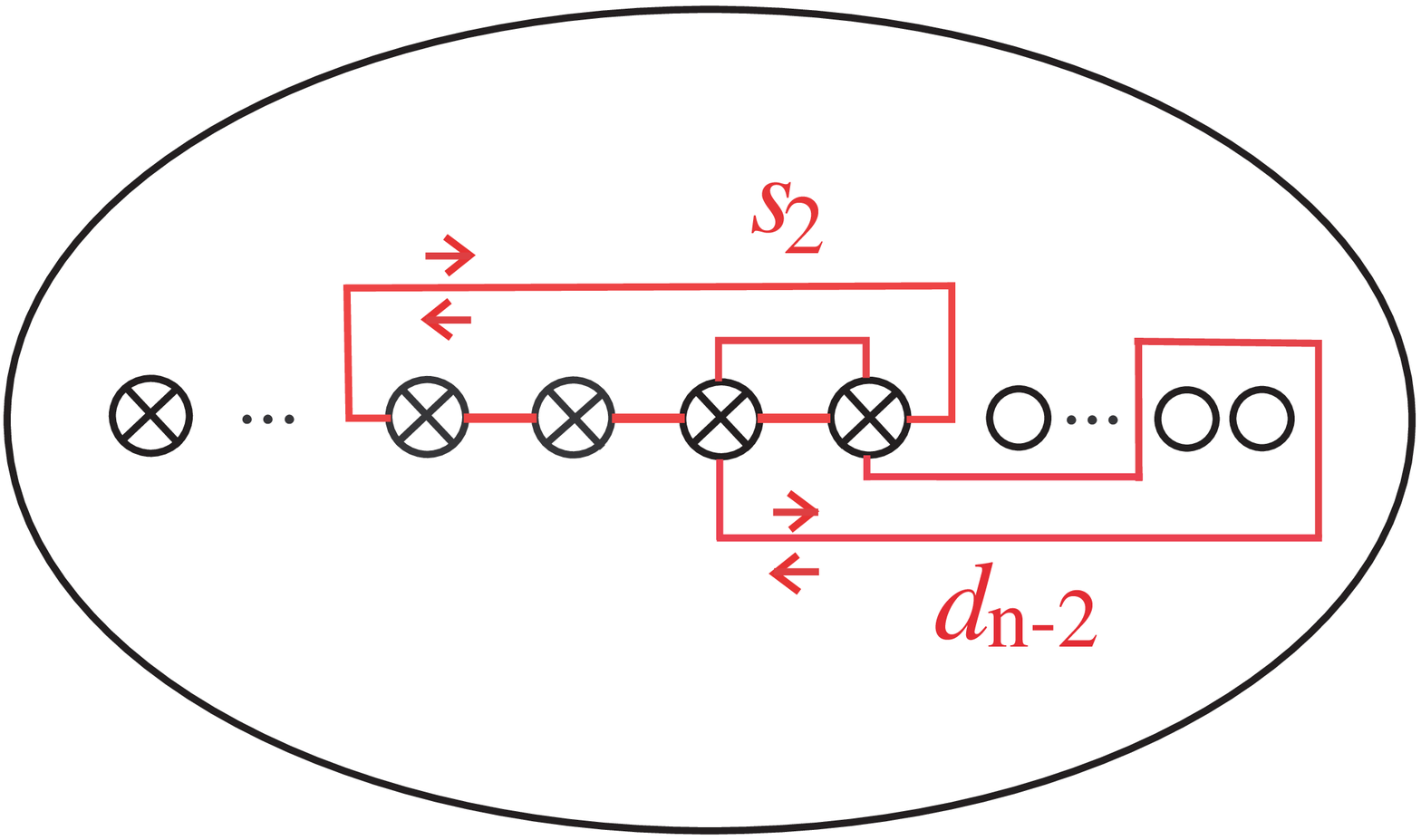}  
					
						(v)  \hspace{6cm}   (vi) 
					
						\epsfxsize=2.7in \epsfbox{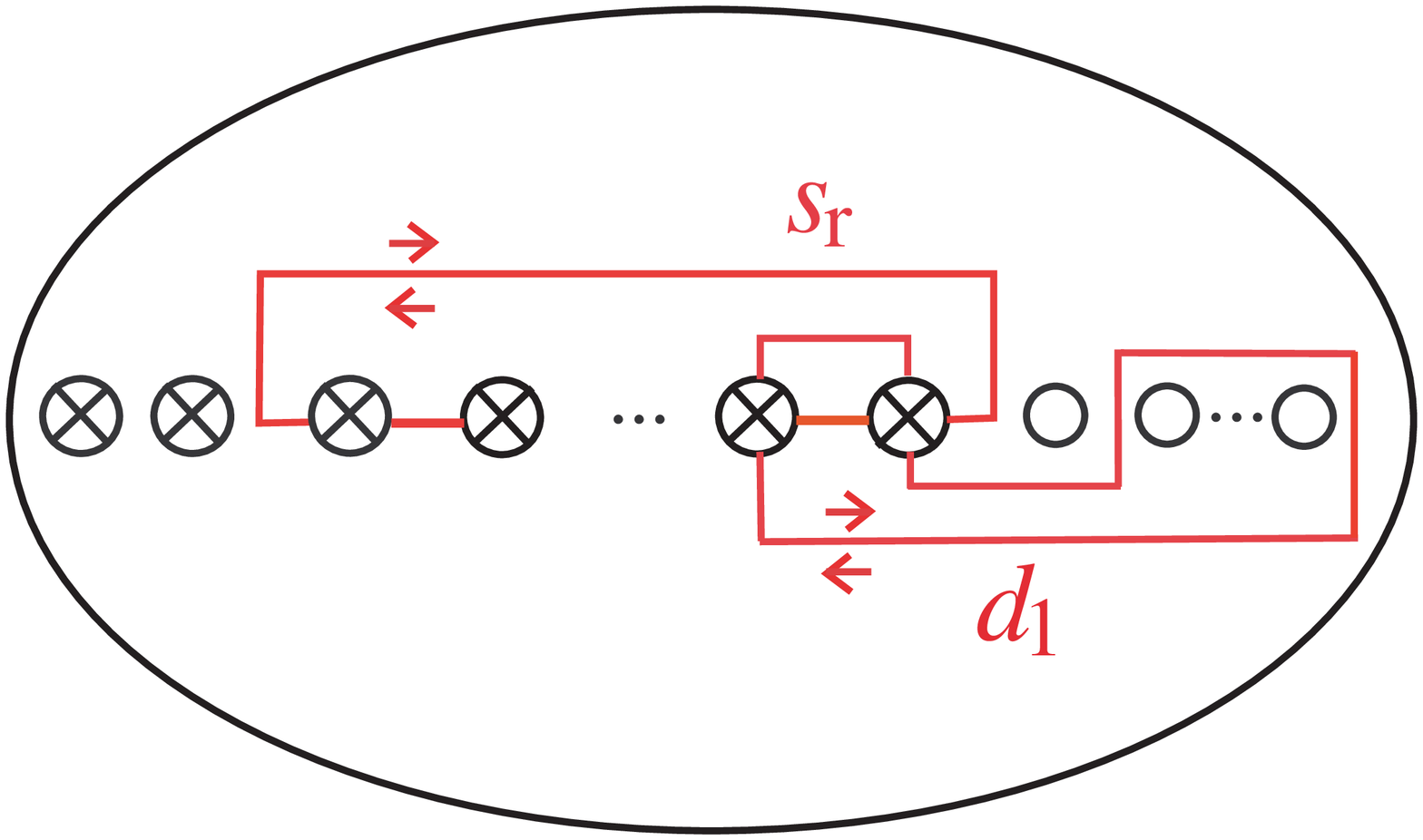}  \hspace{0.1cm} \epsfxsize=2.7in \epsfbox{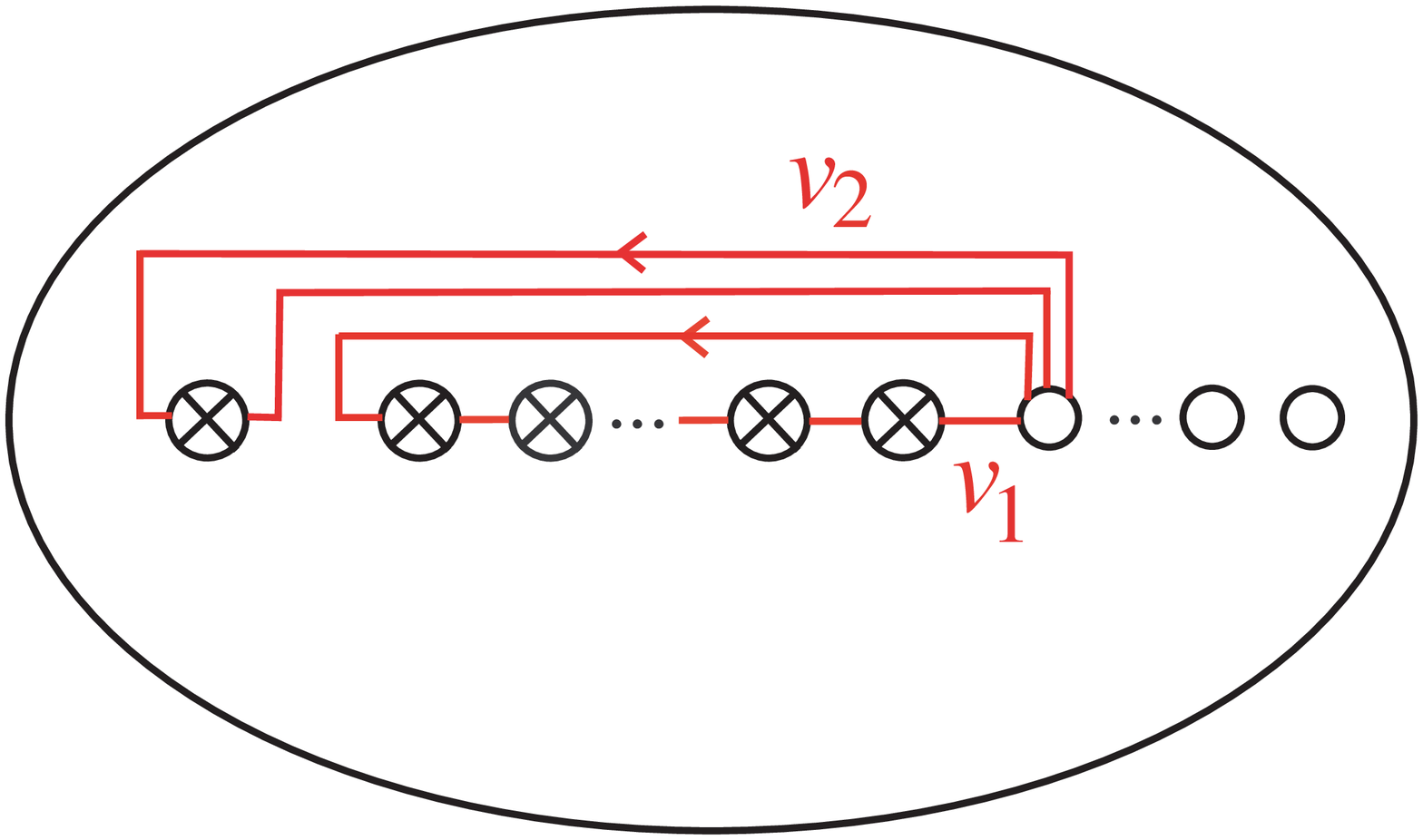}
						
							(vii)  \hspace{6cm}   (viii)   
		\caption{Generators for even genus $g=2r+2$, $g \geq 4$}
		\label{Fig-9}
	\end{center}
\end{figure}

Let $\mathcal{C}= \bigcup_{i=1} ^6\mathcal{C}_i$ when $g+n \geq 5$. 
We note that $\mathcal{C}$ has curve of every topological type on $N$. 
In the following theorem we consider the curves given in Figure \ref{Fig-9}. Let $t_x$ be the Dehn twist about $x$. For $b_{i,i+2}$ that separates a sphere with three holes, let $\sigma_{i-g-1}$ be the half twist along $b_{i,i+2}$ (i.e. the elementary braid supported in a regular neighborhood of an arc connecting the two boundary components of $N$ that $b_{i, i+2}$ separates so that $\sigma_i ^2$ is a right Dehn twist along $b_{i,i+2}$.) We will work with half twists $\{\sigma_1, \sigma_2, \cdots, \sigma_{n-1}\}$. To make the notation simpler we will rename some of the curves in $\mathcal{C}$ and call them $m_i$ as shown in Figure \ref{Fig-9} (ii) so that $\sigma_i$ becomes a half twists along the curve $m_i$. Let $y= y_{a_g, c_{g-1}}$ be the crosscap slide of $a_g$ along $c_{g-1}$. We recall that the crosscap slide is 
defined as follows: Consider a M\"{o}bius band $A$ with one hole. By attaching another M\"{o}bius band $B$ to $A$ 
along one of the boundary components of $A$, we get a Klein bottle $K$ with one hole. By sliding $B$ once along the 
core of $A$, we get a homeomorphism of $K$ which is identity on the boundary of $K$. This homeomorphism can be 
extended by identity to a homeomorphism of the whole surface if $K$ is embedded in a surface. The isotopy class of this homeomorphism is called a crosscap slide. Let $\xi_1$ be the boundary slide of the boundary component $t_{g+n}$ along the arc $v_1$, and $\xi_2$ be the boundary slide of the boundary component $t_{g+n}$ along the arc $v_2$ where the arcs $v_1, v_2$ are given in Figure \ref{Fig-9} (viii). We recall that boundary slide $\xi_1$ is defined as follows: Let $A$ be a regular neighborhood of $t_{g+n} \cup v_1$ on $N$. Then $A$ is a M\"{o}bius band with one hole $t_{g+n}$ (a nonorientable surface of genus one with two boundary components). Let $k$ be the other boundary component of $A$. By sliding $t_{g+n}$ along the core of $A$ (in the direction of $v_1$) we get a homeomorphism which is identity on $k$. This homeomorphism can be extended by identity to a homeomorphism of $N$. The isotopy class of this homeomorphism is called boundary slide of $t_{g+n}$ along $v_1$.  
  
\begin{theorem}\label{prop1} If $g \geq 4$, $g=2r+2$ for some $r \in \mathbb{Z}^+$, then $Mod_N$ is generated by $\{t_x: x \in \{c_1, c_2, \cdots, c_{g-1}, s_1, s_2, \cdots, s_r, z_1, z_2, \cdots, z_{r+1}, d_1, d_2, \cdots, d_{n-1}\} \cup \{\sigma_1, \sigma_2, \cdots, \sigma_{n-1}, $ $y, \xi_1, \xi_2\}\}$.\end{theorem}

\begin{proof} The proof follows from the proof of Theorem 4.14 given by Korkmaz in \cite{K2} by representing the curves used in his generating set given in Figure \ref{Fig-10}, in our crosscap model of $N$ as shown in Figure \ref{Fig-9} (see section 3 in Stukow's paper \cite{St} to see some part of this transfer).\end{proof}\\
    
\begin{figure}
	\begin{center}
		\epsfxsize=2.7in \epsfbox{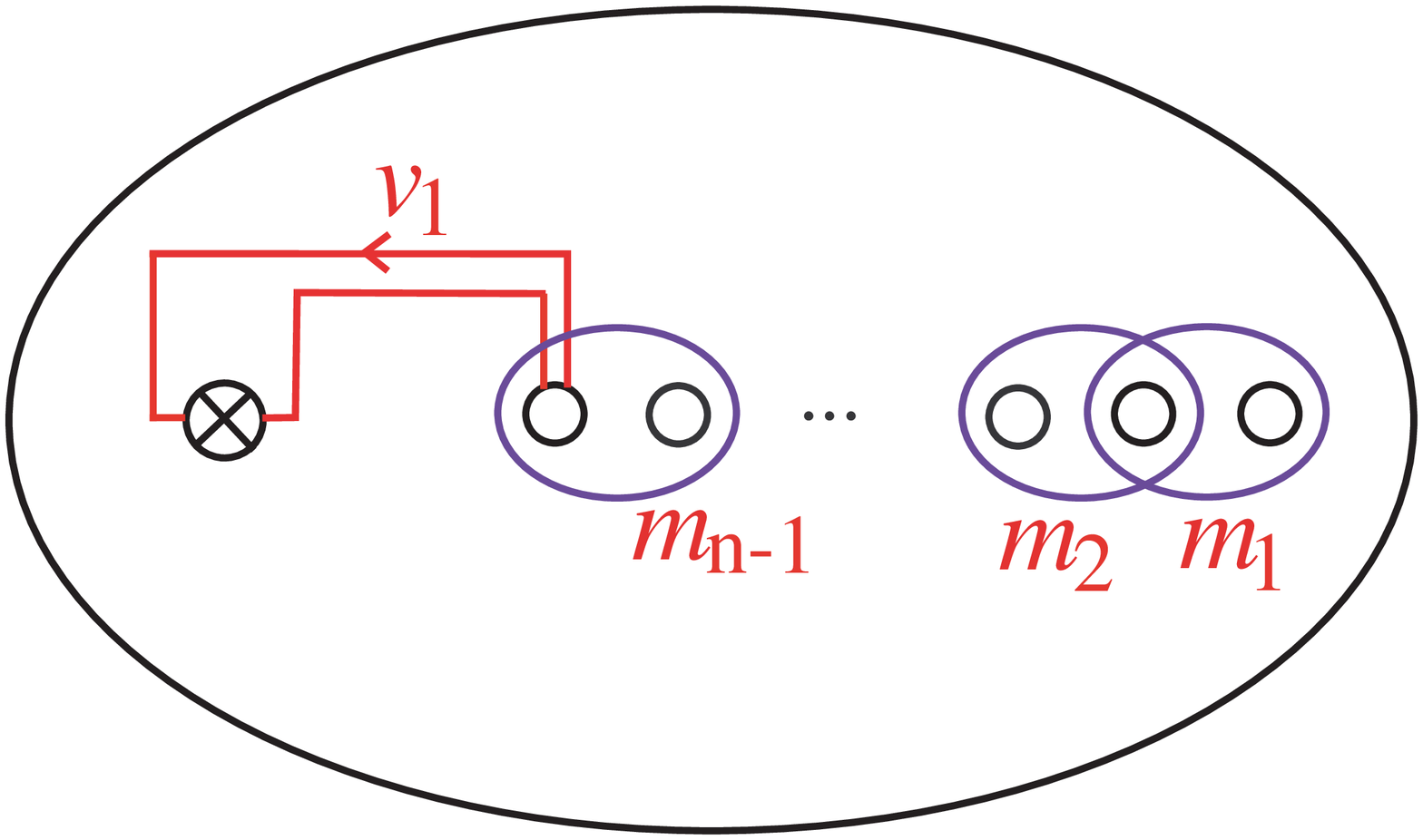}   \hspace{0.1cm} 
		\epsfxsize=2.7in \epsfbox{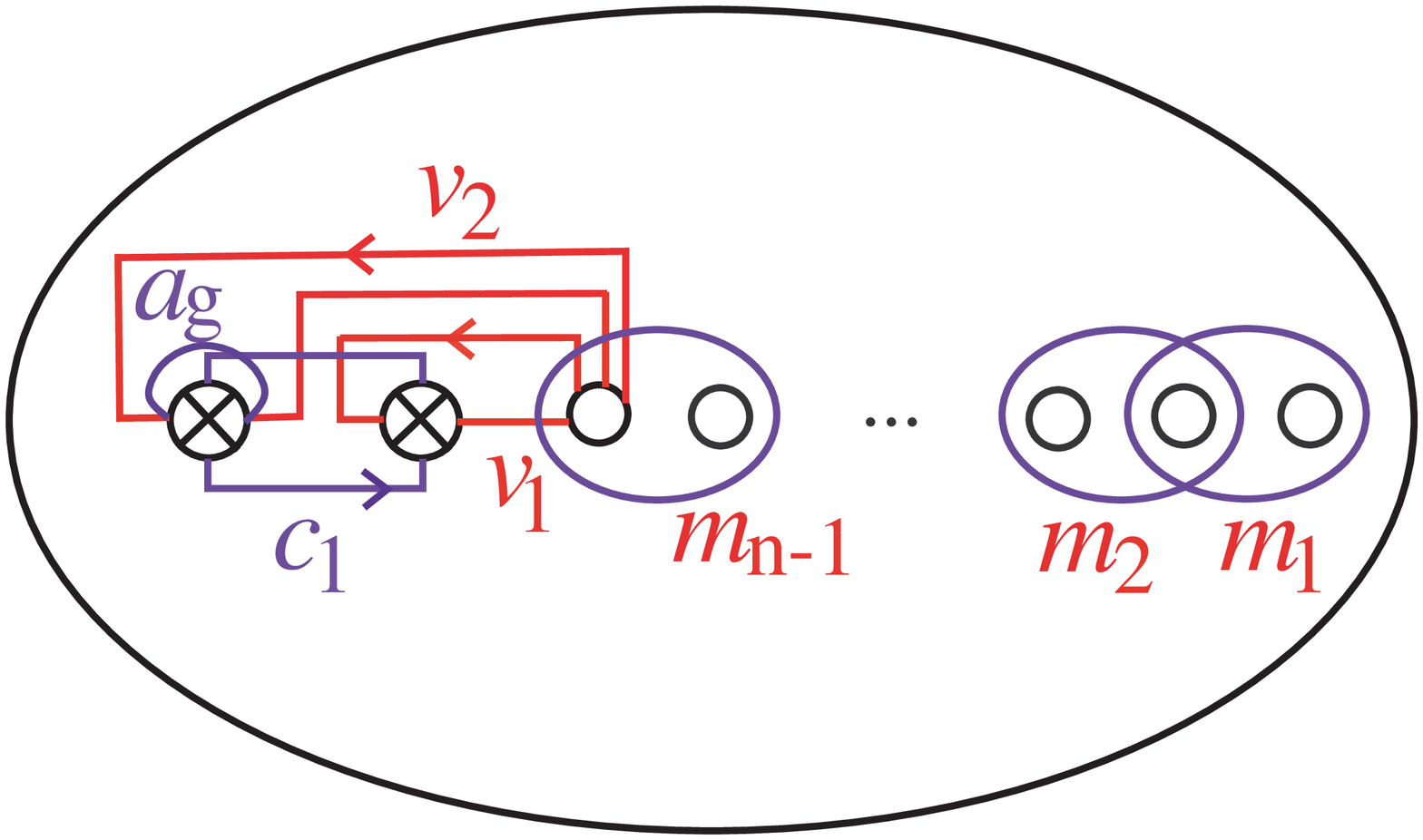}  
		
		(i)  \hspace{6cm}   (ii) 
		
		\caption{Generators (i) when $g=1, n \geq 1$, (ii) when $g=2, n \geq 1$}
		\label{Fig-09}
	\end{center}
\end{figure}

Consider Figure \ref{Fig-09} (ii). The proof of the following theorem was given by Korkmaz (see Theorem 4.10 in  \cite{K2}).
  
\begin{theorem}\label{theorem-2} If $g = 2, n \geq 1$, then $Mod_N$ is generated by $\{t_{c_1}, \sigma_1, \sigma_2, \cdots, \sigma_{n-1}, y, \xi_1,$ 
$\xi_2\}$.\end{theorem} 

If $g = 2, n \geq 1$, let $G_1= \{t_{c_1}, \sigma_1, \sigma_2, \cdots, \sigma_{n-1}, y, \xi_1,$ $\xi_2\}$. If $g \geq 4$, $g=2r+2$ for some $r \in \mathbb{Z}$, let $G_1= \{t_x: x \in \{c_1, c_2, \cdots, c_{g-1}, s_1, s_2, \cdots, s_r, $ $ z_1, z_2, \cdots, z_{r+1}, d_1, d_2, \cdots, $ $d_{n-1}\}\}$ $ \cup \{\sigma_1, \sigma_2, \cdots, \sigma_{n-1}, y, \xi_1, \xi_2\}$. We remind that $\mathcal{C}= \bigcup_1 ^6\mathcal{C}_i$. 
 
\begin{lemma} \label{L_f_even} If $g + n \geq 5$ and $g$ is even, then $\forall \ f \in G_1$, $\exists$ a set $L_f \subset \mathcal{C}$ such that $L_f$ has trivial pointwise stabilizer and $f(L_f) \subset \mathcal{C}$.\end{lemma}

\begin{proof} Suppose $g + n \geq 5$ and $g$ is even. Let $f \in G_1$. For $f=t_{c_1}$, let $L_f = \{b_{2,4}, b_{3,5} \cdots, b_{g+n-2,g+n}, c_1, b_{g+n,2}, a_{1,g+n-1}, e_{1,2}\}$. The set $L_f$ has trivial pointwise stabilizer. Since $t_{c_1}(a_{1,g+n-1}) = a_{1,g+n-2}$ and the other curves in $L_f$ are fixed by $t_{c_1}$, we have $f(L_f) \subset \mathcal{C}$. 
Similarly we get the result for every $t_{c_i}$ when $i=2, 3, \cdots, n-1$.
	
For $f=t_{d_1}$, let $L_f = \{b_{2,4}, b_{3,5} \cdots, b_{g-2,g}, b_{g,g+2}, b_{g+1,g+3}, \cdots, b_{g+n-3,g+n-1}, d_{n-1}, d_{1,2},$ $ s_{1,2}, c_1,  a_{1,g+n-1}\}$. The set $L_f$ has trivial pointwise stabilizer. Since $t_{d_1}(a_{1,g+n-1}) = a_{2,g}$ and the other curves in $L_f$ are fixed by $t_{d_1}$, we have $f(L_f) \subset \mathcal{C}$. For $f=t_{d_2}$, let $L_f = \{b_{2,4}, b_{3,5} \cdots, b_{g-2,g}, b_{g,g+2}, b_{g+1,g+3}, \cdots, b_{g+n-4,g+n-2}, d_{n-1}, d_{1,2}, s_{1,2}, c_1, a_{1,g+n-2}\}$. The set $L_f$ has trivial pointwise stabilizer. Since $t_{d_2}(a_{1,g+n-2}) = a_{2,g}$ and the other curves in $L_f$ are fixed by $t_{d_2}$, we have $f(L_f) \subset \mathcal{C}$. Similarly we get the result for every $t_{d_i}$ when $i=3, 4, \cdots, n-1$.

We know the result for $t_{s_1}$ since $s_1=c_1$. For $f=t_{s_2}$, let $L_f = \{b_{4,6}, b_{5,7}, \cdots, $ $ b_{g+n-2,g+n}, $ $e_{1,2}, r_{3,4}, c_1, c_2, c_3, a_{4,g+n}\}$. The set $L_f$ has trivial pointwise stabilizer. Since $t_{s_2}(a_{4,g+n}) = q_2$ and the other curves in $L_f$ are fixed by $t_{s_2}$, we have $f(L_f) \subset \mathcal{C}$. 
Similarly we get the result for every $t_{s_i}$ when $i=3, 4, \cdots, r$.
	
For $f=t_{z_1}$, if $n =0$ then let $L_f = \{b_{2,4}, b_{3,5} \cdots, b_{g-2,g}, s_{1,2}, c_1,$ $ a_1\}$. The set $L_f$ has trivial pointwise stabilizer. Since $t_{z_1}(a_1) = a_{2,g}$ and the other curves in $L_f$ are fixed by $t_{z_1}$, we have $f(L_f) \subset \mathcal{C}$. 	
For $f=t_{z_1}$, if $n \geq 1$ then let $L_f = \{b_{2,4}, b_{3,5} \cdots, b_{g-2,g}, b_{g,g+2},b_{g+1,g+3}, \cdots, b_{g+n-2,g+n},$ $s_{1,2}, e_{1,2}, c_1,$ $ a_1\}$. The set $L_f$ has trivial pointwise stabilizer. Since $t_{z_1}(a_1) = a_{2,g}$ and the other curves in $L_f$ are fixed by $t_{z_1}$, we have $f(L_f) \subset \mathcal{C}$. 
For $f=t_{z_2}$ when $g=4, n \geq 1$, let $L_f = \{b_{4,6},b_{5,7}, \cdots, b_{2+n,4+n},$ $e_{1,2}, c_1, c_2, c_3, a_{4,g}\}$. The set $L_f$ has trivial pointwise stabilizer. Since $t_{z_2}(a_{4,g}) = o_2$ and the other curves in $L_f$ are fixed by $t_{z_2}$, we have $f(L_f) \subset \mathcal{C}$. 
For $f=t_{z_2}$, when $g \geq 6, n =0$, let $L_f = \{b_{4,6}, b_{5,7} \cdots, b_{g-2,g},$ $s_{3,4}, d_{1,2}, c_1, c_2, c_3, a_{4,g}\}$. The set $L_f$ has trivial pointwise stabilizer. Since $t_{z_2}(a_{4,g}) = o_2$ and the other curves in $L_f$ are fixed by $t_{z_2}$, we have $f(L_f) \subset \mathcal{C}$. 
For $f=t_{z_2}$, when $g \geq 6, n \geq 1$, let $L_f = \{b_{4,6}, b_{5,7}, \cdots,  b_{g-2,g}, b_{g,g+2}, b_{g+1,g+3}, \cdots, b_{g+n-2,g+n},$ $s_{3,4}, e_{1,2}, c_1, c_2, c_3, a_{4,g}\}$. The set $L_f$ has trivial pointwise stabilizer. Since $t_{z_2}(a_{4,g}) = o_2$ and the other curves in $L_f$ are fixed by $t_{z_2}$, we have $f(L_f) \subset \mathcal{C}$. Similarly we get the result for every $t_{z_i}$ when $i=3, 4, \cdots, r+1$.
	
For $f=\sigma_{n-1}$, let $L_f = \{b_{1,3}, b_{2,4}, \cdots, b_{g+n-2,g+n}, b_{g+n-1,1}, b_{g+n,2}, w_{i_o}\}$, where $w_{i_o}$ is an element in $\{w_1, w_2, \cdots, w_n\}$ disjoint from  $b_{g+n-2,g+n}$. The set $L_f$ has trivial pointwise stabilizer. Since $\sigma_{n-1}(b_{g+n-3,g+n-1})= r_{g+n-1}$,  $\sigma_{n-1}(b_{g+n-1,1})= w_{g+n}$ and the other curves in $L_f$ are fixed by $\sigma_{n-1}$, we have $f(L_f) \subset \mathcal{C}$. 
With similar construction, we get the result for every $\sigma_i$ when $i=1, 2, \cdots, n-2$.
  
For $f=y$, when $g=2$ let $L_f = \{a_1, a_2, b_{2,2+n}, c_{1}, d_{1,2}, b_{2,4}, b_{3,5}, \cdots, b_{n,2+n}\} \subset \mathcal{C}$. The set $L_f$ has trivial pointwise stabilizer. Since $y(a_{1})=a_{1,2}$ and the other curves in $L_f$ are fixed by $y$, we have $f(L_f) \subset \mathcal{C}$.   
For $f=y$, when $g \geq 4$ let $L_f = \{a_{g-1}, a_g, b_{g-2,g}, c_{g-1}, d_{g-1,g}, b_{1,3}, b_{2,4}, \cdots, b_{g-4,g-2}, b_{g,g+2}, b_{g+1,g+3},$ $ \cdots, b_{g+n-2,g+n},$ $ b_{g+n-1,1}, b_{g+n,2} \} \subset \mathcal{C}$. The set $L_f$ has trivial pointwise stabilizer. Since $y(a_{g-1})=a_{g-1,g}$ and the other curves in $L_f$ are fixed by $y$, we have $f(L_f) \subset \mathcal{C}$. 

For $f=\xi_1$, when $g = 2$ let $L_f = \{a_1, b_{1,3}, b_{2,4}, \cdots, b_{n-1,n+1}, b_{1,n+1}, w_3\} \subset \mathcal{C}$. The set $L_f$ has trivial pointwise stabilizer. Since $\xi_1(a_1)= a_{1,1+n}$ and the other curves in $L_f$ are fixed by $\xi_1$, we have $f(L_f) \subset \mathcal{C}$. For $f=\xi_1$, when $g \geq 4, n \geq 2$ let $L_f = \{a_1, c_1, c_2, \cdots, c_{g-2}, a_g, r_{g-2,g-1}, b_{g-1,g+1}, b_{g,g+2}, \cdots,$ $ b_{g+n-3,g+n-1}\} \subset \mathcal{C}$. The set $L_f$ has trivial pointwise stabilizer. Since  
$\xi_1(a_1)= a_{1,g+n-1}$ and the other curves in $L_f$ are fixed by $\xi_1$, we have $f(L_f) \subset \mathcal{C}$. For $f=\xi_1$, when $g \geq 4, n = 1$ let $L_f = \{a_1, c_1, c_2, \cdots, c_{g-2}, a_g, r_{g-2,g-1}\} \subset \mathcal{C}$. The set $L_f$ has trivial pointwise stabilizer. Since  
$\xi_1(a_1)= a_{1,g}$ and the other curves in $L_f$ are fixed by $\xi_1$, we have $f(L_f) \subset \mathcal{C}$.

For $f=\xi_2$, when $g = 2$ let $L_f = \{a_{2,2+n}, w_{2+n}, b_{2,4}, b_{3,5}, \cdots, b_{n-1, n+1}\} \subset \mathcal{C}$. The set $L_f$ has trivial pointwise stabilizer. Since $\xi_2(a_{2,2+n})= a_{2,1+n}$ and the other curves in $L_f$ are fixed by $\xi_2$, we have $f(L_f) \subset \mathcal{C}$. For $f=\xi_2$, when $g \geq 4, n \geq 3$ let $L_f = \{a_{g,g+n}, b_{g+n,2}, b_{1,3}, b_{2,4}, \cdots, b_{g-3, g-1}, b_{g+n-1,g+n-3}, $ 
$b_{g+n-2,g+n-4}, \cdots, b_{g+2,g}, w_{g+n}\} \subset \mathcal{C}$. The set $L_f$ has trivial pointwise stabilizer. Since $\xi_2(a_{g,g+n})= a_{g,g+n-1}$ and the other curves in $L_f$ are fixed by $\xi_2$, we have $f(L_f) \subset \mathcal{C}$. For $f=\xi_2$, when $g \geq 4, n=2$ let $L_f = \{a_{g,g+2}, b_{g+2,2}, b_{1,3}, b_{2,4}, \cdots,$ $ b_{g-3, g-1}, w_{g+2}\} \subset \mathcal{C}$. The set $L_f$ has trivial pointwise stabilizer. Since $\xi_2(a_{g,g+2})= a_{g,g+1}$ and the other curves in $L_f$ are fixed by $\xi_2$, we have $f(L_f) \subset \mathcal{C}$.
For $f=\xi_2$, when $g \geq 4, n=1$ let $L_f = \{a_{g,g+1}, b_{g+1,2}, b_{1,3}, b_{2,4}, \cdots,$ $ b_{g-3,g-1}, w_2, c_{g-1}\} \subset \mathcal{C}$. The set $L_f$ has trivial pointwise stabilizer. Since $\xi_2(a_{g,g+1})= a_g$ and  $\xi_2(c_{g-1})= s_{g-1,g}$ the other curves in $L_f$ are fixed by $\xi_2$, we have $f(L_f) \subset \mathcal{C}$.\end{proof}  
 
\begin{theorem} \label{thm-4} If $g + n \geq 5$ and $g$ is even, then there exists a sequence $\mathcal{Z}_1 \subset \mathcal{Z}_2 \subset \cdots \subset \mathcal{Z}_n \subset \cdots$  such that $\mathcal{Z}_i$ is a finite superrigid set in $\mathcal{C}(N)$, $\mathcal{Z}_i$ has trivial pointwise stabilizer in $Mod_N$ for all $i$ and $\bigcup_{i \in \mathbb{N}} \mathcal{Z}_i = \mathcal{C}(N)$. \end{theorem}

\begin{proof} For each vertex $x$ in the curve complex, there 
exists $r \in Mod_N$ and a vertex $y$ in the set $\mathcal{C}$ such that $r(y)=x$. By following the construction given in \cite{IrP1}, we let
$\mathcal{Z}_1 = \mathcal{C}$ and
$\mathcal{Z}_k = \mathcal{Z}_{k-1} \cup (\bigcup _{f \in G} (f(\mathcal{Z}_{k-1}) \cup f^{-1}(\mathcal{Z}_{k-1})))$ when $k \geq 2$. We observe that 
$\bigcup _{k=1} ^{\infty} \mathcal{Z}_k$ is the set of all vertices in  $\mathcal{C}(N)$. The set $\mathcal{Z}_k$ is a finite set with trivial pointwise stabilizer for all $k \geq 1$. We will prove that $\mathcal{Z}_k$ is superrigid for every $k$. We will give the proof by induction on $k$. 
We know that $\mathcal{Z}_1$ is superrigid. Assume that $\mathcal{Z}_{k-1}$ is superrigid for some $k \geq 2$. Let $\lambda: \mathcal{Z}_k \rightarrow \mathcal{C}(N)$ be a superinjective simplicial map. The map $\lambda$ restricts to a superinjective simplicial map on $\mathcal{Z}_{k-1}$. Since $\mathcal{Z}_{k-1}$ is superrigid, there exists a homeomorphism $h$ of $N$ such that $h([x]) = \lambda([x])$ for all $x \in \mathcal{Z}_{k-1}$. Let $f \in G$. Since $\mathcal{Z}_{k-1}$ is superrigid, $f(\mathcal{Z}_{k-1})$ is also superrigid. So, there exists a homeomorphism $h_f$ of $N$ such that $h_f([x]) = \lambda([x])$ for all $x \in f(\mathcal{Z}_{k-1})$. There exists $L_f \subset \mathcal{C}$ such that $L_f$ has trivial pointwise stabilizer and $f(L_f) \subset \mathcal{C}$ by Lemma \ref{L_f_even}. So,  $L_f \subset \mathcal{C} \subset \mathcal{Z}_{k-1}$ and $f(L_f) \subset \mathcal{C} \subset \mathcal{Z}_{k-1}$.
We have $f(L_f) \subset \mathcal{Z}_{k-1} \cap f(\mathcal{Z}_{k-1})$. This implies that $h_f = h$ since $f(L_f)$ has trivial pointwise stabilizer. Similarly, there exists a homeomorphism 
$h'_f$ of $N$ such that $h'_f([x]) = \lambda([x])$ for all $x \in f^{-1}(\mathcal{Z}_{k-1})$. We have $L_f \subset \mathcal{Z}_{k-1} \cap f^{-1}(\mathcal{Z}_{k-1})$. This implies that $h'_f = h$ since $L_f$ has trivial pointwise stabilizer. So, $h([x]) = \lambda([x])$ for each $x \in \mathcal{Z}_k$. This shows that $\mathcal{Z}_k$ is a superrigid set. Hence, by induction  $\mathcal{Z}_k$ is a finite superrigid set for all $k \geq 1$.\end{proof}\\
  
  \begin{figure}[t]
  	\begin{center}
  		\epsfxsize=2.9in \epsfbox{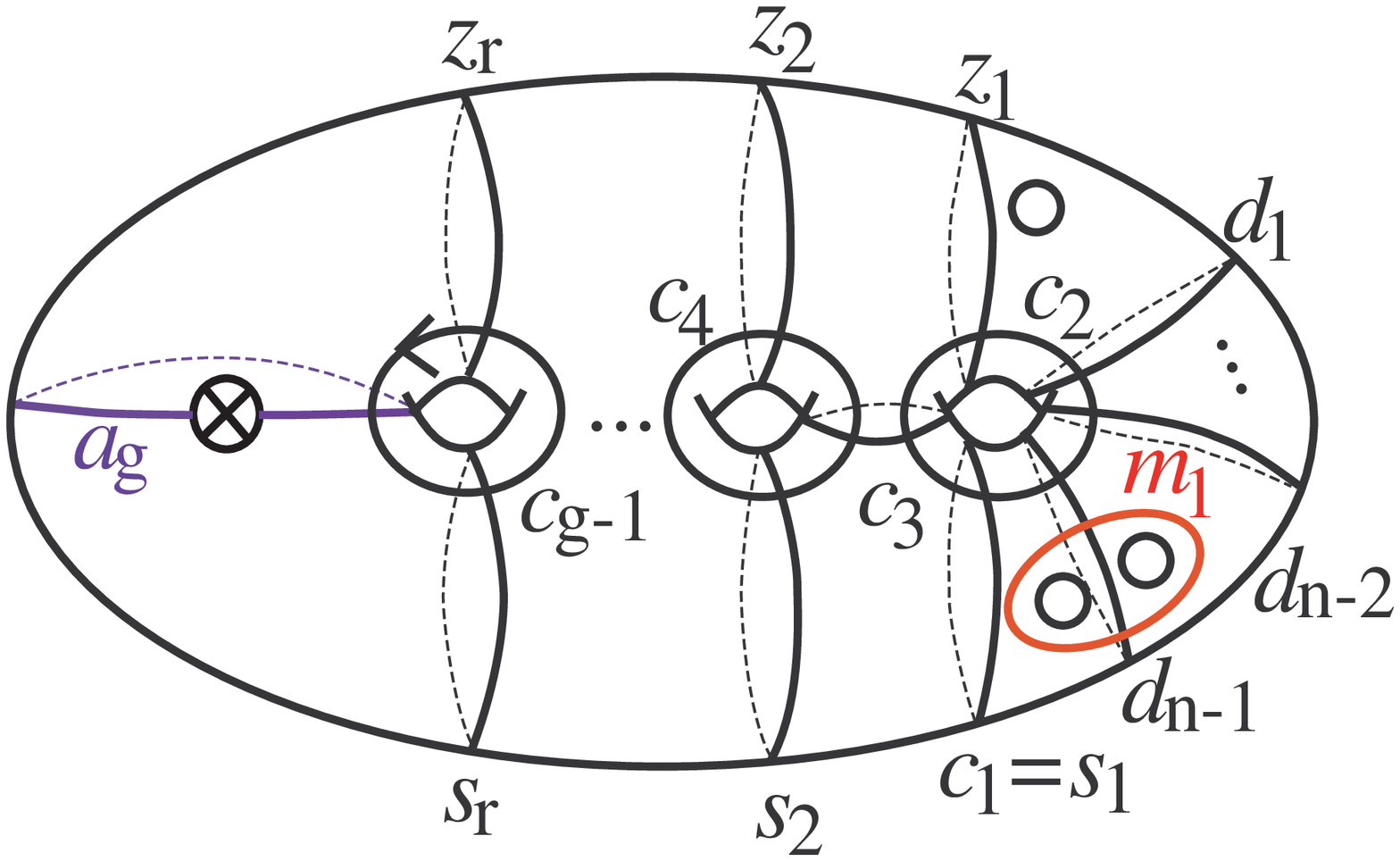} \hspace{-0.1cm} \epsfxsize=2.9in \epsfbox{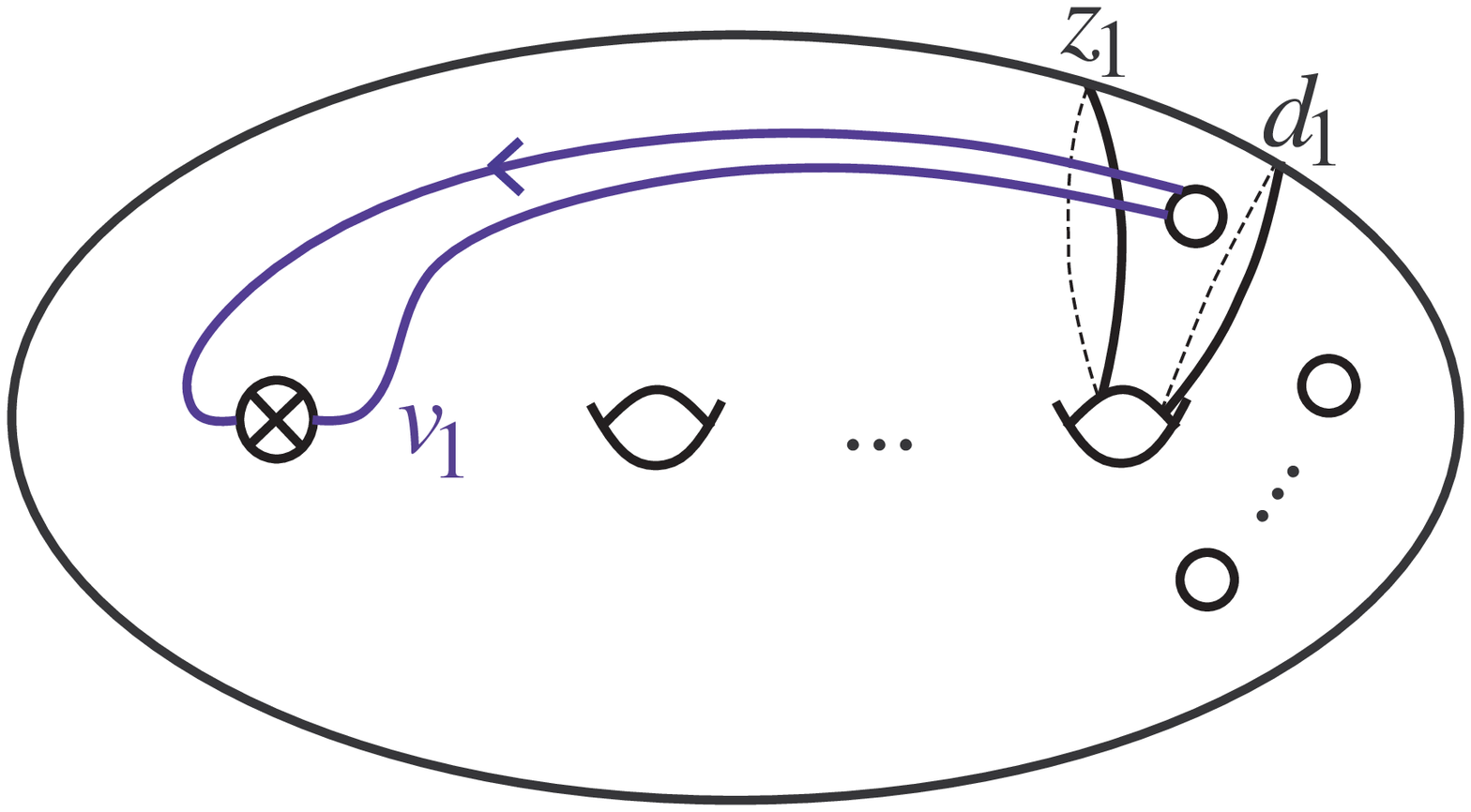}
  		
  		\hspace{-1cm} (i) \hspace{6.5cm} (ii)
  		
  		\hspace{0.71cm} \epsfxsize=3.2in \epsfbox{ext-fig114.eps}
  		
  		(iii) 
  		
  		\caption{Generators for odd genus $g=2r+1$, $g \geq 3$} \label{Fig-13}
  	\end{center}
  \end{figure}

\begin{figure}
	\begin{center}
		\epsfxsize=2.55in \epsfbox{ext-fig109.eps}   \hspace{0.1cm} 
		\epsfxsize=2.55in \epsfbox{ext-fig66.eps}  
		
		(i)  \hspace{6cm}   (ii) 
		
		\epsfxsize=2.55in \epsfbox{ext-fig67.eps}  \hspace{0.1cm}  	\epsfxsize=2.55in \epsfbox{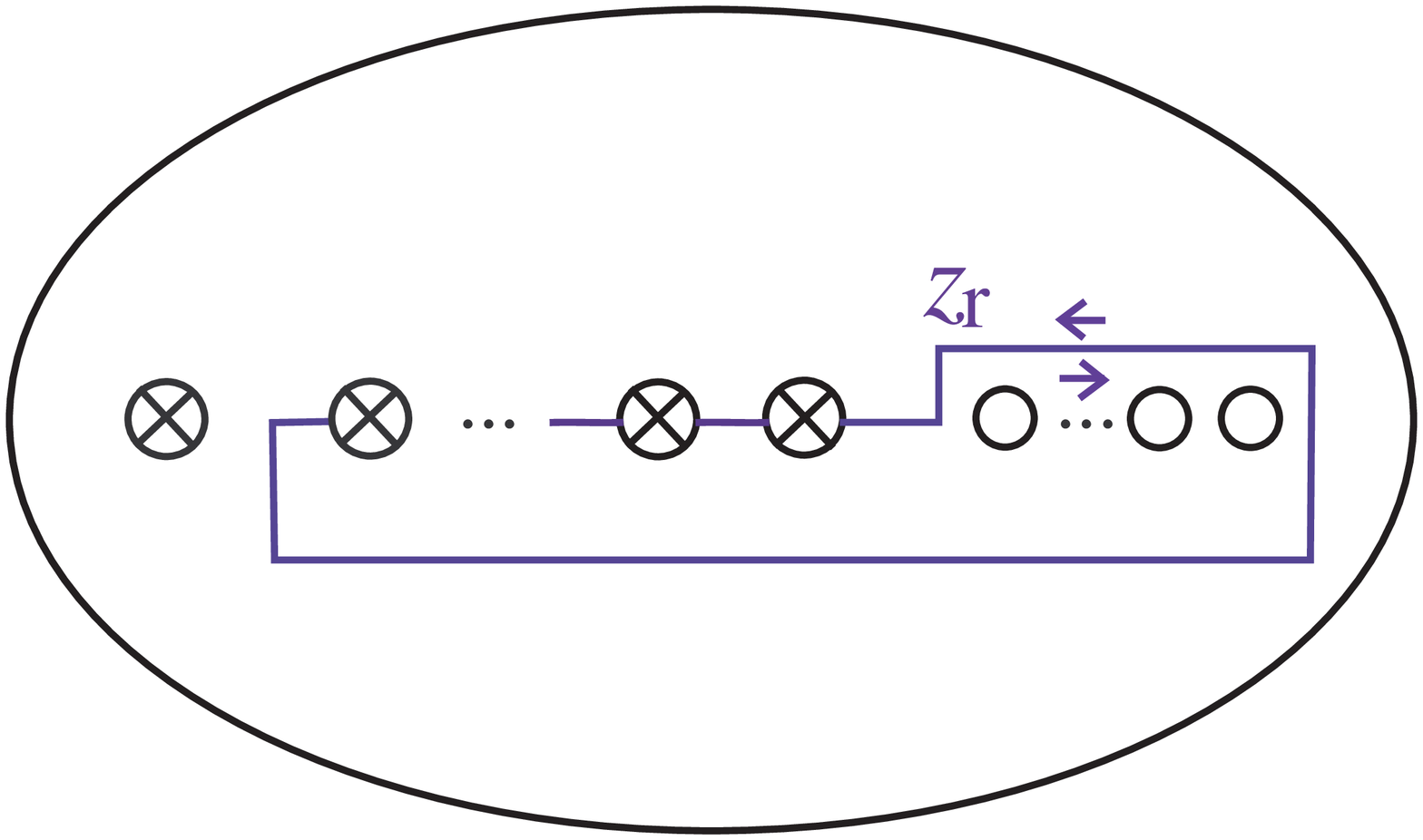}  
		
		(iii)  \hspace{6cm}   (iv) 
		
		\epsfxsize=2.55in \epsfbox{ext-fig70.eps}  \hspace{0.1cm}  	\epsfxsize=2.55in \epsfbox{ext-fig69.eps}  
		
		(v)  \hspace{6cm}   (vi) 
		
		\epsfxsize=2.55in \epsfbox{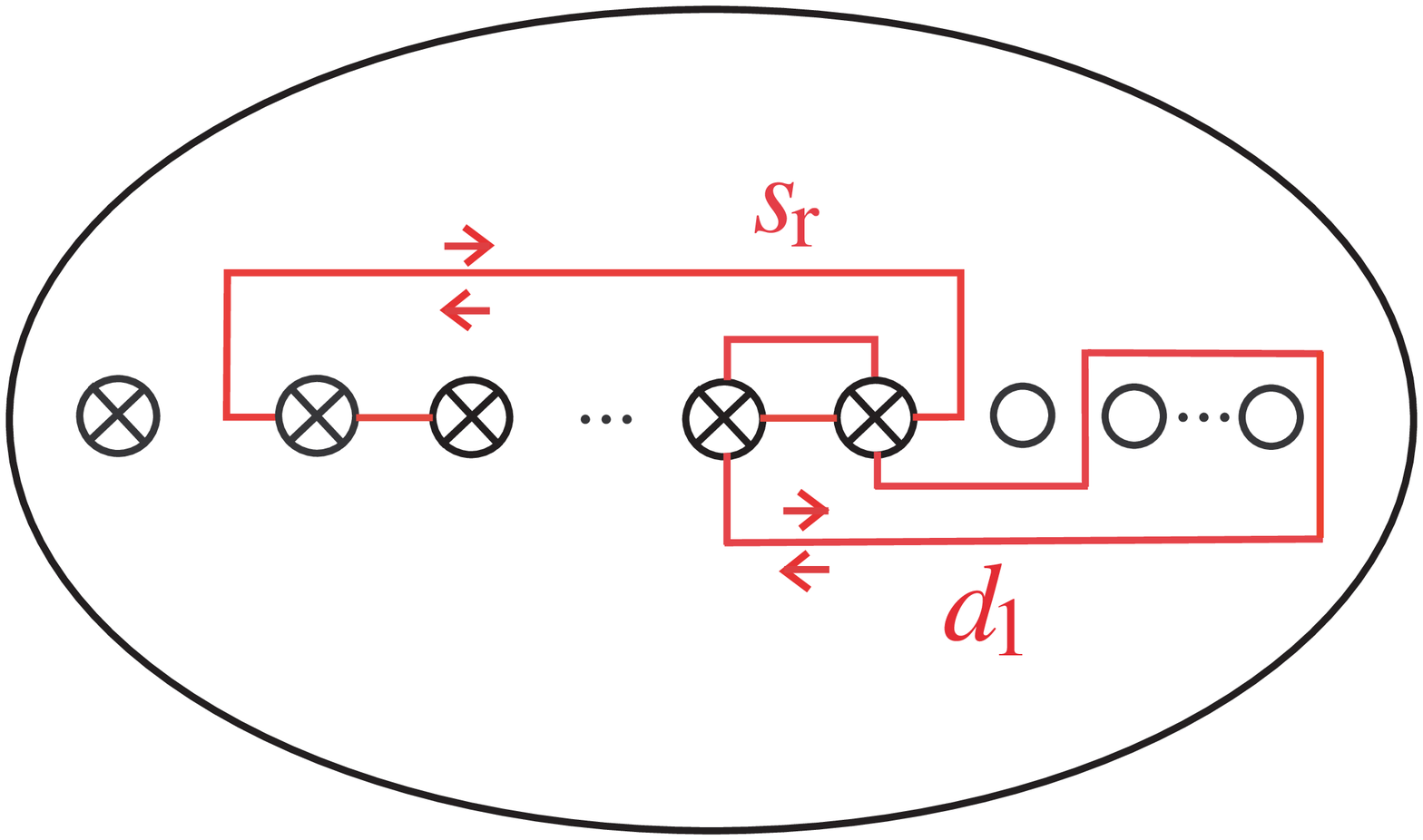}  \hspace{0.1cm} \epsfxsize=2.55in \epsfbox{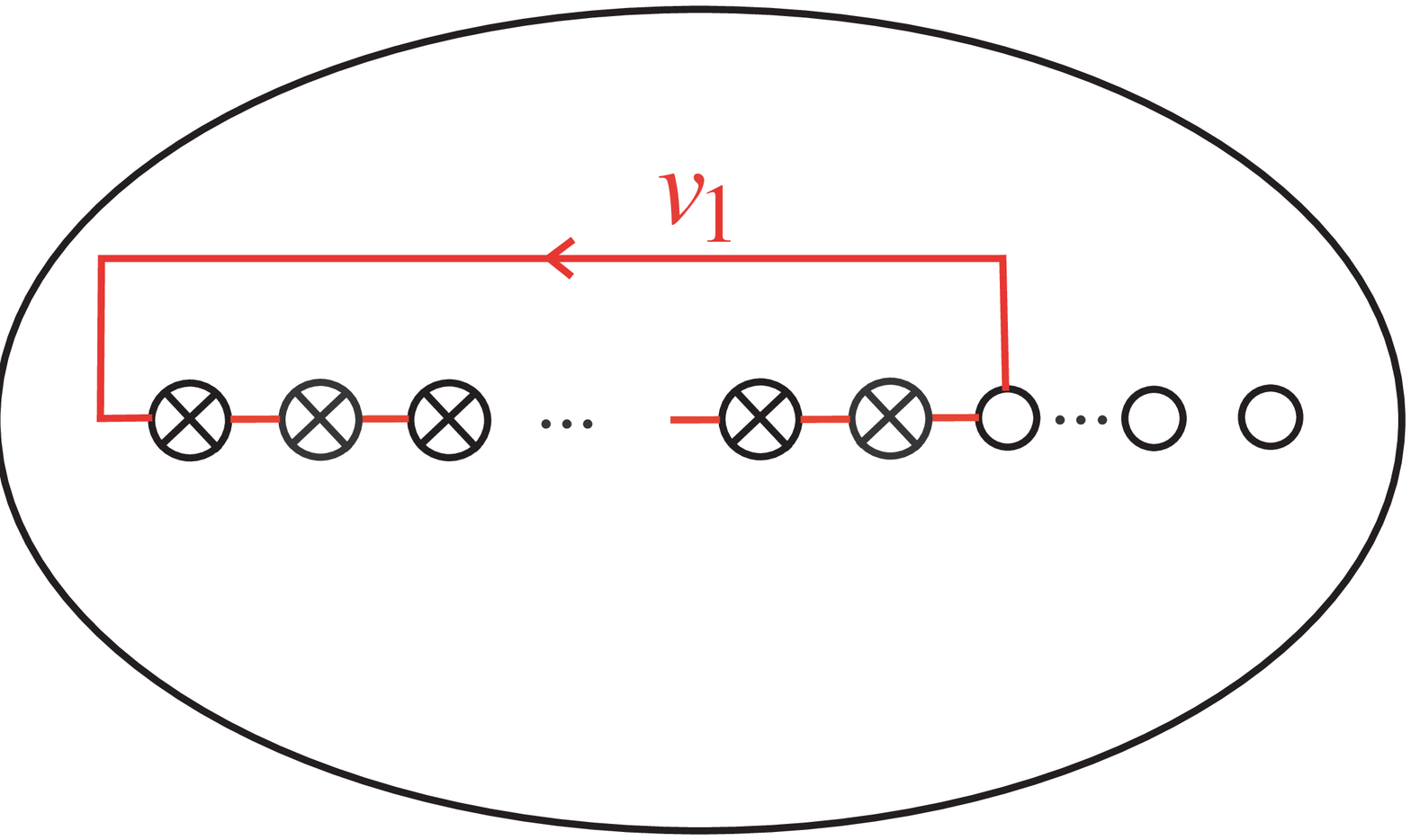}
		
		(vii)  \hspace{6cm}   (viii)   
		\caption{Generators for odd genus, $g \geq 3$}
		\label{Fig-12}
	\end{center}
\end{figure}  

In the following theorem we consider the curves given in Figure \ref{Fig-12}. 
Let $\sigma_i$ denote the half twist along the curve $m_i$. Let $y= y_{a_g, c_{g-1}}$ be the crosscap slide of $a_g$ along $c_{g-1}$.  Let $\xi_1$ be the boundary slide of the boundary component $t_{g+n}$ along the arc $v_1$. 

\begin{theorem}\label{gen-odd} If $g \geq 3$ and $g=2r+1$ for some $r \in \mathbb{Z}^+$, then $Mod_N$ is generated by $\{t_x: x \in \{c_1, c_2, \cdots, c_{g-1}, s_1, s_2, \cdots, s_r, z_1, z_2, \cdots, z_r, d_1, d_2, \cdots, d_{n-1}\} \} \cup \{\sigma_1, \sigma_2, \cdots,$ $ \sigma_{n-1}, y,  \xi_1\}$.\end{theorem}

\begin{proof} The proof follows from the proof of Theorem 4.14 given by Korkmaz in \cite{K2} by representing the curves used in his generating set given in Figure \ref{Fig-13}, in our crosscap model of $N$ as shown in Figure \ref{Fig-12} (see section 3 in Stukow's paper \cite{St} to see some part of this transfer).\end{proof}

The proof of the following theorem was given by Korkmaz (see Theorem 4.4 in  \cite{K2}). Consider Figure \ref{Fig-09} (i),  $\sigma_i$ denotes the half twist along the curve $m_i$ and $\xi_1$ is the boundary slide along the arc $v_1$.

\begin{theorem}\label{theorem-1} If $g = 1, n \geq 1$, then $Mod_N$ is generated by $\{\sigma_1, \sigma_2, \cdots, \sigma_{n-1},  \xi_1\}$.\end{theorem}

If $g = 1, n \geq 1$, let $G_2=\{\sigma_1, \sigma_2, \cdots, \sigma_{n-1},  \xi_1\}$. If $g \geq 3$ and $g=2r+1$ for some $r \in \mathbb{Z}^+$, let $G_2= \{t_x: x \in \{c_1, c_2, \cdots, c_{g-1}, s_1, 
s_2, \cdots, $ $s_r, z_1, z_2, \cdots, z_r, d_1, d_2, \cdots, d_{n-1}\} \} \cup \{\sigma_1, \sigma_2, \cdots, \sigma_{n-1}, y, \xi_1\}$.

\begin{lemma} \label{L_f_odd} If $g + n \geq 5$ and $g$ is odd, then $\forall \ f \in G_2$, $\exists$ a set $L_f \subset \mathcal{C}$ such that $L_f$ has trivial pointwise stabilizer and $f(L_f) \subset \mathcal{C}$.\end{lemma}

\begin{proof} Suppose $g + n \geq 5$ and $g$ is odd. Let $f \in G_2$. For $f=t_{c_1}$, let $L_f = \{b_{2,4}, b_{3,5} \cdots, b_{g+n-2,g+n}, c_1, b_{g+n,2}, a_{1,g+n-1}, e_{1,2}\}$. The set $L_f$ has trivial pointwise stabilizer. Since $t_{c_1}(a_{1,g+n-1}) = a_{1,g+n-2}$ and the other curves in $L_f$ are fixed by $t_{c_1}$, we have $f(L_f) \subset \mathcal{C}$. Similarly we get the result for every $t_{c_i}$ when $i=2, 3, \cdots, n-1$.
	
For $f=t_{d_1}$, let $L_f = \{b_{2,4}, b_{3,5} \cdots, b_{g-2,g}, b_{g,g+2}, b_{g+1,g+3}, \cdots, b_{g+n-3,g+n-1}, d_{n-1}, d_{1,2},$ $ s_{1,2}, c_1,  a_{1,g+n-1}\}$. The set $L_f$ has trivial pointwise stabilizer. Since $t_{d_1}(a_{1,g+n-1}) = a_{2,g}$ and the other curves in $L_f$ are fixed by $t_{d_1}$, we have $f(L_f) \subset \mathcal{C}$. For $f=t_{d_2}$, let $L_f = \{b_{2,4}, b_{3,5} \cdots, b_{g-2,g}, b_{g,g+2}, b_{g+1,g+3}, \cdots, b_{g+n-4,g+n-2}, d_{n-1}, d_{1,2}, s_{1,2}, c_1, a_{1,g+n-2}\}$. The set $L_f$ has trivial pointwise stabilizer. Since $t_{d_2}(a_{1,g+n-2}) = a_{2,g}$ and the other curves in $L_f$ are fixed by $t_{d_2}$, we have $f(L_f) \subset \mathcal{C}$. Similarly we get the result for every $t_{d_i}$ when $i=3, 4, \cdots, n-1$.
	
We know the result for $t_{s_1}$ since $s_1=c_1$. For $f=t_{s_2}$, let $L_f = \{b_{4,6}, b_{5,7}, \cdots, $ $b_{g+n-2,g+n}, $ $e_{1,2}, r_{3,4}, c_1, c_2, c_3, a_{4,g+n}\}$. The set $L_f$ has trivial pointwise stabilizer. Since $t_{s_2}(a_{4,g+n}) = q_2$ and the other curves in $L_f$ are fixed by $t_{s_2}$, we have $f(L_f) \subset \mathcal{C}$. 
Similarly we get the result for every $t_{s_i}$ when $i=3, 4, \cdots, r$.
	 
For $f=t_{z_1}$, if $n =0$ then let $L_f = \{b_{2,4}, b_{3,5} \cdots, b_{g-2,g}, s_{1,2}, c_1, a_1\}$. The set $L_f$ has trivial pointwise stabilizer. Since $t_{z_1}(a_1) = a_{2,g}$ and the other curves in $L_f$ are fixed by $t_{z_1}$, we have $f(L_f) \subset \mathcal{C}$. 	
For $f=t_{z_1}$, if $n \geq 1$ then let $L_f = \{b_{2,4}, b_{3,5} \cdots, b_{g-2,g}, b_{g,g+2}, b_{g+1,g+3}, \cdots, b_{g+n-2,g+n},$ $s_{1,2}, e_{1,2}, c_1, a_1\}$. The set $L_f$ has trivial pointwise stabilizer. Since $t_{z_1}(a_1) = a_{2,g}$ and the other curves in $L_f$ are fixed by $t_{z_1}$, we have $f(L_f) \subset \mathcal{C}$. 
For $f=t_{z_2}$, if $n=0$ then let $L_f = \{b_{4,6}, b_{5,7} \cdots, b_{g-2,g}, s_{3,4}, c_1, c_2,$ $ c_3, a_{4,g}\}$. The set $L_f$ has trivial pointwise stabilizer. Since $t_{z_2}(a_{4,g}) = o_2$ and the other curves in $L_f$ are fixed by $t_{z_2}$, we have $f(L_f) \subset \mathcal{C}$.
For $f=t_{z_2}$, if $n \geq 1$ then let $L_f = \{b_{4,6}, b_{5,7} \cdots, b_{g-2,g},$ $ b_{g,g+2},b_{g+1,g+3}, \cdots, b_{g+n-2,g+n},$ $e_{1,2}, s_{3,4}, c_1, c_2, c_3, a_{4,g}\}$. The set $L_f$ has trivial pointwise stabilizer. Since $t_{z_2}(a_{4,g}) = o_2$ and the other curves in $L_f$ are fixed by $t_{z_2}$, we have $f(L_f) \subset \mathcal{C}$. Similarly we get the result for every $t_{z_i}$ when $i=3, 4, \cdots, r$.
	
For $f=\sigma_{n-1}$, let $L_f = \{b_{1,3}, b_{2,4}, \cdots, b_{g+n-2,g+n}, b_{g+n-1,1}, b_{g+n,2}, w_{i_o}\}$, where $w_{i_o}$ is an element in $\{w_1, w_2, \cdots, w_n\}$ disjoint from  $b_{g+n-2,g+n}$. The set $L_f$ has trivial pointwise stabilizer. Since $\sigma_{n-1}(b_{g+n-3,g+n-1})= r_{g+n-1}$,  $\sigma_{n-1}(b_{g+n-1,1})= w_{g+n}$ and the other curves in $L_f$ are fixed by $\sigma_{n-1}$, we have $f(L_f) \subset \mathcal{C}$. 
With similar construction, we get the result for every $\sigma_i$ when $i=1, 2, \cdots, n-2$.
	
For $f=y$, let $L_f = \{a_{g-1}, a_g, b_{g-2,g}, c_{g-1}, d_{g-1,g}, b_{1,3}, b_{2,4}, \cdots, b_{g-4,g-2}, b_{g,g+2}, b_{g+1,g+3},$ $ \cdots, b_{g+n-2,g+n}, b_{g+n-1,1}, b_{g+n,2} \} \subset \mathcal{C}$. The set $L_f$ has trivial pointwise stabilizer. Since $y(a_{g-1})=a_{g-1,g}$ and the other curves in $L_f$ are fixed by $y$, we have $f(L_f) \subset \mathcal{C}$. 
	
For $f=\xi_1$, when $g = 1$, let $L_f = \{a_1, b_{1,3}, b_{2,4}, \cdots,  b_{n-2,n}, u_1, w_{n-1}\} \subset \mathcal{C}$. The set $L_f$ has trivial pointwise stabilizer. Since  
$\xi_1(a_1)= a_{1,g+n-1}$ and $\xi_1(u_1)= a_{1,2}$ the other curves in $L_f$ are fixed by $\xi_1$, we have $f(L_f) \subset \mathcal{C}$. For $f=\xi_1$, when $g \geq 3, n \geq 3$, let $L_f = \{a_1, c_1, c_2, \cdots, c_{g-1},  r_{g-1,g}, b_{g,g+2}, b_{g+1,g+3}, \cdots,$ $ b_{g+n-3,g+n-1}\} \subset \mathcal{C}$. The set $L_f$ has trivial pointwise stabilizer. Since  
$\xi_1(a_1)= a_{1,g+n-1}$ and the other curves in $L_f$ are fixed by $\xi_1$, we have $f(L_f) \subset \mathcal{C}$. 
For $f=\xi_1$, when $g \geq 3, n = 2$, let $L_f = \{a_1, c_1, c_2, \cdots, c_{g-1},  r_{g-1,g }\} \subset \mathcal{C}$. The set $L_f$ has trivial pointwise stabilizer. Since $\xi_1(a_1)= a_{1,g+1}$ and the other curves in $L_f$ are fixed by $\xi_1$, we have $f(L_f) \subset \mathcal{C}$. 
For $f=\xi_1$, when $g \geq 3, n = 1$ (in this case actually we have $g \geq 5$ since $g$ is odd and $g+n \geq 5$) let $L_f = \{a_1, c_1, c_2, \cdots, c_{g-1}, l_0\} \subset \mathcal{C}$. The set $L_f$ has trivial pointwise stabilizer. Since $\xi_1(a_1)= a_{1,g}$ and the other curves in $L_f$ are fixed by $\xi_1$, we have $f(L_f) \subset \mathcal{C}$.\end{proof}

\begin{theorem} \label{B-odd} If $g + n \geq 5$ and $g$ is odd, then there exists a sequence $\mathcal{Z}_1 \subset \mathcal{Z}_2 \subset \cdots \subset \mathcal{Z}_n \subset \cdots$  such that $\mathcal{Z}_i$ is a finite superrigid set in $\mathcal{C}(N)$, $\mathcal{Z}_i$ has trivial pointwise stabilizer in $Mod_N$ for all $i$ and $\bigcup_{i \in \mathbb{N}} \mathcal{Z}_i = \mathcal{C}(N)$. \end{theorem}

\begin{proof} The proof is similar to the even genus case (use Lemma \ref{L_f_odd}).\end{proof}
   
\section{Exhaustion of $\mathcal{C}(N)$ by finite superrigid sets when $(g,n)=(3,0)$}

For $(g,n)=(3,0)$, let $\mathcal{C} = \{a_1, a_2, a_3, a_{1,2}, a_{2,3}, a_{3,1}, c_1, c_2, l, w, j\}$ where the curves are as shown in Figure \ref{fig-new}. Note that $l$ has orientable complement on $N$ and $\mathcal{C}$ has nontrivial curve of every topological type on $N$.

\begin{figure}
	\begin{center}
		\epsfxsize=1.8in \epsfbox{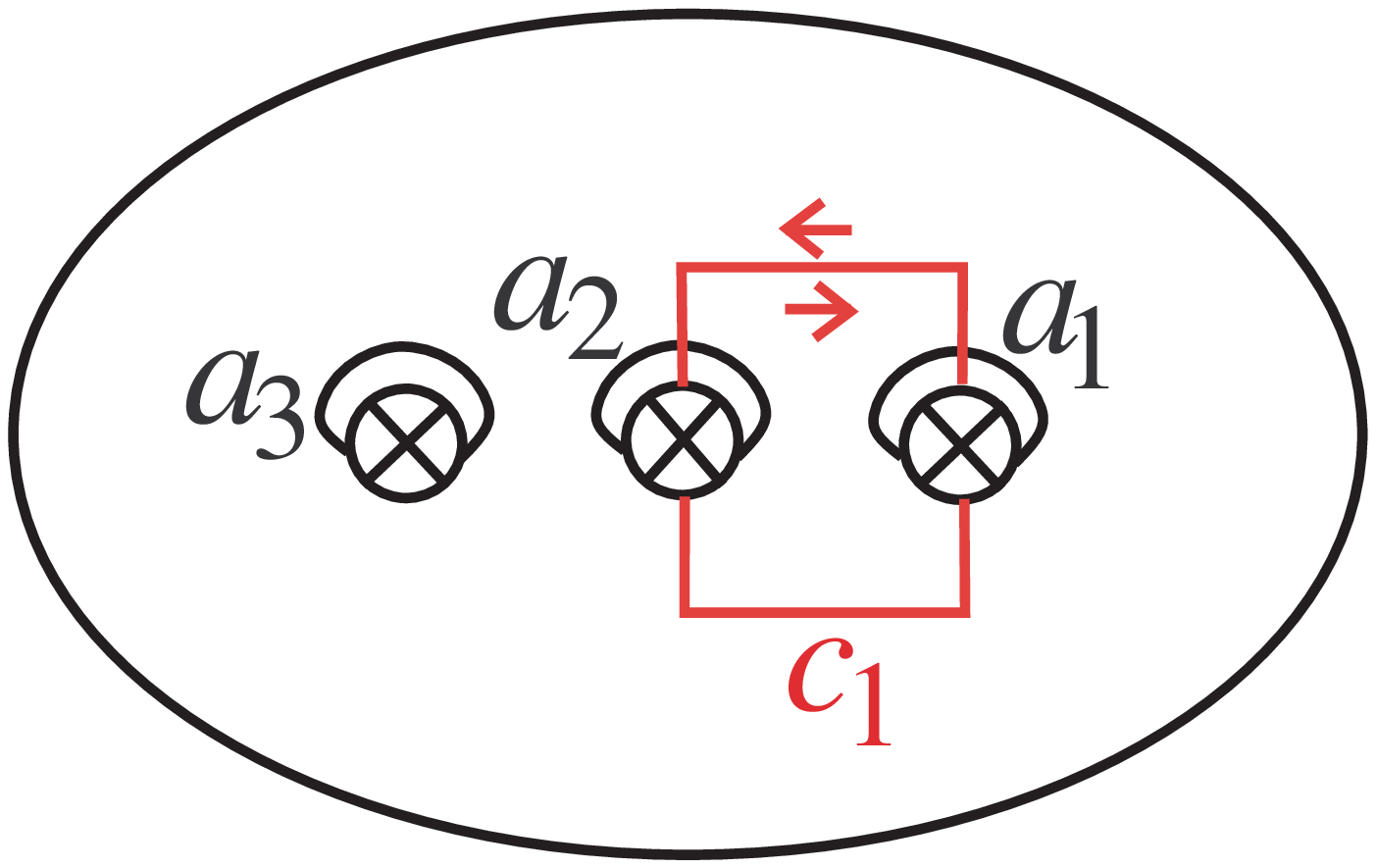}   \hspace{0.2cm}   	\epsfxsize=1.8in \epsfbox{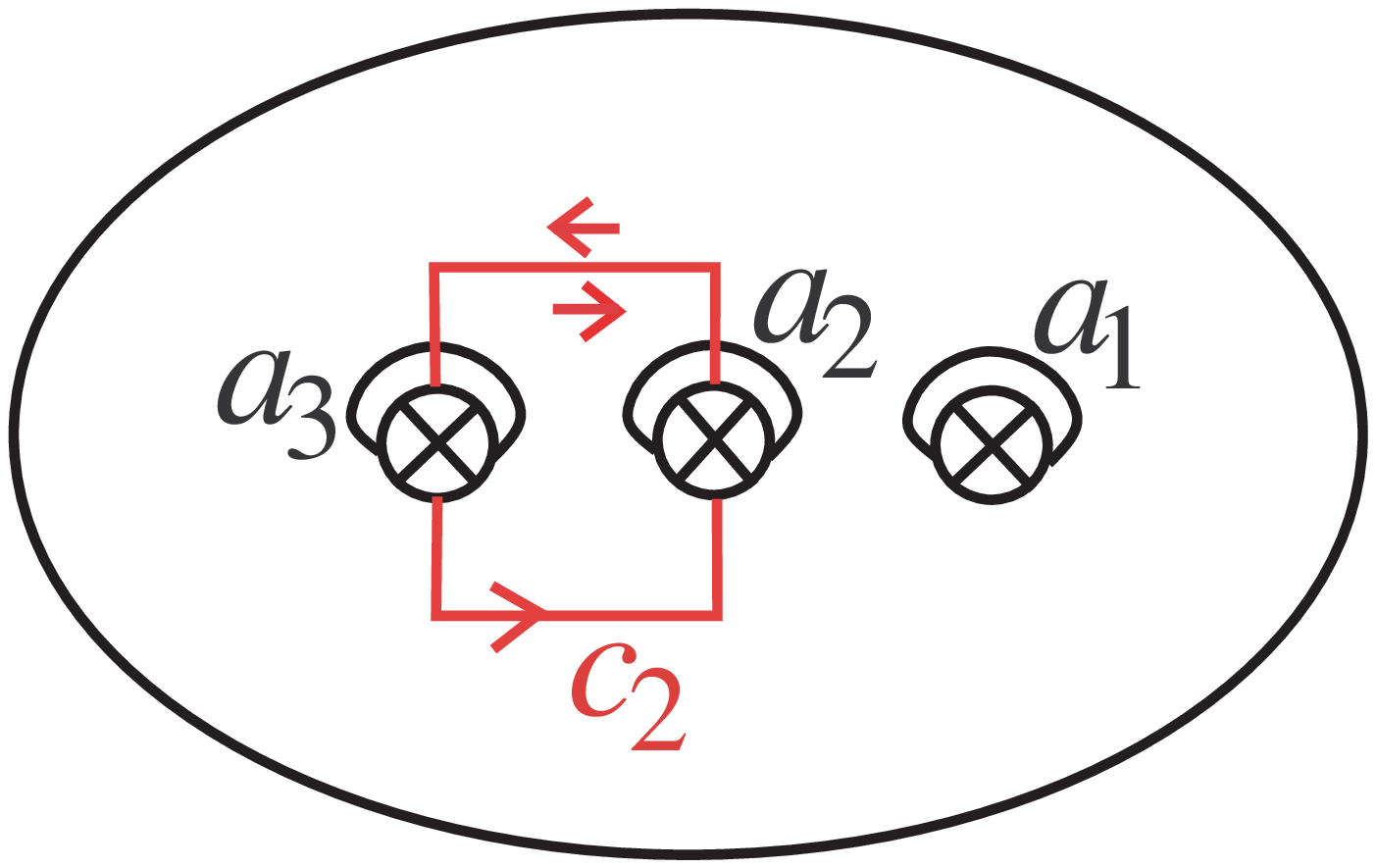}   \hspace{0.2cm}  	\epsfxsize=1.8in \epsfbox{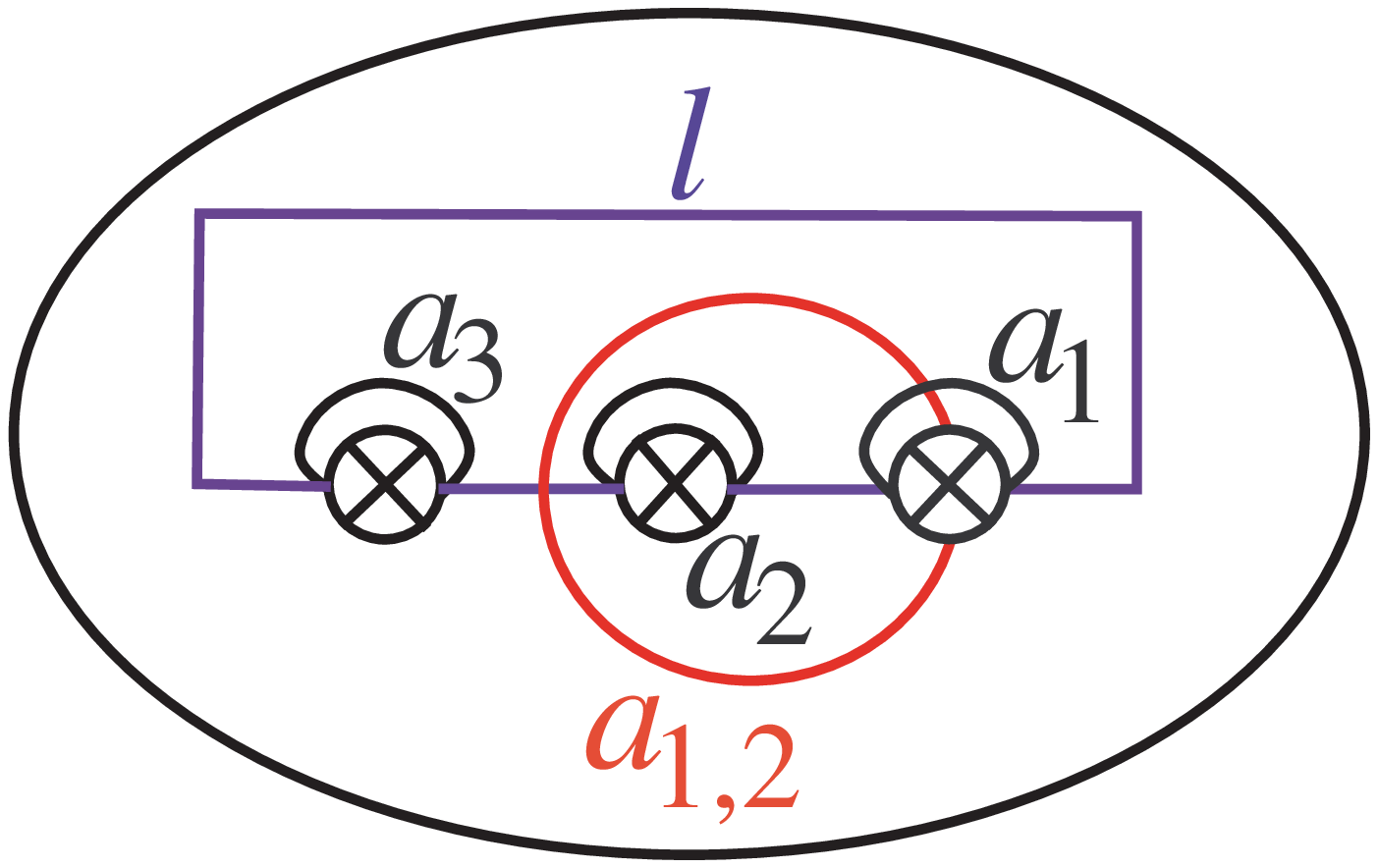}   
		
			(i)  \hspace{4cm}   (ii)  \hspace{4cm}   (iii) 
			
		\epsfxsize=1.8in \epsfbox{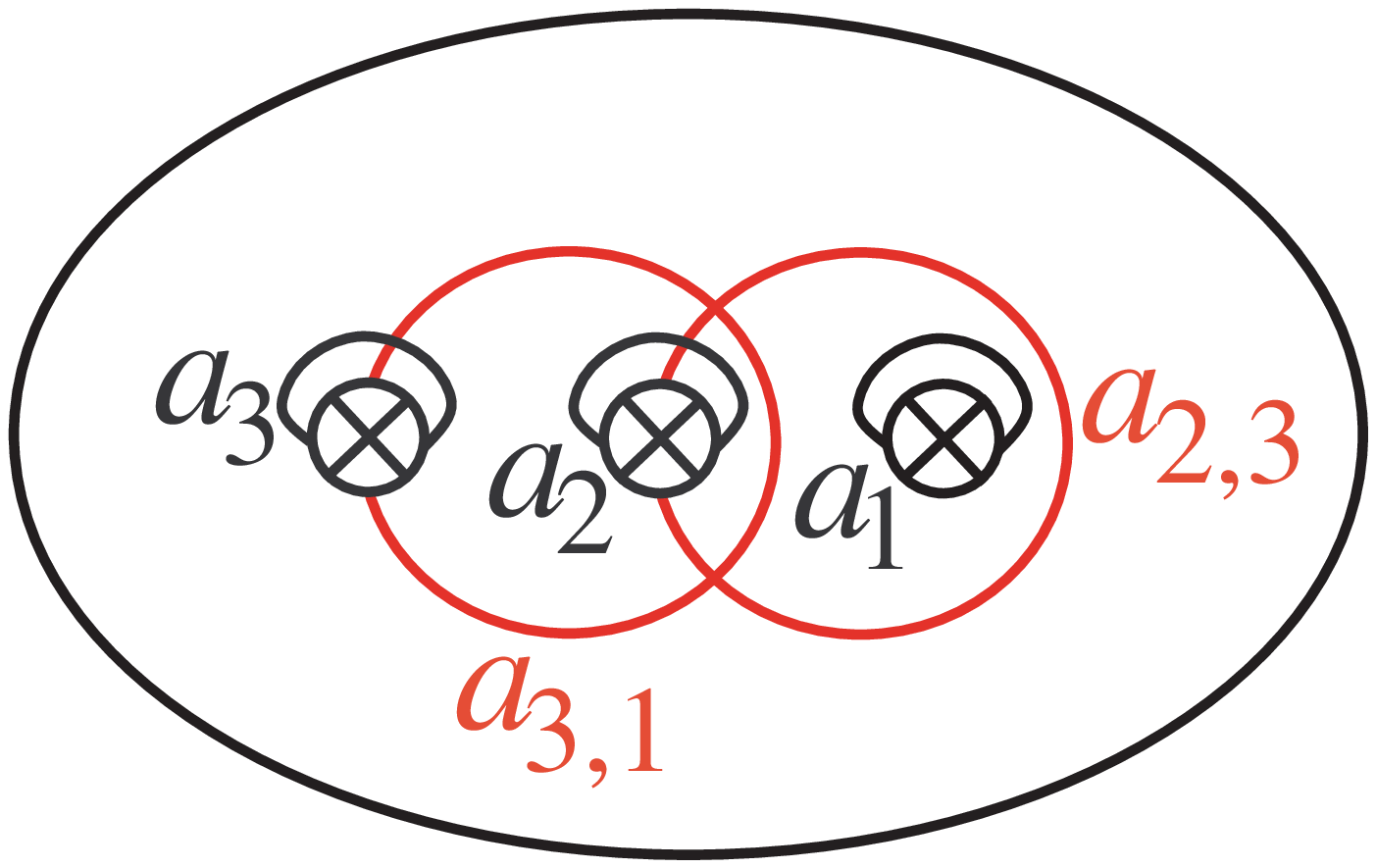}   \hspace{0.2cm}   			\epsfxsize=1.8in \epsfbox{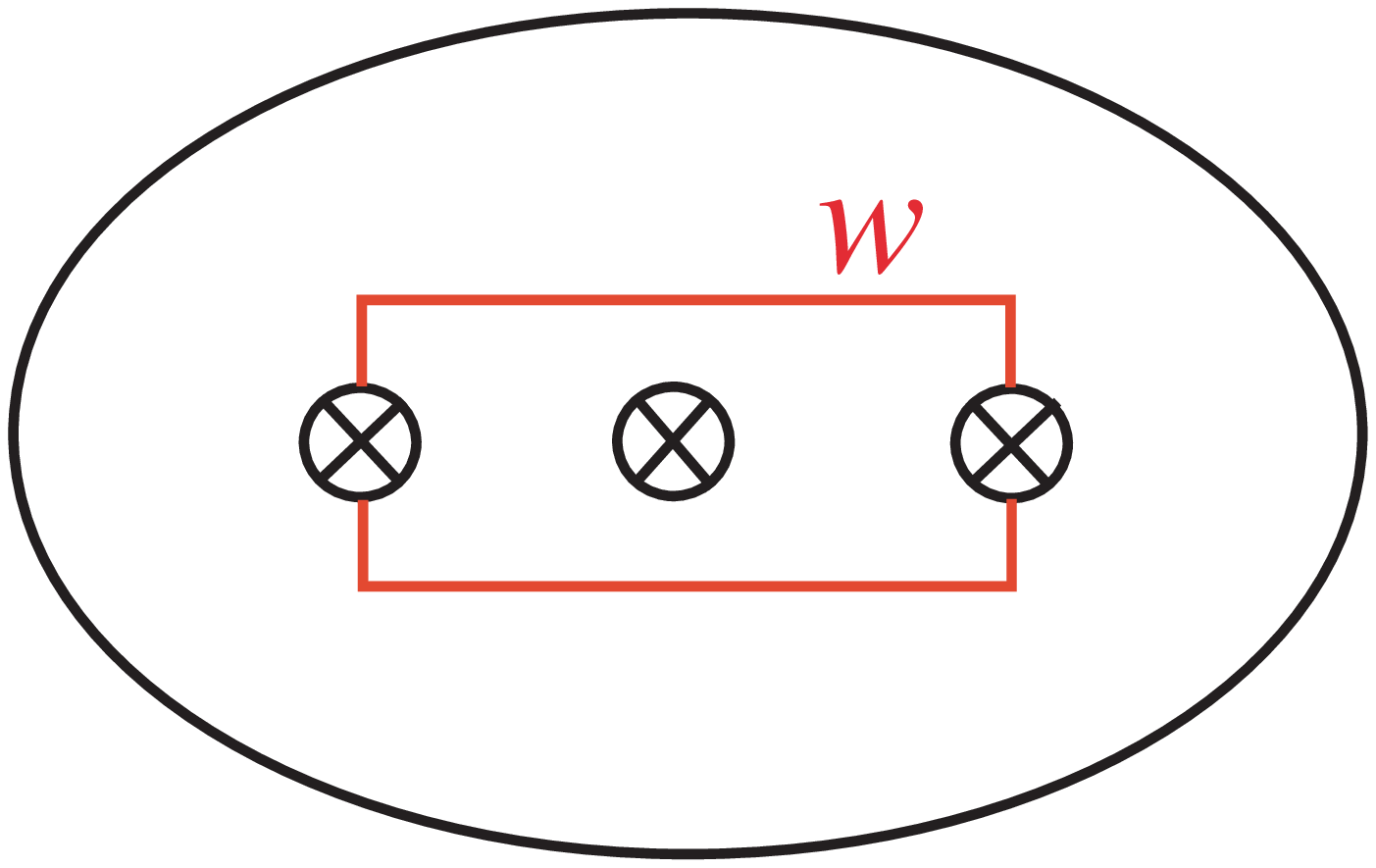}   \hspace{0.2cm}   	\epsfxsize=1.8in \epsfbox{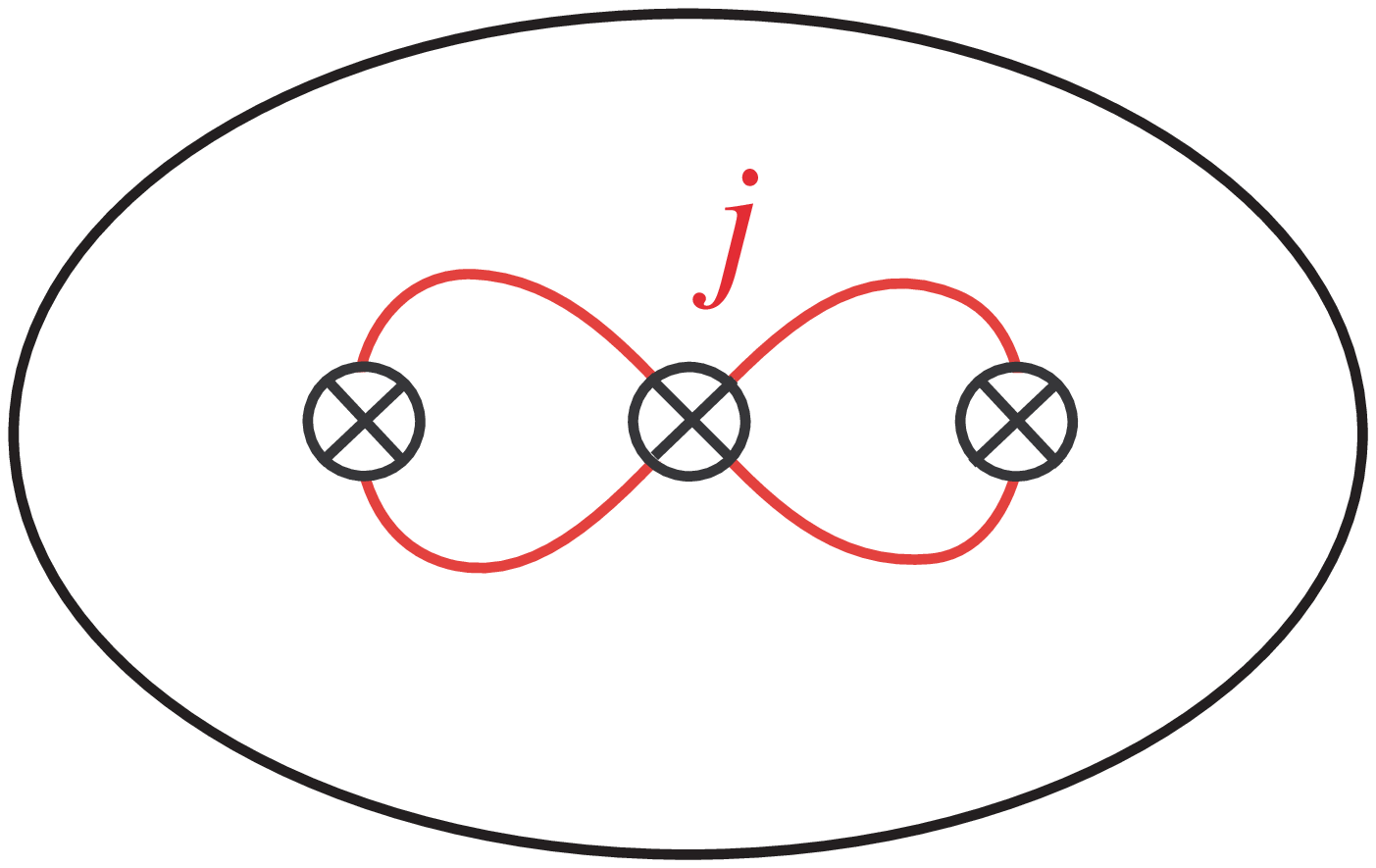}  
			
			(iv)  \hspace{4cm}   (v)  \hspace{4cm}   (vi) 
	\caption{$(g,n)=(3,0)$}
\label{fig-new}
\end{center}
\end{figure} 

\begin{lemma} \label{C_5_g3} If $(g,n)=(3,0)$, then $\mathcal{C}$ is a finite superrigid set.\end{lemma}

\begin{proof} Let $\lambda: \mathcal{C} \rightarrow \mathcal{C}(N)$ be a superinjective simplicial map. Since $a_1, a_2, a_3$ give a top dimensional pants decomposition on $N$, $\lambda([a_1]), \lambda([a_2]), \lambda([a_3])$ corresponds to a top dimensional pants decomposition on $N$. Since $(g,n)= (3,0)$, this implies that $\lambda([a_i])$ is the isotopy class of a 1-sided curve for each $i$. So, there exists a homeomorphism $h$ such that $h([x])=\lambda([x])$ for each $x \in \{a_1, a_2, a_3\}$. Since superinjective simplicial maps send 2-sided curves to 2-sided curves and 
$c_1$ is the unique nontrivial 2-sided curve up to isotopy disjoint from $a_3$, we see that $h([c_1])=\lambda([c_1])$. Since $c_2$ is the unique nontrivial 2-sided curve up to isotopy disjoint from $a_1$, we have $h([c_2])=\lambda([c_2])$. Since $a_{1,2}$ is the unique 1-sided curve up to isotopy disjoint from and nonisotopic to $a_2, a_3$ and nonisotopic to $a_1$, and $\lambda$ preserves these properties we have  $h([a_{1,2}])=\lambda([a_{1,2}])$. Similarly, we can see that $h([a_{2,3}])=\lambda([a_{2,3}])$ and $h([a_{3,1}])=\lambda([a_{3,1}])$. Since $w$ is the unique nontrivial 2-sided curve up to isotopy disjoint from $a_2$, we have $h([w])=\lambda([w])$. Since $l$ is the unique 1-sided curve up to isotopy disjoint from $c_1, c_2$, we see that $h([l])=\lambda([l])$. Since $j$ is the unique nontrivial 2-sided curve up to isotopy disjoint from $a_{2,3}$, we have $h([j])=\lambda([j])$. So, $\mathcal{C}$ is a finite superrigid set.\end{proof}\\
  
When $(g,n)=(3,0)$, the generating set for the mapping class group is $G_3= \{t_{c_1}, t_{c_2}, y\}$ where the maps are as defined in section 3.
 
\begin{lemma} \label{L_f_odd_g3} If $(g,n)=(3,0)$, then $\forall \ f \in G_3$, $\exists$ a set $L_f \subset \mathcal{C}$ such that the pointwise stabilizer of $L_f$ is the center of $Mod_N$ and $f(L_f) \subset \mathcal{C}$.\end{lemma}

\begin{proof} Let $f \in G_3$. For $f=c_1$, let $L_f = \{c_1, c_2, w\}$. The pointwise stabilizer of $L_f$ is the center of $Mod_N$ and it is generated by an involution, see Stukow's Theorem 6.1 in \cite{St2}. We have $t_{c_1}(c_1)= c_1$, $t_{c_1}(c_2)= j$, $t_{c_1}(w)= c_2$. So, $f(L_f) \subset \mathcal{C}$. For $f=c_2$, let $L_f = \{c_1, c_2, j\}$. The pointwise stabilizer of $L_f$ is the center of $Mod_N$. We have $t_{c_2}(c_1)= w$, $t_{c_2}(c_2)= c_2$, $t_{c_2}(j)= c_1$. So, $f(L_f) \subset \mathcal{C}$. For $f=y$, let $L_f = \{c_1, c_2, w\}$. We have $y(c_1)= c_1$, $y(c_2)= c_2$, $y(w)=j$. So, $f(L_f) \subset \mathcal{C}$.\end{proof}

\begin{theorem} \label{B-odd-g3} If $(g,n)=(3,0)$, then there exists a sequence $\mathcal{Z}_1 \subset \mathcal{Z}_2 \subset \cdots \subset \mathcal{Z}_n \subset \cdots$  such that $\mathcal{Z}_i$ is a finite superrigid set in $\mathcal{C}(N)$ for all $i$ and $\bigcup_{i \in \mathbb{N}} \mathcal{Z}_i = \mathcal{C}(N)$.\end{theorem}

\begin{proof} For each vertex $x$ in the curve complex, there 
exists $r \in Mod_N$ and a vertex $y$ in the set $\mathcal{C}$ such that $r(y)=x$. By following the construction given in \cite{IrP1}, we let
$\mathcal{Z}_1 = \mathcal{C}$ and
$\mathcal{Z}_k = \mathcal{Z}_{k-1} \cup (\bigcup _{f \in G} (f(\mathcal{Z}_{k-1}) \cup f^{-1}(\mathcal{Z}_{k-1})))$ when $k \geq 2$. We observe that 
$\bigcup _{k=1} ^{\infty} \mathcal{Z}_k$ is the set of all vertices in  $\mathcal{C}(N)$. We will prove that $\mathcal{Z}_k$ is superrigid for every $k$. We will give the proof by induction on $k$. 
We know that $\mathcal{Z}_1$ is superrigid by Lemma \ref{C_5_g3}. Assume that $\mathcal{Z}_{k-1}$ is a superrigid set for some $k \geq 2$. Let $\lambda: \mathcal{Z}_k \rightarrow \mathcal{C}(N)$ be a superinjective simplicial map. The map $\lambda$ restricts to a superinjective simplicial map on $\mathcal{Z}_{k-1}$. Since $\mathcal{Z}_{k-1}$ is a superrigid set, there exists a homeomorphism $h$ of $N$ such that $h([x]) = \lambda([x])$ for all $x \in \mathcal{Z}_{k-1}$. Let $f \in G$. Since $\mathcal{Z}_{k-1}$ is superrigid, $f(\mathcal{Z}_{k-1})$ is also superrigid. So, there exists a homeomorphism $h_f$ of $N$ such that $h_f([x]) = \lambda([x])$ for all $x \in f(\mathcal{Z}_{k-1})$. There exists $L_f \subset \mathcal{C}$ such that the pointwise stabilizer of $L_f$ is the center of $Mod_N$ and $f(L_f) \subset \mathcal{C}$ by Lemma \ref{L_f_odd_g3}. So, $L_f \subset \mathcal{C} \subset \mathcal{Z}_{k-1}$ and $f(L_f) \subset \mathcal{C} \subset \mathcal{Z}_{k-1}$.
We have $f(L_f) \subset \mathcal{Z}_{k-1} \cap f(\mathcal{Z}_{k-1})$. This implies that $h_f = h$ or $h_f = h \circ i$ where $i$ is the generator for the center of $Mod_N$ since the stabilizer of $f(L_f)$ is the center of $Mod_N$. Similarly, there exists a homeomorphism 
$h'_f$ of $N$ such that $h'_f([x]) = \lambda([x])$ for all $x \in f^{-1}(\mathcal{Z}_{k-1})$. We have $L_f \subset \mathcal{Z}_{k-1} \cap f^{-1}(\mathcal{Z}_{k-1})$. So, $h'_f = h$ or $h'_f = h \circ i$ since the pointwise stabilizer of $L_f$ is the center of $Mod_N$. We know that $i([x])=[x]$ for every vertex $[x] \in \mathcal{C}(N)$. So, $h([x]) = \lambda([x])$ for each $x \in \mathcal{Z}_k$. Hence, by induction 
$h([x]) = \lambda([x])$ for each $x \in \mathcal{Z}_k$ for all $k \geq 1$. This shows that $\mathcal{Z}_k$ is a superrigid set. Hence, by induction  $\mathcal{Z}_k$ is a finite superrigid set for all $k \geq 1$.\end{proof}\\
   
{\bf Acknowledgements}\\

The author thanks Peter Scott and Blazej Szepietowski for some comments.

{\bf Elmas Irmak} \vspace{0.2cm}
 
University of Michigan, Department of Mathematics, Ann Arbor, 48105, MI, USA, 

e-mail: eirmak@umich.edu\\

\end{document}